\documentclass[11pt,a4paper]{article}
\usepackage[T1]{fontenc}
\usepackage{amsfonts}
\usepackage{amsmath}
\usepackage{empheq}
\usepackage{amssymb}
\usepackage{wrapfig}
\usepackage{tikz}
\usepackage{pgfplots}
\pgfplotsset{compat=1.18}
\usepackage{enumitem}
\setlist{leftmargin=0mm}
\usepackage{setspace}
\usepackage{skull}
\usepackage{nicematrix}
\usepackage[english]{babel}
\usepackage{hyperref}
\usepackage{comment}
\usepackage{amsthm}
\usepackage{changepage}   
\usepackage{dsfont}
\usepackage[
backend=biber,
style=numeric,
sorting=nyt,
maxbibnames=99]{biblatex}
\usepackage{graphicx}
\usepackage{caption}
\usepackage[pagewise]{lineno} 
\newtheorem{theorem}{Theorem}

\newtheorem{prop}{Proposition}
\theoremstyle{definition}
\newtheorem{definition}{Definition}
\newtheorem{lemma}{Lemma}
\newtheorem{remark}{Remark}

\newtheorem{notation}{Notation}

\usepackage{lipsum}

\newcommand\phantomarrow[2]{
  \setbox0=\hbox{$\displaystyle #1\to$}%
  \hbox to \wd0{%
    $#2\mapstochar
     \cleaders\hbox{$\mkern-1mu\relbar\mkern-3mu$}\hfill
     \mkern-7mu\rightarrow$}%
  \,}
\makeatletter
\newcommand*{\rom}[1]{\expandafter\@slowromancap\romannumeral #1@}
\makeatother

\newcommand{\R}{\mathbb{R}}

\newcommand{\N}{\mathbb{N}}

\newcommand{\T}{\mathbb{T}}

\newcommand{\tend}[2]{\underset{#1 \to #2}{\longrightarrow}}
\newcommand{\tendf}[2]{\underset{#1 \to #2}{\rightharpoonup}}

\DeclareMathOperator{\cE}{{\cal E}}
\DeclareMathOperator{\cM}{{\cal M}}

\DeclareMathOperator{\cv}{\mathfrak{c}_v}
\usepackage{csquotes}
\usepackage{xcolor} 
\title{{Vanishing conductivity limit for the 1D compressible Navier-Stokes system}\thanks{This work benefited of the support of the ANR under France 2030 bearing the reference ANR-23-EXMA-004 (Complexflows project)}}
\date{\today}
\author{Pierre Gonin-{}-Joubert
\thanks{Universit\'e Claude Bernard Lyon 1, ICJ UMR5208, CNRS, \'Ecole Centrale de Lyon, INSA Lyon, Université Jean Monnet, 69622 Villeurbanne, France \texttt{goninjoubert@math.univ-lyon1.fr}}
\thanks{LAMA UMR5127 CNRS, Université Savoie Mont Blanc, Le Bourget du lac, France}}

\pgfplotstableread{
0.001 5.908755278982137
0.013011011011011011 3.3549702901874077
0.025022022022022024 2.713020983002951
0.037033033033033035 2.3329780129895727
0.04904404404404405 2.0640805706048155
0.06105505505505506 1.8570343350237615
0.07306606606606607 1.6894572985749667
0.08507707707707708 1.5492747213117166
0.0970880880880881 1.4292246760430691
0.10909909909909911 1.3245977428464322
0.12111011011011012 1.2321652563853136
0.13312112112112112 1.1496170011824018
0.14513213213213214 1.0752428275517902
0.15714314314314315 1.0077412923939677
0.16915415415415416 0.9460989792152545
0.18116516516516518 0.8895113142452631
0.1931761761761762 0.8373288527009252
0.2051871871871872 0.7890197954304553
0.21719819819819822 0.7441431847433694
0.22920920920920923 0.7023293243620965
0.24122022022022024 0.6632652061042221
0.25323123123123126 0.6266834814549411
0.26524224224224224 0.5923539911240867
0.2772532532532533 0.5600771720674329
0.28926426426426427 0.5296788637324867
0.3012752752752753 0.5010061716433526
0.3132862862862863 0.4739241400962958
0.32529729729729734 0.448313051173723
0.3373083083083083 0.4240662137197704
0.34931931931931937 0.4010881393455952
0.36133033033033035 0.379293026909441
0.3733413413413414 0.35860349491104615
0.3853523523523524 0.3389495146725151
0.39736336336336336 0.3202675073096194
0.4093743743743744 0.30249957521396587
0.4213853853853854 0.2855928446966409
0.43339639639639643 0.2694989010399692
0.4454074074074074 0.2541733007943596
0.45741841841841846 0.23957514898285137
0.46942942942942945 0.2256667311154006
0.4814404404404405 0.21241319170152906
0.4934514514514515 0.19978225238418756
0.5054624624624625 0.18774396397584814
0.5174734734734735 0.1762704876183303
0.5294844844844845 0.16533590105564067
0.5414954954954956 0.154916026639062
0.5535065065065066 0.14498827820307003
0.5655175175175176 0.13553152438078675
0.5775285285285285 0.1265259662854592
0.5895395395395395 0.11795302778329952
0.6015505505505506 0.10979525683364011
0.6135615615615616 0.1020362365833184
0.6255725725725726 0.09466050508044788
0.6375835835835836 0.08765348262384087
0.6495945945945947 0.08100140589289495
0.6616056056056057 0.07469126811245369
0.6736166166166166 0.06871076460105902
0.6856276276276276 0.06304824313163915
0.6976386386386387 0.05769265860309791
0.7096496496496497 0.05263353158120998
0.7216606606606607 0.0478609103191186
0.7336716716716717 0.0433653359127672
0.7456826826826828 0.039137810285782115
0.7576936936936938 0.03516976673249805
0.7697047047047048 0.03145304277769084
0.7817157157157157 0.027979855137758053
0.7937267267267267 0.02474277659106322
0.8057377377377378 0.021734714585381815
0.8177487487487488 0.018948891428208492
0.8297597597597598 0.016378825921431195
0.8417707707707708 0.014018316315812657
0.8537817817817819 0.011861424473073506
0.8657927927927929 0.009902461134345758
0.8778038038038039 0.008135972203529435
0.8898148148148148 0.006556725962787696
0.9018258258258259 0.005159701145184228
0.9138368368368369 0.003940075796413686
0.9258478478478479 0.0028932168637962052
0.9378588588588589 0.0020146704562865603
0.94986986986987 0.0013001527242594915
0.961880880880881 0.0007455413123400936
0.973891891891892 0.0003468673426079219
0.985902902902903 0.000100307889165277
0.9979139139139139 2.178908366934438e-06
1.009924924924925 4.892859199891013e-05
1.021935935935936 0.00023713111340144602
1.0339469469469469 0.0005634807389857041
1.0459579579579579 0.0010247862798213256
1.0579689689689689 0.0016179658599959207
1.0699799799799798 0.0023400419802891287
1.081990990990991 0.003188136857380397
1.094002002002002 0.004159468020338508
1.106013013013013 0.005251344147536197
1.118024024024024 0.006461161128407417
1.130035035035035 0.007786398335630815
1.142046046046046 0.009224615094390015
1.154057057057057 0.010773447336338476
1.166068068068068 0.012430604426794528
1.178079079079079 0.01419386615451379
1.1900900900900901 0.016061079874144596
1.2021011011011011 0.018030157792165435
1.2141121121121121 0.02009907438774594
1.226123123123123 0.02226586396056096
1.238134134134134 0.024528618298132215
1.250145145145145 0.026885484455773656
1.262156156156156 0.029334662642679837
1.274167167167167 0.031874404208125456
1.2861781781781783 0.03450300972213971
1.2981891891891892 0.03721882714538649
1.3102002002002002 0.04002025008332111
1.3222112112112112 0.042905716120008486
1.3342222222222222 0.04587370522728079
1.3462332332332332 0.04892273824518262
1.3582442442442442 0.05205137542990523
1.3702552552552552 0.05525821506564366
1.3822662662662661 0.05854189213703004
1.3942772772772773 0.061901077058998
1.4062882882882883 0.06533447446112206
1.4182992992992993 0.06884082202365371
1.4303103103103103 0.07241888936263569
1.4423213213213213 0.07606747696163452
1.4543323323323323 0.07978541514776993
1.4663433433433433 0.08357156310985564
1.4783543543543542 0.08742480795659024
1.4903653653653655 0.09134406381285393
1.5023763763763764 0.09532827095227431
1.5143873873873874 0.09937639496433126
1.5263983983983984 0.1034874259543605
1.5384094094094094 0.10766037777491078
1.5504204204204204 0.11189428728699075
1.5624314314314314 0.11618821364982157
1.5744424424424424 0.12054123763778574
1.5864534534534533 0.12495246098333151
1.5984644644644646 0.12942100574465892
1.6104754754754755 0.13394601369707299
1.6224864864864865 0.13852664574695145
1.6344974974974975 0.1431620813673229
1.6465085085085085 0.14785151805410868
1.6585195195195195 0.15259417080212612
1.6705305305305305 0.15738927159999783
1.6825415415415415 0.162236068943154
1.6945525525525527 0.16713382736415638
1.7065635635635636 0.17208182697960817
1.7185745745745746 0.17707936305295413
1.7305855855855856 0.18212574557250505
1.7425965965965966 0.18722029884405544
1.7546076076076076 0.19236236109749294
1.7666186186186186 0.19755128410682654
1.7786296296296296 0.20278643282308917
1.7906406406406405 0.2080671850195922
1.8026516516516518 0.21339293094904
1.8146626626626627 0.21876307301202746
1.8266736736736737 0.22417702543647544
1.8386846846846847 0.22963421396756856
1.8506956956956957 0.23513407556778854
1.8627067067067067 0.24067605812665027
1.8747177177177177 0.24625962017976577
1.8867287287287287 0.2518842306368797
1.8987397397397399 0.2575493685185338
1.9107507507507508 0.2632545227010349
1.9227617617617618 0.2689991916694128
1.9347727727727728 0.274782883278071
1.9467837837837838 0.28060511451884396
1.9587947947947948 0.28646541129618697
1.9708058058058058 0.29236330820923784
1.9828168168168168 0.2982983483405002
1.9948278278278277 0.30427008305090697
2.006838838838839 0.3102780717810356
2.01884984984985 0.3163218818582546
2.030860860860861 0.32240108830958913
2.042871871871872 0.3285152736801056
2.054882882882883 0.33466402785661786
2.066893893893894 0.3408469478965327
2.078904904904905 0.3470636378616522
2.090915915915916 0.35331370865676515
2.102926926926927 0.3595967778728596
2.114937937937938 0.3659124696348016
2.126948948948949 0.37226041445332647
2.13895995995996 0.37864024908119687
2.150970970970971 0.38505161637338947
2.162981981981982 0.3914941651511734
2.174992992992993 0.3979675500699521
2.187004004004004 0.4044714314907456
2.199015015015015 0.41100547535519083
2.211026026026026 0.4175693530639466
2.223037037037037 0.4241627413583935
2.235048048048048 0.43078532220551957
2.247059059059059 0.4374367826858925
2.25907007007007 0.44411681488461574
2.271081081081081 0.4508251157851787
2.283092092092092 0.4575613871661034
2.295103103103103 0.4643253355003053
2.307114114114114 0.47111667185707895
2.319125125125125 0.4779351118066293
2.331136136136136 0.4847803753270684
2.343147147147147 0.4916521867138025
2.355158158158158 0.4985502744912347
2.3671691691691694 0.5054743713267148
2.3791801801801804 0.5124242139466648
2.3911911911911914 0.5193995430548188
2.4032022022022024 0.5264001032525093
2.4152132132132134 0.5334256429609403
2.4272242242242243 0.5404759143453902
2.4392352352352353 0.547550673241284
2.4512462462462463 0.5546496790820814
2.4632572572572573 0.5617726948289279
2.4752682682682683 0.5689194869020161
2.4872792792792793 0.5760898251136091
2.4992902902902903 0.5832834826026766
2.5113013013013012 0.5905002357710978
2.5233123123123122 0.5977398642213873
2.535323323323323 0.6050021506958978
2.547334334334334 0.6122868810174619
2.559345345345345 0.6195938440314275
2.5713563563563566 0.6269228315490527
2.5833673673673676 0.6342736382922167
2.5953783783783786 0.6416460618394167
2.6073893893893896 0.6490399025730086
2.6194004004004006 0.6564549636276604
2.6314114114114115 0.6638910508399866
2.6434224224224225 0.6713479726993276
2.6554334334334335 0.6788255402996454
2.6674444444444445 0.686323567292505
2.6794554554554555 0.6938418698411116
2.6914664664664665 0.7013802665753766
2.7034774774774775 0.7089385785479819
2.7154884884884884 0.7165166291914211
2.7274994994994994 0.7241142442759858
2.7395105105105104 0.7317312518686769
2.7515215215215214 0.739367482293013
2.7635325325325324 0.7470227680897166
2.775543543543544 0.75469694397825
2.787554554554555 0.7623898468191825
2.799565565565566 0.7701013155773688
2.811576576576577 0.7778311912859128
2.8235875875875878 0.7855793170109004
2.8355985985985988 0.7933455378168812
2.8476096096096097 0.8011297007330795
2.8596206206206207 0.8089316547203165
2.8716316316316317 0.8167512506386254
2.8836426426426427 0.8245883412155428
2.8956536536536537 0.8324427810150585
2.9076646646646647 0.8403144264072087
2.9196756756756757 0.8482031355382953
2.9316866866866866 0.8561087683017172
2.9436976976976976 0.864031186309397
2.9557087087087086 0.8719702528637896
2.9677197197197196 0.87992583293046
2.979730730730731 0.8878977931112126
2.991741741741742 0.895886001617761
3.003752752752753 0.9038903282459287
3.015763763763764 0.9119106443503606
3.027774774774775 0.9199468228197387
3.039785785785786 0.9279987380524888
3.051796796796797 0.9360662659329653
3.063807807807808 0.9441492838081051
3.075818818818819 0.9522476704645362
3.08782982982983 0.9603613061061342
3.099840840840841 0.9684900723320136
3.111851851851852 0.9766338521149449
3.123862862862863 0.9847925297801876
3.135873873873874 0.9929659909847284
3.147884884884885 1.001154122696918
3.159895895895896 1.0093568131764934
3.171906906906907 1.0175739519549818
3.1839179179179182 1.0258054298164725
3.195928928928929 1.0340511387787508
3.20793993993994 1.0423109720747894
3.219950950950951 1.05058482413458
3.231961961961962 1.0588725905673082
3.243972972972973 1.0671741681438554
3.255983983983984 1.075489454779626
3.267994994994995 1.0838183495176907
3.280006006006006 1.092160752512238
3.292017017017017 1.100516565012329
3.304028028028028 1.108885689345947
3.316039039039039 1.1172680289043382
3.32805005005005 1.1256634881266345
3.340061061061061 1.1340719724847526
3.352072072072072 1.1424933884685646
3.364083083083083 1.1509276435713354
3.376094094094094 1.1593746462754166
3.3881051051051054 1.1678343060381977
3.4001161161161164 1.1763065332783023
3.4121271271271274 1.1847912393620317
3.4241381381381384 1.193288336590043
3.4361491491491494 1.2017977381842633
3.4481601601601604 1.2103193582750296
3.4601711711711713 1.218853111888455
3.4721821821821823 1.2273989149340112
3.4841931931931933 1.2359566841923284
3.4962042042042043 1.244526337303201
3.5082152152152153 1.253107792753803
3.5202262262262263 1.2617009698671036
3.5322372372372373 1.2703057887904774
3.5442482482482482 1.2789221704845128
3.5562592592592592 1.287550036712005
3.56827027027027 1.296189310027138
3.580281281281281 1.304839913764846
3.5922922922922926 1.313501772030355
3.6043033033033036 1.3221748096888966
3.6163143143143146 1.3308589523555951
3.6283253253253256 1.3395541263855222
3.6403363363363366 1.3482602588639128
3.6523473473473476 1.3569772775965465
3.6643583583583585 1.3657051111002847
3.6763693693693695 1.3744436885937608
3.6883803803803805 1.3831929399882268
3.7003913913913915 1.3919527958785438
3.7124024024024025 1.4007231875343218
3.7244134134134135 1.4095040468912017
3.7364244244244245 1.4182953065422768
3.7484354354354354 1.4270968997296523
3.7604464464464464 1.4359087603361418
3.7724574574574574 1.4447308228770923
3.7844684684684684 1.4535630224923435
3.79647947947948 1.4624052949383108
3.808490490490491 1.471257576580195
3.820501501501502 1.4801198043843156
3.832512512512513 1.4889919159105627
3.844523523523524 1.4978738493049673
3.8565345345345348 1.5067655432923888
3.8685455455455457 1.5156669371693139
3.8805565565565567 1.52457797079677
3.8925675675675677 1.5334985845933453
3.9045785785785787 1.5424287195283184
3.9165895895895897 1.5513683171148922
3.9286006006006007 1.5603173194035325
3.9406116116116117 1.5692756689754066
3.9526226226226227 1.5782433089359242
3.9646336336336336 1.5872201829083739
3.9766446446446446 1.596206235027657
3.9886556556556556 1.6052014099341159
4.0006666666666675 1.6142056527674558
4.0126776776776785 1.6232189091607545
4.0246886886886895 1.6322411252345674
4.0366996996997 1.6412722475911141
4.048710710710711 1.650312223308556
4.060721721721722 1.659360999935356
4.072732732732733 1.6684185254847224
4.084743743743744 1.6774847484291346
4.096754754754755 1.68655961769495
4.108765765765766 1.6956430826570876
4.120776776776777 1.7047350931337915
4.132787787787788 1.7138355993814696
4.144798798798799 1.7229445520896067
4.15680980980981 1.7320619023757529
4.168820820820821 1.7411876017805816
4.180831831831832 1.7503216022630221
4.192842842842843 1.759463856195458
4.204853853853854 1.7686143163589982
4.216864864864865 1.7777729359388124
4.228875875875876 1.7869396685195338
4.240886886886887 1.7961144680807288
4.252897897897898 1.8052972889924275
4.264908908908909 1.81448808601072
4.27691991991992 1.8236868142734128
4.288930930930931 1.8328934292957473
4.300941941941942 1.842107886966178
4.312952952952954 1.851330143542209
4.324963963963965 1.8605601556462872
4.336974974974976 1.8697978802617568
4.348985985985987 1.879043274728865
4.360996996996998 1.8882962967408252
4.373008008008009 1.8975569043399336
4.38501901901902 1.9068250559137392
4.397030030030031 1.916100710191266
4.409041041041042 1.925383826239288
4.421052052052053 1.9346743634586525
4.433063063063064 1.9439722815806557
4.445074074074075 1.9532775406634657
4.457085085085086 1.962590101088594
4.469096096096097 1.9719099235574153
4.481107107107108 1.9812369690877338
4.493118118118119 1.990571199010394
4.50512912912913 1.999912574965939
4.517140140140141 2.0092610589013113
4.529151151151152 2.0186166130665995
4.541162162162163 2.027979200011825
4.553173173173174 2.037348782583775
4.565184184184185 2.0467253239228733
4.577195195195196 2.0561087874600963
4.589206206206207 2.0654991369139246
4.601217217217218 2.074896336287339
4.613228228228229 2.0843003498648542
4.62523923923924 2.093711142209588
4.637250250250251 2.103128678160374
4.649261261261262 2.112552922828906
4.661272272272273 2.1219838415969248
4.673283283283284 2.131421400113436
4.685294294294295 2.1408655642919676
4.6973053053053055 2.15031630030786
4.7093163163163165 2.1597735745955915
4.721327327327328 2.169237353846139
4.733338338338339 2.1787076050043694
4.74534934934935 2.1881842952664674
4.757360360360361 2.1976673920773946
4.769371371371372 2.207156863128379
4.781382382382383 2.216652676354437
4.793393393393394 2.2261547999319298
4.805404404404405 2.2356632022761422
4.817415415415416 2.2451778520389025
4.829426426426427 2.254698718106223
4.841437437437438 2.264225769595974
4.853448448448449 2.2737589758555856
4.86545945945946 2.283298306459779
4.877470470470471 2.2928437312083236
4.889481481481482 2.302395220123824
4.901492492492493 2.311952743449531
4.913503503503504 2.3215162716471864
4.925514514514515 2.3310857753948824
4.937525525525526 2.3406612255849604
4.949536536536537 2.350242593321923
4.961547547547548 2.3598298499203816
4.973558558558559 2.369422966903023
4.98556956956957 2.3790219159986004
4.997580580580581 2.388626669139951
5.009591591591592 2.398237198462037
5.021602602602603 2.407853476300008
5.033613613613614 2.4174754751872904
5.045624624624625 2.4271031678536943
5.057635635635636 2.436736527223548
5.069646646646647 2.44637552641385
5.081657657657658 2.456020138732449
5.093668668668669 2.4656703376762374
5.10567967967968 2.475326096929374
5.117690690690691 2.484987390361521
5.129701701701703 2.4946541920261067
5.141712712712714 2.504326476158604
5.153723723723725 2.514004217174834
5.165734734734736 2.523687389669286
5.177745745745747 2.5333759684134565
5.189756756756758 2.5430699283542095
5.201767767767769 2.552769244612155
5.21377877877878 2.562473892480046
5.225789789789791 2.5721838474211944
5.237800800800802 2.581899085067903
5.249811811811813 2.5916195812199194
5.261822822822824 2.601345311842905
5.273833833833835 2.611076253066921
5.285844844844846 2.620812381184932
5.297855855855857 2.6305536726513283
5.309866866866868 2.640300104080464
5.321877877877879 2.6500516522452093
5.33388888888889 2.6598082940755217
5.3458998998999006 2.6695700066570347
5.3579109109109115 2.6793367672296577
5.3699219219219225 2.6891085531861956
5.3819329329329335 2.698885342070981
5.3939439439439445 2.708667111578526
5.4059549549549555 2.718453839552182
5.4179659659659665 2.728245503982821
5.4299769769769775 2.73804208300753
5.441987987987988 2.7478435549083136
5.453998998998999 2.757649898110823
5.46601001001001 2.767461091183087
5.478021021021021 2.7772771128342626
5.490032032032032 2.7870979419134008
5.502043043043043 2.796923557408222
5.514054054054054 2.806753938443906
5.526065065065065 2.816589064281897
5.538076076076077 2.826428914318721
5.550087087087088 2.8362734680848103
5.562098098098099 2.846122705243352
5.57410910910911 2.855976605589139
5.586120120120121 2.8658351490474367
5.598131131131132 2.8756983156728637
5.610142142142143 2.885566085648282
5.622153153153154 2.8954384392836996
5.634164164164165 2.905315357015186
5.646175175175176 2.9151968194037963
5.658186186186187 2.9250828071345127
5.670197197197198 2.9349733010151895
5.682208208208209 2.9448682819755145
5.69421921921922 2.9547677310659815
5.706230230230231 2.964671629456869
5.718241241241242 2.9745799584372348
5.730252252252253 2.98449269941392
5.742263263263264 2.994409833910559
5.754274274274275 3.004331343566607
5.766285285285286 3.014257210136372
5.778296296296297 3.0241874154880604
5.790307307307308 3.0341219416028276
5.802318318318319 3.044060770573846
5.81432932932933 3.0540038846053763
5.826340340340341 3.0639512660118498
5.838351351351352 3.073902897216963
5.850362362362363 3.0838587607527774
5.862373373373374 3.0938188392588315
5.874384384384385 3.103783115481261
5.886395395395396 3.1137515722719273
5.898406406406407 3.123724192587556
5.910417417417418 3.1337009594888823
5.922428428428429 3.143681856139809
5.93443943943944 3.1536668658065685
5.946450450450452 3.163655971856895
5.958461461461463 3.173649157759205
5.970472472472474 3.1836464070817887
5.982483483483485 3.193647703492005
5.994494494494496 3.203653030755487
6.0065055055055065 3.2136623727353566
6.0185165165165175 3.223675713391442
6.0305275275275285 3.2336930367795107
6.0425385385385395 3.2437143270505038
6.0545495495495505 3.2537395684497783
6.0665605605605615 3.2637687453163635
6.0785715715715725 3.273801842082215
6.0905825825825834 3.283838843271483
6.102593593593594 3.2938797334997867
6.114604604604605 3.303924497473493
6.126615615615616 3.3139731199890052
6.138626626626627 3.324025585932058
6.150637637637638 3.334081880277018
6.162648648648649 3.344141988086191
6.17465965965966 3.3542058945091426
6.186670670670671 3.364273584782013
6.198681681681682 3.374345044226848
6.210692692692693 3.3844202582509357
6.222703703703704 3.394499212346145
6.234714714714715 3.4045818920882747
6.246725725725726 3.4146682831364084
6.258736736736737 3.4247583712322696
6.270747747747748 3.4348521421995954
6.282758758758759 3.4449495819435025
6.29476976976977 3.455050676449869
6.306780780780781 3.465155411784716
6.318791791791792 3.4752637740936008
6.330802802802803 3.485375749601009
6.342813813813814 3.4954913246097608
6.354824824824825 3.5056104855004127
6.366835835835837 3.5157332187306753
6.378846846846848 3.5258595108348283
6.390857857857859 3.535989348423145
6.40286886886887 3.5461227181813237
6.414879879879881 3.5562596068699195
6.426890890890892 3.566400001323786
6.438901901901903 3.57654388845152
6.450912912912914 3.5866912552349097
6.462923923923925 3.596842088728394
6.474934934934936 3.60699637605852
6.486945945945947 3.617154104423409
6.498956956956958 3.6273152610922272
6.510967967967969 3.6374798334046625
6.52297897897898 3.6476478087704005
6.534989989989991 3.6578191746686146
6.547001001001002 3.6679939186474506
6.559012012012013 3.6781720283235257
6.571023023023024 3.688353491381423
6.583034034034035 3.698538295573199
6.595045045045046 3.7087264287178883
6.607056056056057 3.718917878701019
6.619067067067068 3.7291126334741267
6.631078078078079 3.7393106810542784
6.64308908908909 3.7495120095235968
6.655100100100101 3.75971660702879
6.667111111111112 3.769924461780687
6.679122122122123 3.7801355620537747
6.691133133133134 3.790349896185741
6.703144144144145 3.800567452577022
6.715155155155156 3.81078821969035
6.727166166166167 3.8210121860503126
6.739177177177178 3.8312393402429077
6.7511881881881886 3.8414696709151075
6.7631991991991995 3.851703166774425
6.775210210210211 3.8619398165884844
6.787221221221222 3.8721796091845926
6.799232232232233 3.882422533449322
6.811243243243244 3.892668578328089
6.823254254254255 3.90291773282474
6.835265265265266 3.9131699860011384
6.847276276276277 3.923425326976762
6.859287287287288 3.9336837449282926
6.871298298298299 3.943945229089222
6.88330930930931 3.9542097687494504
6.895320320320321 3.9644773532548956
6.907331331331332 3.974747972007102
6.919342342342343 3.9850216144628545
6.931353353353354 3.9952982701337953
6.943364364364365 4.005577928586043
6.955375375375376 4.015860579439818
6.967386386386387 4.026146212369066
6.979397397397398 4.036434817101092
6.991408408408409 4.046726383416191
7.00341941941942 4.057020901147279
7.015430430430431 4.067318360179546
7.027441441441442 4.077618750450084
7.039452452452453 4.0879220619475385
7.051463463463464 4.098228284711759
7.063474474474475 4.108537408833448
7.075485485485486 4.118849424453813
7.087496496496497 4.129164321764231
7.099507507507508 4.139482091005899
7.111518518518519 4.149802722469505
7.12352952952953 4.160126206494891
7.135540540540541 4.170452533470723
7.147551551551552 4.180781693834159
7.159562562562563 4.191113678070527
7.171573573573574 4.201448476713003
7.183584584584586 4.211786080342287
7.195595595595597 4.222126479586283
7.207606606606608 4.232469665119795
7.219617617617619 4.242815627664203
7.23162862862863 4.25316435798716
7.243639639639641 4.2635158469022825
7.255650650650652 4.273870085268846
7.267661661661663 4.284227063991484
7.279672672672674 4.294586774019887
7.291683683683685 4.304949206348505
7.303694694694696 4.315314352016257
7.315705705705707 4.325682202106231
7.327716716716718 4.336052747745401
7.339727727727729 4.346425980104339
7.35173873873874 4.356801890396925
7.363749749749751 4.3671804698800685
7.375760760760762 4.377561709853427
7.387771771771773 4.387945601659127
7.399782782782784 4.398332136681489
7.4117937937937945 4.408721306346754
7.4238048048048055 4.419113102122807
7.4358158158158165 4.429507515518921
7.4478268268268275 4.43990453808547
7.4598378378378385 4.4503041614136825
7.4718488488488495 4.460706377135368
7.4838598598598605 4.471111176922662
7.4958708708708714 4.48151855248776
7.507881881881882 4.491928495582668
7.519892892892893 4.502340997998948
7.531903903903904 4.512756051567457
7.543914914914915 4.52317364815811
7.555925925925926 4.5335937796796175
7.567936936936937 4.5440164380792485
7.579947947947948 4.554441615342583
7.59195895895896 4.564869303493269
7.603969969969971 4.575299494592779
7.615980980980982 4.58573218074018
7.627991991991993 4.596167354071887
7.640003003003004 4.606605006761432
7.652014014014015 4.617045131019231
7.664025025025026 4.627487719092352
7.676036036036037 4.6379327632642875
7.688047047047048 4.648380255854721
7.700058058058059 4.658830189219307
7.71206906906907 4.669282555749444
7.724080080080081 4.679737347872052
7.736091091091092 4.690194558049354
7.748102102102103 4.700654178778654
7.760113113113114 4.71111620259212
7.772124124124125 4.721580622056569
7.784135135135136 4.732047429773255
7.796146146146147 4.742516618377652
7.808157157157158 4.7529881805392495
7.820168168168169 4.7634621089613365
7.83217917917918 4.773938396380798
7.844190190190191 4.784417035567909
7.856201201201202 4.794898019326132
7.868212212212213 4.8053813404919055
7.880223223223224 4.815866991934455
7.892234234234235 4.826354966555585
7.904245245245246 4.8368452572894824
7.916256256256257 4.847337857102523
7.928267267267268 4.85783275899307
7.940278278278279 4.868329955991288
7.95228928928929 4.878829441158945
7.964300300300301 4.889331207589223
7.976311311311312 4.8998352484065295
7.988322322322323 4.910341556766311
8.000333333333334 4.920850125854862
8.012344344344344 4.931360948889145
8.024355355355356 4.941874019116602
8.036366366366366 4.952389329814974
8.048377377377378 4.962906874292125
8.060388388388388 4.973426645885853
8.0723993993994 4.983948637963719
8.08441041041041 4.994472843922862
8.096421421421422 5.004999257189841
8.108432432432432 5.015527871220433
8.120443443443444 5.02605867949949
8.132454454454454 5.036591675540739
8.144465465465466 5.047126852886636
8.156476476476476 5.057664205108178
8.168487487487488 5.06820372580475
8.180498498498498 5.078745408603941
8.19250950950951 5.0892892471614
8.204520520520521 5.09983523516065
8.216531531531531 5.110383366312938
8.228542542542543 5.120933634357073
8.240553553553553 5.1314860330592555
8.252564564564565 5.142040556212931
8.264575575575575 5.152597197638622
8.276586586586587 5.163155951183774
8.288597597597597 5.173716810722596
8.30060860860861 5.184279770155915
8.31261961961962 5.194844823411007
8.324630630630631 5.205411964441461
8.336641641641641 5.215981187227011
8.348652652652653 5.226552485773402
8.360663663663663 5.237125854112222
8.372674674674675 5.247701286300774
8.384685685685685 5.258278776421908
8.396696696696697 5.268858318583893
8.408707707707707 5.279439906920256
8.42071871871872 5.290023535589652
8.43272972972973 5.300609198775708
8.444740740740741 5.31119689068689
8.456751751751751 5.321786605556358
8.468762762762763 5.332378337641827
8.480773773773773 5.342972081225422
8.492784784784785 5.353567830613555
8.504795795795795 5.364165580136765
8.516806806806807 5.374765324149607
8.528817817817817 5.385367057030493
8.540828828828829 5.395970773181576
8.552839839839839 5.406576467028605
8.564850850850851 5.417184133020802
8.576861861861861 5.427793765630719
8.588872872872873 5.438405359354119
8.600883883883883 5.449018908709835
8.612894894894895 5.459634408239653
8.624905905905907 5.470251852508177
8.636916916916917 5.480871236102697
8.648927927927929 5.491492553633079
8.660938938938939 5.50211579973162
8.67294994994995 5.512740969052944
8.68496096096096 5.5233680562738545
8.696971971971973 5.533997056093238
8.708982982982983 5.544627963231921
8.720993993993995 5.555260772432563
8.733005005005005 5.565895478459522
8.745016016016017 5.576532076098756
8.757027027027027 5.587170560157681
8.769038038038039 5.597810925465072
8.781049049049049 5.608453166870933
8.79306006006006 5.619097279246393
8.80507107107107 5.629743257483578
8.817082082082083 5.640391096495508
8.829093093093093 5.651040791215973
8.841104104104105 5.66169233659943
8.853115115115115 5.672345727620879
8.865126126126126 5.68300095927577
8.877137137137137 5.693658026579865
8.889148148148148 5.7043169245691585
8.901159159159159 5.714977648299743
8.91317017017017 5.7256401928477185
8.92518118118118 5.736304553309071
8.937192192192192 5.746970724799578
8.949203203203203 5.75763870245469
8.961214214214214 5.768308481429436
8.973225225225224 5.778980056898309
8.985236236236236 5.789653424055172
8.997247247247246 5.8003285781131435
9.009258258258258 5.8110055143045045
9.02126926926927 5.8216842278805885
9.03328028028028 5.832364714111685
9.045291291291292 5.843046968286939
9.057302302302302 5.853730985714247
9.069313313313314 5.864416761720162
9.081324324324324 5.875104291649789
9.093335335335336 5.8857935708666975
9.105346346346346 5.89648459475281
9.117357357357358 5.907177358708315
9.129368368368368 5.917871858151569
9.14137937937938 5.9285680885189995
9.15339039039039 5.939266045265009
9.165401401401402 5.949965723861886
9.177412412412412 5.960667119799702
9.189423423423424 5.971370228586233
9.201434434434434 5.982075045746846
9.213445445445446 5.9927815668244335
9.225456456456456 6.003489787379295
9.237467467467468 6.014199702989069
9.249478478478478 6.024911309248625
9.26148948948949 6.035624601769992
9.2735005005005 6.04633957618225
9.285511511511512 6.0570562281314615
9.297522522522522 6.067774553280563
9.309533533533534 6.078494547309298
9.321544544544544 6.089216205914116
9.333555555555556 6.099939524808095
9.345566566566566 6.110664499720846
9.357577577577578 6.121391126398443
9.369588588588588 6.132119400603322
9.3815995995996 6.142849318114212
9.39361061061061 6.153580874726041
9.405621621621622 6.164314066249856
9.417632632632632 6.175048888512745
9.429643643643644 6.185785337357749
9.441654654654656 6.196523408643786
9.453665665665666 6.207263098245566
9.465676676676678 6.218004402053514
9.477687687687688 6.228747315973687
9.4896986986987 6.239491835927701
9.50170970970971 6.250237957852645
9.513720720720722 6.260985677701006
9.525731731731732 6.27173499144059
9.537742742742743 6.282485895054451
9.549753753753754 6.293238384540803
9.561764764764765 6.303992455912956
9.573775775775776 6.314748105199227
9.585786786786787 6.3255053284428815
9.597797797797798 6.336264121702038
9.60980880880881 6.347024481049615
9.62181981981982 6.357786402573239
9.633830830830831 6.368549882375185
9.645841841841841 6.379314916572293
9.657852852852853 6.390081501295903
9.669863863863863 6.400849632691775
9.681874874874875 6.411619306920031
9.693885885885885 6.422390520155065
9.705896896896897 6.43316326858549
9.717907907907907 6.4439375484140555
9.72991891891892 6.454713355857584
9.74192992992993 6.4654906871468985
9.753940940940941 6.476269538526759
9.765951951951951 6.4870499062557805
9.777962962962963 6.497831786606385
9.789973973973973 6.508615175864713
9.801984984984985 6.519400070330573
9.813995995995995 6.530186466317362
9.826007007007007 6.540974360152012
9.838018018018017 6.5517637481749045
9.850029029029029 6.5625546267398285
9.862040040040041 6.573346992213897
9.874051051051051 6.584140840977487
9.886062062062063 6.594936169424184
9.898073073073073 6.605732973960702
9.910084084084085 6.616531251006833
9.922095095095095 6.627330996995376
9.934106106106107 6.638132208372081
9.946117117117117 6.648934881595576
9.958128128128129 6.659739013137317
9.970139139139139 6.670544599481515
9.98215015015015 6.681351637125086
9.994161161161161 6.692160122577576
10.006172172172173 6.702970052361115
10.018183183183183 6.713781423010344
10.030194194194195 6.724594231072365
10.042205205205205 6.735408473106674
10.054216216216217 6.746224145685108
10.066227227227227 6.757041245391777
10.078238238238239 6.767859768823018
10.090249249249249 6.778679712587325
10.10226026026026 6.789501073305298
10.11427127127127 6.800323847609583
10.126282282282283 6.8111480321448195
10.138293293293293 6.821973623567571
10.150304304304305 6.832800618546285
10.162315315315315 6.843629013761224
10.174326326326327 6.854458805904419
10.186337337337337 6.865289991679603
10.198348348348349 6.87612256780217
10.210359359359359 6.886956530999106
10.22237037037037 6.897791878008945
10.23438138138138 6.908628605581708
10.246392392392393 6.919466710478853
10.258403403403404 6.930306189473221
10.270414414414414 6.941147039348979
10.282425425425426 6.951989256901575
10.294436436436436 6.962832838937672
10.306447447447448 6.973677782275115
10.318458458458458 6.9845240837428575
10.33046946946947 6.9953717401809286
10.34248048048048 7.006220748440364
10.354491491491492 7.017071105383174
10.366502502502502 7.027922807882273
10.378513513513514 7.038775852821448
10.390524524524524 7.0496302370952915
10.402535535535536 7.060485957609162
10.414546546546546 7.07134301127913
10.426557557557558 7.082201395031933
10.438568568568568 7.093061105804917
10.45057957957958 7.103922140546004
10.46259059059059 7.11478449621362
10.474601601601602 7.125648169776674
10.486612612612612 7.136513158214484
10.498623623623624 7.147379458516751
10.510634634634634 7.1582470676834955
10.522645645645646 7.169115982725021
10.534656656656656 7.179986200661858
10.546667667667668 7.19085771852473
10.558678678678678 7.201730533354489
10.57068968968969 7.212604642202091
10.5827007007007 7.223480042128529
10.594711711711712 7.234356730204805
10.606722722722722 7.245234703511871
10.618733733733734 7.256113959140594
10.630744744744744 7.266994494191705
10.642755755755756 7.277876305775759
10.654766766766766 7.288759391013084
10.666777777777778 7.299643747033747
10.67878878878879 7.3105293709775
10.6907997997998 7.32141625999374
10.702810810810812 7.33230441124147
10.714821821821822 7.343193821889249
10.726832832832834 7.3540844891151576
10.738843843843844 7.364976410106744
10.750854854854856 7.375869582060995
10.762865865865866 7.386764002184277
10.774876876876878 7.3976596676923165
10.786887887887888 7.408556575810135
10.7988988988989 7.419454723772025
10.81090990990991 7.430354108821499
10.822920920920922 7.4412547282112556
10.834931931931932 7.452156579203128
10.846942942942944 7.463059659068058
10.858953953953954 7.473963965086044
10.870964964964966 7.484869494546109
10.882975975975976 7.49577624474625
10.894986986986988 7.506684212993413
10.906997997997998 7.51759339660344
10.91900900900901 7.52850379290104
10.93102002002002 7.53941539921974
10.943031031031031 7.550328212901859
10.955042042042042 7.561242231298458
10.967053053053053 7.572157451769307
10.979064064064064 7.583073871682846
10.991075075075075 7.593991488416147
11.003086086086086 7.604910299354874
11.015097097097097 7.615830301893254
11.027108108108107 7.626751493434027
11.03911911911912 7.637673871388419
11.05113013013013 7.648597433176098
11.063141141141141 7.659522176225146
11.075152152152153 7.670448097972015
11.087163163163163 7.681375195861488
11.099174174174175 7.692303467346656
11.111185185185185 7.703232909888868
11.123196196196197 7.714163520957708
11.135207207207207 7.725095298030942
11.14721821821822 7.736028238594505
11.15922922922923 7.746962340142447
11.171240240240241 7.757897600176907
11.183251251251251 7.768834016208077
11.195262262262263 7.779771585754167
11.207273273273273 7.790710306341371
11.219284284284285 7.801650175503832
11.231295295295295 7.812591190783608
11.243306306306307 7.823533349730639
11.255317317317317 7.834476649902712
11.267328328328329 7.8454210888654305
11.279339339339339 7.8563666641921746
11.291350350350351 7.867313373464079
11.303361361361361 7.878261214269989
11.315372372372373 7.889210184206432
11.327383383383383 7.900160280877584
11.339394394394395 7.911111501895245
11.351405405405405 7.9220638448787914
11.363416416416417 7.9330173074551595
11.375427427427427 7.943971887258798
11.387438438438439 7.954927581931656
11.399449449449449 7.965884389123131
11.41146046046046 7.976842306490051
11.423471471471471 7.987801331696636
11.435482482482483 7.998761462414476
11.447493493493493 8.009722696322484
11.459504504504505 8.020685031106886
11.471515515515515 8.031648464461172
11.483526526526527 8.042612994086078
11.495537537537539 8.053578617689547
11.507548548548549 8.064545332986704
11.51955955955956 8.075513137699824
11.53157057057057 8.086482029558306
11.543581581581583 8.097452006298639
11.555592592592593 8.10842306566437
11.567603603603605 8.119395205406088
11.579614614614615 8.130368423281372
11.591625625625626 8.141342717054785
11.603636636636637 8.152318084497834
11.615647647647648 8.163294523388942
11.627658658658659 8.174272031513414
11.63966966966967 8.185250606663425
11.65168068068068 8.196230246637974
11.663691691691692 8.207210949242867
11.675702702702702 8.21819271229068
11.687713713713714 8.229175533600745
11.699724724724724 8.240159410999102
11.711735735735736 8.251144342318495
11.723746746746746 8.26213032539832
11.735757757757758 8.273117358084624
11.747768768768768 8.284105438230052
11.75977977977978 8.295094563693839
11.77179079079079 8.306084732341773
11.783801801801802 8.317075942046174
11.795812812812812 8.32806819068586
11.807823823823824 8.339061476146133
11.819834834834834 8.350055796318738
11.831845845845846 8.361051149101847
11.843856856856856 8.37204753240003
11.855867867867868 8.383044944124228
11.867878878878878 8.394043382191725
11.87988988988989 8.40504284452613
11.891900900900902 8.416043329057349
11.903911911911912 8.427044833721546
11.915922922922924 8.438047356461142
11.927933933933934 8.449050895224769
11.939944944944946 8.460055447967257
11.951955955955956 8.471061012649598
11.963966966966968 8.482067587238937
11.975977977977978 8.493075169708536
11.98798898898899 8.504083758037748
12.0 8.515093350212
}\mesdonnees
\begin{document}
\maketitle

\begin{abstract}
The present article studies solutions to the compressible Navier-Stokes equations for ideal gases in one dimension when thermal conductivity is present but very weak, {while viscosity is positive and constant}. The main novelty is the establishment of bounds that do not explode when the conductivity coefficient approaches zero. The conductivity coefficient is assumed to be constant and the framework is that of "à la Hoff" solutions. More precisely, the velocity is initially assumed to be regular, while the density and temperature are only in $L^\infty$ and far from zero. A new proof of a stability result for cases without conductivity is given. {Then, the proof of the zero-conductivity limit to the Navier-Stokes system without conduction is established in the "à la Hoff" framework.}
\end{abstract}

\section*{Introduction}
We consider in this article the unidimensional one periodic non-isentropic Navier-Stokes equations for perfect gases, which is written in its Eulerian form as
\begin{empheq}[left=(NS_\kappa)\empheqlbrace]{alignat=1}
\partial_t \rho + \partial_x (\rho u) &= 0 \label{eq:1}\\
\partial_t (\rho u) + \partial_x (\rho u^2) & = \partial_x(\mu\partial_x u - R\rho\theta) \label{eq:2}\\
\partial_t \left( \frac{\rho u^2}{2} + \cv\rho\theta \right)+ \partial_x \left(\frac{\rho u^3}{2} + \cv\rho\theta u \right)  & = \partial_x(\mu(\partial_x u)u - R\rho\theta u) + \partial_x(\kappa \partial_x \theta) \label{eq:3}
\end{empheq}
where $\rho>0$ is the density, $u$ the velocity and $\theta>0$ the temperature. The Navier-Stokes system models the motion of a fluid from a macroscopic point of view. The equation \eqref{eq:1} (continuity equation) describes the conservation of mass. The equation \eqref{eq:2} (momentum equation) comes from the second Newton's law of motion. The left hand side of \eqref{eq:2} is the time derivative of the momentum. The right hand side is the sum of the external forces, from one part some viscous force $\mu\partial_x u$, from an other part a pressure force $p = R\rho\theta$. The quantity $\mu>0$ is called the viscosity, and $R>0$ the constant of perfect gases. Finally, the equation \eqref{eq:3} (energy equation) describes the global conservation of the energy. We define the specific total energy 
\begin{equation*}
E = \frac{u^2}{2} + e
\end{equation*} 
where 
\begin{equation*}
e = \cv\theta
\end{equation*}
is the specific internal energy, and $\cv>0$ the specific heat. We then define the adiabatic constant
\begin{equation*}
\gamma = \frac{R}{\cv} + 1 > 1,
\end{equation*}
so that
\begin{equation}\label{eq:4}
p = (\gamma -1)\rho e = R\rho\theta.
\end{equation}
The right hand side term of \eqref{eq:3} can be seen as an energy flux. We consider some conduction term $\partial_x(\kappa\partial_x\theta)$, where $\kappa>0$ is the heat conductivity. This term models the molecular shaking. In the whole paper, we will assume that $\mu, R, \cv, \gamma, \kappa$ are constants. From \eqref{eq:2} and \eqref{eq:3} we derive the equation on the temperature
\begin{equation}\label{eq:5}
\partial_t (\rho \cv \theta) + \partial_x(\rho u\cv\theta) = (\mu\partial_x u - R\rho\theta)\partial_x u + \partial_x(\kappa\partial_x\theta).
\end{equation}

\ 

The study of the Navier-Stokes non-barotropic system is much harder than the barotropic one, thus there are fewer results. We will not go into detail here about the literature on the study of barotropic cases, but we can cite here as references the famous book by Lions \cite{Lions1996}, and some articles by Hoff \cite{Hoff1986, Hoff1998} and Desjardins \cite{Des}.

For the non-barotropic multidimensional case with heat conduction $(\kappa>0)$, a first uniqueness result is given by Serrin \cite{Se59} in 1959, for strong solutions. Some existence results have been given by Nash \cite{Na62} and Itaya \cite{It71,It75} for short times, and by Matsumura and Nishida \cite{MaNi78} for long times but near an equilibrium. They also give some results on the asymptotic behaviour when $t\rightarrow +\infty$, proving that the solution converges to the equilibrium. Some multidimensional results for initial data close to constants are then obtained by Hoff \cite{Ho96} for solutions with less regularity. Danchin proved some theorems in Besov spaces \cite{Da99, Da01}. A real breakthrough was made in 2018 by Huang and Li \cite{HuLi18}, with an existence result in dimension $3$, for classical and weak solutions, on the whole space, requiring the smallness of the energy but allowing large oscillations of solutions.

The one-dimensional case has been very well studied, because in this simpler case some bounds for large times and large data are available. The founding article is certainly that by Kazikhov and Shelukhin in 1976 \cite{KaSh76}. They indeed show the global existence in dimension $1$, for bounded domains and classical solutions. This result was improved by Kanel \cite{Ka79} and Kazikhov \cite{Ka82}. The authors simplify the proof by introducing the physical entropy
\begin{equation}\label{eq:6}
s = s(\rho,\theta)= \cv \ln \theta - R\ln \rho
\end{equation}
satisfying the useful equation
\begin{equation}\label{eq:7}
\partial_t (\rho s) + \partial_x(\rho s u) = \mu\frac{(\partial_x u)^2}{\theta} + \kappa\frac{(\partial_x \theta)^2}{\theta^2} + \partial_x(\kappa \partial_x(\ln\theta)).
\end{equation}
Moreover, they adapted the proof in order to obtain asymptotic in time behaviour. Global existence was also proved when the domain is $\R$ \cite{Ji98}.
In fact (as remarked by Serre \cite{Se86}), \cite{KaSh76} is not very far from some existence results for weak solutions. Amosov and Zlotnik even studied the case with oscillating coefficients \cite{AmZl92, AmZl88, AmZl97a, AmZl97b}, and Hoff obtained some existence and asymptotic in time behaviour results for small time, unbounded domains, using intermediate regularity between weak and strong \cite{Ho92}. In 2000, Chen, Hoff and Trivisa \cite{ChHoTr00} used a semi-discrete approach and succeeded in proving the global existence and asymptotic in time behaviour of 1D solutions, for bounded domains and without restriction on the data size. Then Jiang and Zhang \cite{JiZh00} and Jiang and Zlotnik \cite{JiZl04} obtained a similar result on $\R$. For strong solutions, Li and Liang managed to obtain the upper bound on $\theta$ on $[0,T]$, for each $T>0$ \cite{LiLi14}. All the articles quoted here strongly use the assumption that the initial density $\rho$ is far from vacuum.
In 2019, Li \cite{Li19} got rid of this hypothesis, in part by exploiting some bounds on the Cauchy stress
\begin{equation}\label{eq:8}
\sigma = \mu\partial_x u - R\rho\theta.
\end{equation}
Even in the barotropic case, this magical quantity is known to give some extra regularity (see \cite{Se91}). It can be explained in physical terms by the principle of opposing forces, and mathematically for example by remarking that from \eqref{eq:1}--\eqref{eq:3}, $\sigma$ satisfies the equation
\begin{equation}\label{eq:9}
\partial_t \sigma + u\partial_x\sigma - \mu \partial_x\left(\frac{\partial_x\sigma}{\rho}\right) = -\gamma\sigma\partial_x u - (\gamma - 1)\partial_x(\kappa\partial_x\theta)
\end{equation}
which is partially parabolic. Hoff strongly used the Cauchy stress in his analysis, particularly in the articles \cite{Ho92,Ho96} cited above. That is one of the main ideas of \cite{HuLi18}, which also allows one to get closer to vacuum.

The case $\kappa = 0$ is less studied and considered more difficult, because of the lack of ellipticity. Some nasty behaviours are known on this subject: for example, Xin and Yan \cite{XiYa12} proved the blow-up of classical solutions if $\rho$ has compact support. However, in the multidimensional case, some existence theorems for smooth solutions had been proved for this type of partially parabolic system from 1982 by Kawashima \cite{Ka83, KaOk82}. In 1997, Liu and Zeng studied the existence of classical solutions near equilibrium, looking at the linearised system \cite{LiZe97}. Duan and Ma then obtained the asymptotic behaviour in dimension 3, for small data in $H^3$ \cite{DuMa08}. The main idea is to consider the change of variables
\begin{equation*}
(\rho,u,\theta)\rightarrow (p, u, s).
\end{equation*}
The authors then consider the system of PDEs on $(p,u)$, then $s$, using some decay information on $(p,u)$. Note that in the case $\kappa = 0$, the equation \eqref{eq:7} becomes
\begin{equation}\label{eq:10}
\partial_t(\rho s) + \partial_x(\rho u s) = \mu\frac{(\partial_x u)^2}{\theta}.
\end{equation}
Using this technique, improvements were obtained by Tang and Wang \cite{TaWa12,TaWa16}. Finally, in 2019, Chen, Tan, Wu and Zou obtained global existence without using the decay of $(p,u)$, by exploiting Besov spaces. Remark that Wu \cite{Wu16} exhibited an interesting equation, derived from \eqref{eq:10}:
\begin{equation}\label{eq:11}
\partial_t (p^{1/\gamma}) + \partial_x(p^{1/\gamma} u) = \frac{\gamma-1}{\gamma}\mu(\partial_x u)^2 p^{1/\gamma - 1},
\end{equation}
a form of writing that is particularly well-suited to variable changes.

In dimension 1, \eqref{eq:9} is easy to use in order to obtain bounds on $\sigma$. Li \cite{Li20} then succeeded in showing the global existence of strong solutions for $(NS_0)$, without any condition of smallness on the data. Adapting this idea for weaker solutions is immediate, and thanks to some tricks we can even adapt the estimates in the case of oscillating coefficients (see \cite{BrBuGJLa24}).

In this article, we will consider solutions on the torus $\T$ with a regularity intermediate between strong (where $\rho,\theta,u\in H^1$) and weak (where $\rho,\theta,u\in L^2$), which we call "à la Hoff" solutions. This framework is specified by
\begin{definition}\label{def:Hoff} Let us define $w:[0,+\infty[\rightarrow \R$ by 
\begin{equation*}
\text{for all }t\in [0,+\infty[,\quad w(t) = \min(1, t).
\end{equation*}
Let $\kappa\geq 0$. We call global "à la Hoff" solution of  $(NS_\kappa)$ any weak solution $(\rho,u,\theta)\in L^2_{loc}(\R_+\times\T)^3$ of $(NS_\kappa)$ that also satisfies, for all $T>0$, 
\begin{equation}\label{eq:12}
\text{for almost all }(t,x)\in[0,T]\times\T,\quad \underline{\rho}\leq \rho(t,x) \leq \overline{\rho},
\end{equation}
\begin{equation}\label{eq:13}
\text{for almost all }(t,x)\in [0,T]\times\T,\quad \underline{\theta}\leq \theta(t,x) \leq \overline{\theta},
\end{equation}
\begin{equation}\label{eq:14}
\begin{split}
\sup_{0\leq t\leq T}\int_0^1\sigma^2 + \sup_{0\leq t\leq T}\int_0^1(\partial_x u)^2 &+ \kappa\int_0^T\int_0^1 (\partial_x\theta)^2 + \int_0^T\int_0^1 (\partial_x\sigma)^2 + \int_0^T\int_0^1(\partial_t u)^2\\&+ \int_0^T \Vert\sigma\Vert_\infty^2 + \int_0^T \Vert \partial_x u\Vert_\infty^2 + \sup_{[0,T]\times\T} u^2 \leq C_1,
\end{split}
\end{equation}
\begin{equation}\label{eq:15}
\sup_{0\leq t\leq T} w\int_0^1 (\partial_x\sigma)^2 + \sup_{0\leq t \leq T} w \int_0^1 \kappa(\partial_x\theta)^2 + \int_0^T w\int_0^1 (\partial_t \sigma)^2 + \int_0^T w \int_0^1[\partial_x(\kappa\partial_x\theta)]^2\leq C_2.
\end{equation}
for some $\underline{\rho},\overline{\rho}, \underline{\theta},\overline{\theta}, C_1, C_2>0$ that may depend on $T$.
\end{definition}

Let $(\rho_0,u_0,\theta_0)\in L^2(\T)^3$ such that
\begin{equation}\label{eq:16}
\text{for almost all }x\in\T,\quad \underline{\rho_0}\leq \rho_0(x) \leq \overline{\rho_0},
\end{equation}
\begin{equation}\label{eq:17}
\text{for almost all }x\in\T,\quad \underline{\theta_0}\leq \theta_0(x) \leq \overline{\theta_0},
\end{equation}
\begin{equation}\label{eq:18}
\int_\T \rho_0 u_0^2 + \int_\T (\partial_x u_0)^2 \leq C_0,
\end{equation}
with some $\underline{\rho_0}, \overline{\rho_0}, \underline{\theta_0}, \overline{\theta_0}, C_0>0$.

\ 

In what follows, we will refer to a global solution as a solution defined on the interval $[0, + \infty[$. The first result of this paper establishes the existence of "à la Hoff" solutions in the case where the conductivity coefficient is positive. To the best of the author's knowledge, this result has never been stated as such, although the various steps involved are known. This is the
\begin{theorem}\label{thm:1} Let $\overline{\kappa}\geq \underline{\kappa}>0.$ Assume that $(\rho_0, u_0,\theta_0)$ satisfies \eqref{eq:16}--\eqref{eq:18}. Then, for all $\overline{\kappa}\geq\kappa\geq \underline{\kappa}$, $(NS_\kappa)$  with initial condition $(\rho_0,u_0,\theta_0)$ admits a unique "à la Hoff" solution. Moreover, for all $T>0$, the constants $\underline{\rho},\overline{\rho}, \underline{\theta},\overline{\theta}, C_1, C_2$ in Definition \ref{def:Hoff} can be chosen to depend only on $C_0,\cv,\gamma,\underline{\kappa},\overline{\kappa},\mu,\underline{\rho_0}, \overline{\rho_0}, T, \underline{\theta_0}, \overline{\theta_0}$. 
\end{theorem}
The bounds \eqref{eq:12} and \eqref{eq:14} were obtained by Li \cite{Li19} for strong solutions, and can be easily adapted in the “à la Hoff” framework. We then use ideas from Hoff \cite{Ho92,Ho96} to establish the upper bound on the temperature in \eqref{eq:13}. Let us now state one of the main result of this article, establishing the existence of "à la Hoff" solutions with bounds that do not explode when the conductivity tends towards zero. This is a stronger result than Theorem \ref{thm:1}.
\begin{theorem}\label{thm:2}
Let $\overline{\kappa}>0$. Assume that $(\rho_0, u_0,\theta_0)$ satisfies \eqref{eq:16}--\eqref{eq:18}. Then, for all $\overline{\kappa}\geq\kappa\geq 0$, $(NS_\kappa)$ with initial condition $(\rho_0,u_0,\theta_0)$ admits a unique global "à la Hoff" solution. Moreover, for all $T>0$, the constants $\underline{\rho},\overline{\rho}, \underline{\theta},\overline{\theta}, C_1, C_2$ in Definition \ref{def:Hoff} can be chosen to depend only on $C_0,\cv,\gamma,\overline{\kappa},\mu,\underline{\rho_0}, \overline{\rho_0}, T, \underline{\theta_0}, \overline{\theta_0}$. 
\end{theorem}

The third theorem in this paper is a stability result for the Navier-Stokes system without conductivity. We give here a new proof of this theorem already stated in \cite{BrBuGJLa24}, based on the identity \eqref{eq:11}.
\begin{theorem}\label{thm:3}
Let $(\rho_0^n, u_0^n,\theta_0^n)_n\subset L^2(\T)^3$ be a sequence of triplets satisfying \eqref{eq:16}--\eqref{eq:18} with  $\underline{\rho_0},\overline{\rho_0}, \underline{\theta_0},\overline{\theta_0}, C_0$ not depending on $n$. Assume that there exists some $(\rho_0,u_0,\theta_0)\in L^2(\T)\times H^1(\T)\times L^2(\T)$ such that
\begin{equation*}
\rho_0^n,\theta_0^n \tend{n}{+\infty} \rho_0,\theta_0\quad \text{ in }L^2(\T),\quad u_0^n \tendf{n}{+\infty} u_0\quad \text{ in } H^1(\T).
\end{equation*}  
Then, for all $T>0$, there exists some $(\rho,u,\theta)\in L^2(0,T,L^2(\T))\times L^2(0,T,H^1(\T))\times L^2(0,T,L^2(\T))$ such that the "à la Hoff" solution $(\rho^n, u^n,\theta^n)$ of $(NS_0)$ with initial conditions $\rho_0^n,u_0^n,\theta_0^n$ verifies
\begin{equation*}
\rho^n,\theta^n \tend{n}{+\infty} \rho,\theta\quad \text{ in }L^2(0,T,L^2(\T)),\quad u^n \tendf{n}{+\infty}u\quad \text{ in } L^2(0,T,H^1(\T)). 
\end{equation*}
Moreover, $(\rho,u,\theta)$ is the solution of $(NS_0)$ with initial conditions $(\rho_0,u_0,\theta_0)$. 
\end{theorem}
Finally, the second main theorem of this article establishes the zero-conductivity limit of the Navier-Stokes system for non-barotropic ideal gases.
\begin{theorem}\label{thm:main}
Let $(\rho_0,\theta_0,u_0)$ satisfy \eqref{eq:16}--\eqref{eq:18}, for some constants $\underline{\rho_0},\overline{\rho_0},\underline{\theta_0},\overline{\theta_0},C_0>0$. Let $\overline{\kappa}>0$.
Then, for all $T>0$, there exists some $(\rho,u,\theta)\in L^2(0,T,L^2(\T))\times L^2(0,T,H^1(\T))\times L^2(0,T,L^2(\T))$ such that the "à la Hoff" solution $(\rho^{{\kappa}}, u^{{\kappa}},\theta^{{\kappa}})_{0\leq\kappa\leq\overline{\kappa}}$ of $(NS_{{\kappa}})$ with initial conditions $\rho_0,\theta_0, u_0$ verifies
\begin{equation*}
\rho^{{\kappa}},\theta^{{\kappa}} \tend{\kappa}{0} \rho,\theta\quad \text{ in }L^2(0,T,L^2(\T)),\quad u^{{\kappa}} \tendf{\kappa}{0}u\quad \text{ in } L^2(0,T,H^1(\T)). 
\end{equation*}
Moreover, $(\rho,u,\theta)$ is the "à la Hoff" solution of $(NS_0)$ with initial conditions $(\rho_0,u_0,\theta_0)$.
\end{theorem}
Note that this last theorem is by no means obvious; indeed, the bounds known so far on $\rho$, $u$ and $\theta$, whether in the case of strong or weak solutions, degenerated as $\kappa$ tends towards 0. The uniform $\kappa$ bounds obtained in Theorem \ref{thm:2} are therefore crucial to the proof of Theorem \ref{thm:main}.

\

For $\kappa>0$, the uniqueness of the solution to $(NS_\kappa)$ is already known. The same holds for $\kappa=0$ (see \cite{BrBuGJLa24}). Moreover, the existence of strong solutions is established for $\kappa>0$ (see \cite{LiLi14,Li19}) and for $\kappa=0$ (see \cite{Li20}). The stability result Theorem \ref{eq:3} is straightforward to adapt for a fixed $\kappa>0$. Therefore, in order to prove Theorems \ref{thm:1} and \ref{thm:2}, it suffices to show that all strong solutions verifying \eqref{eq:16}--\eqref{eq:18} also satisfy \eqref{eq:12}--\eqref{eq:15}. Indeed, to construct an "à la Hoff" solution, we can approximate the initial condition $(\rho_0,\theta_0,u_0)\in L^2(\T)^3$ satisfying \eqref{eq:16}--\eqref{eq:18} by a triplet in $H^1(\T)^3$, and then use the stability result. Section \ref{sect:1} is devoted to the proof of the bounds \eqref{eq:12}--\eqref{eq:15}. Section \ref{sect:2} is devoted to the proof of Theorem \ref{thm:3} and the proof of some weak version of Theorem \ref{thm:main}, for well-prepared data (see Proposition \ref{prop:4prime}). Finally, section \ref{sect:3} establishes a stability result at fixed $\kappa$, which combined with Proposition \ref{prop:4prime}, allow us to show Theorem \ref{thm:main}.

\

\begin{notation}
For some $f:\T\rightarrow \R$ and $p\in [1,+\infty]$, we will denote $\Vert f\Vert_p$ the $L^p$ norm of $f$. If $f$ is also a function of $t\in [0,T]$, then $\Vert f\Vert_p$ is a function, $\Vert f\Vert_p : [0,T]\rightarrow [0,+\infty[$.
\end{notation}

\begin{notation}
For the rest of the analysis, let us introduce the total derivative $D_t$ formally defined for some function $f:\R_+\rightarrow \R$ by
\begin{equation*}
D_t f := \partial_t f + u\partial_x f.
\end{equation*}
Note that from \eqref{eq:1} we get
\begin{equation*}
\rho D_t f = \partial_t(\rho f) + \partial_x(\rho u f).
\end{equation*}
We then define a reciprocal at right operator $D_t^{-1}$ by
\begin{empheq}[left=\empheqlbrace]{align*}
\partial_t D_t^{-1} f + u\partial_x D_t^{-1} f &= f,\\
D_t^{-1} f(0,\cdot) &= 0.
\end{empheq}
Note that $D_t^{-1}$ is well defined supposing $\partial_x u\in L^1(0,T,L^\infty(\T))$, which corresponds well with the framework of the analysis (see \eqref{eq:14}). Then, defining
\begin{equation}\label{eq:19}
R_t f := f - D_t^{-1} D_t f,
\end{equation}
we get
\begin{empheq}[left=\empheqlbrace]{align*}
D_t (R_t f) &= 0,\\
(R_t f)(0,\cdot) & = f(0, \cdot),
\end{empheq}
then by some maximal principle
\begin{equation}\label{eq:20}
\operatorname*{ess\,inf}_{x\in \T} f(0,x) \leq R_t f \leq \operatorname*{ess\,sup}_{x\in \T} f(0,x)\quad \text{a.e.}.
\end{equation}
In particular, if $f(0,\cdot) = 0$, then $R_t f = 0$ and $D_t^{-1} D_t f = f$. Moreover, $D_t^{-1}$ is a non-negative operator, and in particular, for all $t\in [0,+\infty[$,
\begin{equation}\label{eq:21}
\vert D_t^{-1} f\vert(t) \leq \int_0^t \Vert f\Vert_\infty.
\end{equation}
Remark moreover that
\begin{equation}\label{eq:22}
D_t^{-1} \int_0^1 f(t) = \int_0^t\int_0^1 f.
\end{equation}
\end{notation}

\begin{notation}
Finally, for some function $f:\T\rightarrow \R$, let us formally define $\partial_x^{-1}$ the reciprocal at right operator of $\partial_x$, reciprocal at left on the space of zero mean functions, by
\begin{equation*}
\text{for all }x\in\T,\quad\partial_x^{-1} f (x) = \int_0^1dy\int_y^x f(z)dz.
\end{equation*}
In particular,
\begin{equation}\label{eq:23}
\vert\partial_x^{-1}f\vert\leq \int_0^1 \vert f\vert.
\end{equation}
Remark that $\partial_x^{-1}f$ is one-periodical if and only if $\int_0^1 f = 0$. Thus we get the formula
\begin{equation}\label{eq:24}
\partial_x^{-1}\partial_x f = f - \int_0^1 f.
\end{equation}
\end{notation}

{In all that follows, $\mu>0$, $\cv>0$, $\gamma>1$ and $R$ are fixed once and for all. All the bounds in the various energy estimates will implicitly depend on these variables.}
\section{A priori estimates}\label{sect:1}

This section is divided as follows. The first sub-section is devoted to the establishment of the main estimates valid for $\kappa\geq 0$. Finding a bound on the Cauchy stress is crucial and requires slightly different techniques depending on whether $\kappa$ is far from or close to zero. Thus, the second sub-section deals with this issue in the case where $\kappa$ is far from zero, mainly bringing together ideas from \cite{KaSh76}, \cite{Ho92} and \cite{Li20}. Then, the third sub-section deals with the case where $\kappa$ is small. Finaly, the fourth sub-section is devoted to establish the bound \eqref{eq:15} also known as the second Hoff energy estimate. For the following, let us take $(\rho_0,u_0,\theta_0)\in H^1(\T)^3$ satisfying \eqref{eq:16}--\eqref{eq:18} with some $\underline{\rho_0},\overline{\rho_0},\underline{\theta_0},\overline{\theta_0}, C_0>0$. Let us fix some $C_0,\underline{\kappa},\overline{\kappa},\underline{\rho_0},\overline{\rho_0},T,\underline{\theta_0},\overline{\theta_0}>0$. We denote $(\rho,u,\theta)$ the strong solution of \eqref{eq:1}--\eqref{eq:3} with initial condition $(\rho_0,u_0,\theta_0)$.
\subsection{\texorpdfstring{Some classical estimates for $\kappa \geq 0$}{Some classical estimates for kappa ≥ 0}}
\begin{prop}\label{prop:1}
Let's denote
\begin{equation*}
{\cal M} = \int_0^1\rho_0,\quad {\cal E} = \int_0^1\rho_0 E_0.
\end{equation*}
Then $0<{\cal M}\leq \overline{\rho_0}$, $0<{\cal E}\leq \overline{\rho_0}\cv\overline{\theta_0} + C_0/2$, and
\begin{equation}
   \text{for all }t\in[0,T],\int_0^1\rho(t) = \cM\label{eq:25},
\end{equation}
\begin{equation}
    \text{for all }t\in[0,T],\int_0^1 \rho E(t) = \cE, \label{eq:26}
\end{equation}
\begin{equation}
    \sup_{0\leq t\leq T}\int_0^1 p(t) \leq (\gamma-1)\cE,
\label{eq:27}
\end{equation}
\begin{equation}
    \sup_{0\leq t\leq T}\int_0^1 \frac{\rho u^2}{2}(t) \leq \cE,
\label{eq:28}
\end{equation}
\begin{equation}
    \sup_{0\leq t \leq T}\int_0^1 \vert\rho u\vert(t)\leq  \sqrt{2\cM\cE}.
\label{eq:29}
\end{equation}
\end{prop}
\begin{proof}
Integrating \eqref{eq:1} (resp. \eqref{eq:3}) on the torus gives \eqref{eq:25} (resp. \eqref{eq:26}). We deduce \eqref{eq:27}, \eqref{eq:28} from \eqref{eq:26}, \eqref{eq:4} and the positivity of $\rho$ and $p$. Finally, using Cauchy-Schwarz inequality,
\begin{equation*}
    \int_0^1 \vert\rho u\vert(t) = \sqrt{2}\int_0^1 \sqrt{\rho}\sqrt{\rho}\frac{\vert u\vert}{\sqrt{2}}\leq \sqrt{2}\sqrt{\int_0^1\rho}\sqrt{\int_0^1 \frac{\rho u^2}{2}}\leq  \sqrt{2\cM\cE}.
\end{equation*}
\end{proof}
\begin{remark}
Note that, for any $v\in\R$, the change of unknows
\begin{equation*}
    (\rho(t,x),u(t,x),E(t,x)) \rightarrow (\rho(t,x-vt), u(t,x-vt)+v, E(t,x-vt))
\end{equation*}
is a solution of $(NS_\kappa)$ (it is the Galilean invariance principle). Choosing
\begin{equation*}
    v = -\frac{1}{\cM}\int_0^1\rho_0 u_0,
\end{equation*}
we can assume without any loss of generality
\begin{equation*}
    \int_0^1\rho_0u_0 = 0.
\end{equation*}
Then, integrating \eqref{eq:2} on the torus, we get
\begin{equation}\label{eq:30}
    \text{for all }t\in [0,T],\quad \int_0^1\rho u(t) = \int_0^1\rho_0 u_0 = 0.
\end{equation}
Roughly speaking, we can assume that the mean of the moments vanishes, and we keep this assumption for the remainder of the analysis.
\end{remark}

The next proposition is related to entropy.
\begin{prop}
\label{prop:2}
There exists $H_1 = H_1({\cal E}, \underline{\rho_0}, \overline{\rho_0}, \underline{\theta_0}, \overline{\theta_0})>0$ such that
\begin{equation}\label{eq:31}
\int_0^T\int_0^1\mu\frac{(\partial_x u)^2}{\theta}+\int_0^T\int_0^1 \kappa \frac{(\partial_x \theta)^2}{\theta^2} \leq H_1.
\end{equation}
\end{prop}
\begin{proof}
Considering the nonnegative function $h:]0,+\infty[\rightarrow \R$ defined by
\noindent

\vspace{0.5cm}
\begin{minipage}{0.5\textwidth}
\begin{equation*}
\text{for all }x \in ]0,+\infty[, \quad h(x) = x - 1 - \ln(x),
\end{equation*}
\end{minipage}%
\hfill
\begin{minipage}{0.45\textwidth}
\begin{minipage}{0.1\textwidth}
\begin{tikzpicture}
\begin{axis}[
    height=3cm,
    width=4cm,
    xmin=0,
    xmax=8,
    axis lines=middle,
    xtick={1},
    ytick=\empty
]
\addplot[mark=none, line width=1pt] table {\mesdonnees};
\end{axis}
\end{tikzpicture}
\end{minipage}%
\hfill
\begin{minipage}{0.45\textwidth}
\small\textbf{Figure 1:} Representation of the function \(h\)
\end{minipage}
\begin{minipage}{0.15\textwidth}
~
\end{minipage}
\end{minipage}

\noindent As $\rho>0$ and $\theta>0$, we get from \eqref{eq:6}
\begin{equation}\label{eq:32}
\rho \cv h(\theta) +\rho R h(1/\rho) + \rho s = \cv\rho\theta - \cv\rho - R\rho + R.
\end{equation}
Hence, derivating \eqref{eq:32} in time then integrating on the torus, using \eqref{eq:7} and \eqref{eq:25}, we obtain

\begin{equation}\label{eq:33}
\frac{d}{dt}\int_0^1 \left(\rho \cv h(\theta) + \rho R h(1/\rho)\right) + \int_0^1 \left(\mu\frac{(\partial_x u)^2}{\theta} + \kappa\frac{(\partial_x \theta)^2}{\theta^2}\right) = \frac{d}{dt}\int_0^1 \rho \cv\theta.
\end{equation}
Finally, integrating \eqref{eq:33} on $[0,T]$, we get
\begin{align}
\int_0^1 \rho \cv h(\theta)(T) + \rho R h(1/\rho)(T) &+ \int_0^T\int_0^1 \mu \frac{(\partial_x u)^2}{\theta} + \kappa\frac{(\partial_x\theta)^2}{\theta^2} \nonumber\\&= \int_0^1 \rho \cv \theta(T) - \int_0^1 \rho_0 \cv \theta_0 + \int_0^1 \rho_0 \cv h(\theta_0) + \rho_0 R h(1/\rho_0) \nonumber\\&\leq {\cal E} + \cv {\cal M} \sup_{[\underline{\theta_0},\overline{\theta_0}]} h  + R{\cal M} \sup_{[1/\overline{\rho_0},1/\underline{\rho_0}]} h. \nonumber\end{align}
and in particular the conclusion, because $h\geq 0$ over $R_+^*$.
\end{proof}
\begin{remark}
Controlling
\begin{equation*}
\int_0^T\int_0^1\mu \frac{(\partial_x u)^2}{\theta} 
\end{equation*}
provides some dissipation information. In fact, up to the author's knowledge, this has not been used get. In order to do so, we would need to obtain an upper bound on the temperature, which is one of the major difficulties of the problem.
\end{remark}
The next bound gives some information on the Cauchy stress. It has long been used for both barotropic and non-insentropic Navier-Stokes equations (see \cite{Se91}).
\begin{prop}\label{prop:3}
There exists $H_2 = H_2({\cal E}, {\cal M}, T)>0$ such that
\begin{equation}\label{eq:34}
\text{for almost all }(t,x)\in [0,T]\times \T,\quad \vert D_t^{-1}\sigma(t,x)\vert\leq H_2.
\end{equation} 
\end{prop}
\begin{proof}
Applying $\partial_x^{-1}$ to \eqref{eq:2}, we get, using \eqref{eq:24} and the periodicity of $u$,
\begin{equation}\label{eq:35}
\partial_t \partial_x^{-1}(\rho u) + \partial_x^{-1}\partial_x(\rho u^2) = \sigma - \int_0^1 \sigma = \sigma + \int_0^1 p.
\end{equation}
Due to \eqref{eq:24},
\begin{equation}\label{eq:36}
\partial_x^{-1}\partial_x(\rho u^2) = \rho u^2 - \int_0^1\rho u^2 = u\partial_x\partial_x^{-1}(\rho u) - \int_0^1\rho u^2.
\end{equation}
Then \eqref{eq:36} in \eqref{eq:35} gives
\begin{equation}\label{eq:37}
D_t \partial_x^{-1}(\rho u) = \sigma + \int_0^1 p+ \int_0^1\rho u^2.
\end{equation}
Finally, applying $D_t^{-1}$ to \eqref{eq:37}, we get using \eqref{eq:22},
\begin{equation*}
D_t^{-1}\sigma = D_t^{-1} D_t \partial_x^{-1}(\rho u)-\int_0^t\int_0^1 p -\int_0^t \int_0^1 \rho u^2,
\end{equation*}
thus from \eqref{eq:19}, \eqref{eq:20}, \eqref{eq:27}, \eqref{eq:28}, \eqref{eq:29}, \eqref{eq:30} and \eqref{eq:23},
\begin{align}
\vert D_t^{-1} \sigma\vert &\leq \vert \partial_x^{-1}(\rho u)\vert + \vert R_t\partial_x^{-1}(\rho u)\vert + \left\vert\int_0^t\int_0^1 p\right\vert + \left\vert\int_0^t\int_0^1\rho u^2\right\vert \nonumber\\&\leq \int_0^1 \vert \rho u\vert + \Vert\partial_x^{-1}(\rho_0 u_0)\Vert_\infty + T(\gamma+1){\cal E} \nonumber\\&\leq 2\sqrt{2{\cal M}{\cal E}} + T(\gamma+1){\cal E}. \nonumber\end{align}
\end{proof}

From Proposition \ref{prop:3}, we deduce an upper bound on $\rho$.
\begin{prop}\label{prop:4}
There exists some $\overline{\rho} = \overline{\rho}({\cal E}, \overline{\rho_0}, T)$ such that
\begin{equation*}
\text{for almost all }(t,x)\in [0,T]\times \T,\quad \rho(t,x)\leq \overline{\rho}.
\end{equation*} 
\end{prop}
\begin{proof}
From \eqref{eq:1} and \eqref{eq:8} we obtain
\begin{equation}\label{eq:38}
D_t\ln\rho = -\partial_x u = \frac{-\sigma - p}{\mu}.
\end{equation}
Hence, applying $D_t^{-1}$ to \eqref{eq:38} then using \eqref{eq:19}, \eqref{eq:20} and \eqref{eq:34},
\begin{align}
\ln \rho &= R_t\ln\rho -\frac{D_t^{-1}\sigma}{\mu} - \frac{D_t^{-1} p}{\mu}\label{eq:39}
\\&\leq \ln \overline{\rho_0} + \frac{\vert D_t^{-1}\sigma\vert}{\mu} \nonumber\\&\leq \ln \overline{\rho_0} + \frac{H_2}{\mu} \nonumber\end{align}
because $D_t^{-1} p \geq 0$ (because $p=R\rho\theta\geq 0$). Finally, passing to the exponential,
\begin{equation*}
\text{for almost all }(t,x)\in [0,T]\times\T,\quad\rho(t,x) \leq \overline{\rho_0}\exp{(H_2/\mu)}.
\end{equation*}
\end{proof}

A lower bound is then established for the temperature.
\begin{prop}\label{prop:5}
There exists some $\underline{\theta} = \underline{\theta}({\cal E}, {\cal M}, \overline{\rho_0}, T, \underline{\theta_0})$ such that
\begin{equation*}
\text{for almost all }(t,x)\in [0,T]\times\T,\quad \theta(t,x)\geq \underline{\theta}.
\end{equation*}
\end{prop}
\begin{proof}
Let $\beta:\R\rightarrow\R$ be some convex and nonincreasing function. Multiplying \eqref{eq:5} by $\beta'(\theta)$ we get
\begin{align}\label{eq:40}
\partial_t(\rho\cv\beta(\theta)) + \partial_x(\rho u\cv\beta(\theta)) &= \beta'(\theta)\sigma\partial_x u + \beta'(\theta)\partial_x(\kappa\partial_x\theta).
\end{align}
Note that
\begin{equation}\label{eq:41}
\sigma\partial_x u = \frac{\sigma^2}{\mu} + \frac{\sigma R\rho\theta}{\mu} = \frac{(\sigma + R\rho\theta/2)^2}{\mu} - \frac{R^2\rho^2\theta^2}{4\mu}\geq \frac{-R^2\rho^2\theta^2}{4\mu}
\end{equation}
and $\beta'(\theta)\leq 0$, hence \eqref{eq:40} and \eqref{eq:41} give
\begin{equation}\label{eq:42}
\partial_t(\rho\cv\beta(\theta)) + \partial_x(\rho u\cv\beta(\theta)) \leq -\frac{R^2\rho^2\theta^2\beta'(\theta)}{4\mu} + (\partial_x(\kappa\partial_x\theta))\beta'(\theta).
\end{equation} 
Integrating \eqref{eq:42} on the torus we get by integration by parts
\begin{equation*}
\frac{d}{dt}\int_0^1 \rho\cv\beta(\theta) \leq -\int_0^1 \frac{R^2\rho^2\theta^2\beta'(\theta)}{4\mu} - \int_0^1 \kappa(\partial_x\theta)^2\beta''(\theta) \leq -\int_0^1\frac{R^2\rho^2\theta^2\beta'(\theta)}{4\mu}
\end{equation*}
because $\beta''(\theta)\geq 0$. Choosing now $\beta :x\mapsto (1/x)^k$ for $k\in\N^*$, we obtain
\begin{equation}\label{eq:43}
\frac{d}{dt}\int_0^1\frac{\rho\cv}{\theta^k}\leq \frac{R^2\overline{\rho}k}{4\mu}\int_0^1\frac{\rho}{\theta^{k-1}} = \frac{R^2\overline{\rho}k}{4\mu\cv^{1-1/k}}\int_0^1\frac{\rho^{1-1/k}\cv^{1-1/k}\rho^{1/k}}{\theta^{k-1}} \leq \frac{R^2\overline{\rho} k}{4\mu\cv^{1-1/k}}{\cal M}^{1/k}\left(\int_0^1\frac{\rho\cv}{\theta^k}\right)^{1-1/k}
\end{equation}
using Hölder's inequality then \eqref{eq:25}. Hence, multiplying \eqref{eq:43} by 
\begin{equation*}
\frac{1}{k}\left(\int_0^1 \frac{\rho\cv}{\theta^k}\right)^{1/k-1},
\end{equation*}
we get
\begin{equation}\label{eq:44}
\frac{d}{dt}\left(\int_0^1\frac{\rho\cv}{\theta^k}\right)^{1/k} \leq \frac{R^2\overline{\rho}}{4\mu\cv^{1-1/k}}{\cal M}^{1/k}.
\end{equation}
Moreover, as $\rho>0$ and $\rho$ is smooth, $\rho$ has some positive lower bound which can depend on $\Vert\rho\Vert_{H^1}$. Using the convergence
\begin{equation*}
\text{for all }f\in L^\infty(\T),\quad \Vert f\Vert_p \tend{p}{+\infty} \Vert f\Vert_\infty,
\end{equation*}
we have
\begin{equation*}
\left(\int_0^1 \frac{\rho\cv}{\theta^k}\right)^{1/k}\underset{k\rightarrow +\infty}{\rightarrow}\left\Vert \frac{1}{\theta} \right\Vert_\infty \quad\text{in }{\cal D}'(0,T),
\end{equation*}
hence from \eqref{eq:44}
\begin{equation}\label{eq:45}
\frac{d}{dt}\left\Vert \frac{1}{\theta}\right\Vert_{L^\infty_x} \leq \frac{R^2\overline{\rho}}{4\mu\cv}.
\end{equation}
Finally, integrating \eqref{eq:45} on $[0,t]$ for $t\in [0,T]$, we get
\begin{equation*}
\text{for almost every }(t,x)\in [0,T]\times \T,\quad\theta(t,x) \geq \frac{1}{1/\underline{\theta_0} + T R^2\overline{\rho}/(4\mu\cv)}.
\end{equation*}
\end{proof}

We then establish a link between the lower bound on $\rho$ and the upper bound on $\theta$.
\begin{prop}
\label{prop:6}
There exists some $H_3 = H_3({\cal E}, \underline{\rho_0}, \overline{\rho_0}, T)>0$ such that
\begin{equation}\label{eq:46}
\text{for almost all }(t,x)\in [0,T]\times \T,\quad 1/\rho(t,x) \leq H_3\left(1+\int_0^t \Vert\theta\Vert_\infty\right).
\end{equation}
\end{prop}
\begin{proof}
Applying the exponential function to \eqref{eq:39}, we obtain
\begin{equation}\label{eq:47}
\rho \exp(D_t^{-1} p/\mu) = \exp(R_t\ln \rho - D_t^{-1}\sigma/\mu) =: B.
\end{equation}
Remark that, from \eqref{eq:20} and \eqref{eq:34} we get
\begin{equation}\label{eq:48}
\text{for almost all }(t,x)\in [0,T]\times \T,\quad 0 < \underline{B}\leq B(t,x)\leq \overline{B}
\end{equation}
where
\begin{equation*}
\underline{B} = \underline{\rho_0}\exp(- H_2/\mu),\quad \overline{B}=\overline{\rho_0}\exp(H_2/\mu).
\end{equation*}
Multiplying now \eqref{eq:47} by $R\theta/\mu$, we obtain
\begin{equation*}
D_t \exp(D_t^{-1} p/\mu) = B R\theta/\mu,
\end{equation*}
and applying $D_t^{-1}$,
\begin{equation}\label{eq:49}
\exp(D_t^{-1} p/\mu) = 1 + D_t^{-1}(B R\theta/\mu),
\end{equation}
because $D_t^{-1} p(0,\cdot) = 0$.
Combining \eqref{eq:49} with \eqref{eq:47}, we get
\begin{equation}\label{eq:50}
B/\rho = 1 + D_t^{-1}(BR\theta/\mu).
\end{equation}
Dividing \eqref{eq:50} by $B$, we finally get, using \eqref{eq:48}, the nonnegativity of $D_t^{-1}$ and \eqref{eq:21},
\begin{equation*}
1/\rho = 1/B + D_t^{-1}(B R\theta/\mu)/B \leq \frac{1}{\underline{B}}\left(1 + (\overline{B}R/\mu) \int_0^t\Vert\theta\Vert_\infty\right)\leq H_3\left(1 + \int_0^t \Vert \theta\Vert_\infty\right)
\end{equation*}
for all $t\in[0,T]$ and with
\begin{equation*}
H_3 = \max((\overline{B}R/\mu)/\underline{B}, 1/\underline{B}).
\end{equation*}
\end{proof}

The key to continuing the a priori estimates is to obtain the bound
\begin{equation}\label{eq:51}
\sup_{0\leq t\leq T}\int_0^1 \sigma^2 + \int_0^T\int_0^1 (\partial_x\sigma)^2 \leq H_4
\end{equation}
for some $H_4\in \R_+^*$. We will address this in the next two subsections, depending on whether $\kappa$ is considered to be far from or close to zero. From \eqref{eq:51}, we immediatly get the
\begin{prop}\label{prop:7}
Assume that \eqref{eq:51} is satisfied for some $H_4>0$. Then
\begin{equation}\label{eq:52}
\int_0^T\Vert \sigma\Vert_\infty^2 \leq (2T+1)H_4.
\end{equation}
\end{prop}
\begin{proof}
Let us just recall the Gagliardo-Nirenberg inequality
\begin{equation}\label{eq:53}
\Vert \sigma\Vert_{\infty}^2\leq \Vert \sigma\Vert_{2}^2 + 2\Vert \sigma \Vert_{2}\Vert\partial_x \sigma\Vert_{2}.
\end{equation}
Thus, by Young's inequality,
\begin{align*}
\int_0^T\Vert\sigma\Vert_\infty^2 \leq 2 \int_0^T \int_0^1 \sigma^2 + \int_0^T\int_0^1 (\partial_x\sigma)^2
\end{align*}
and finally \eqref{eq:52}, using \eqref{eq:51}.
\end{proof}
We can now obtain some upper bound on $\theta$. Note that the idea of the following proposition was already used by Hoff \cite{Hoff1986}.
\begin{prop}\label{prop:8}
Assume that \eqref{eq:51} is verified for some $H_4>0$. Then there exists some \newline$\overline{\theta} = \overline{\theta}({\cal E}, H_4, \overline{\rho_0}, T, \overline{\theta}_0)$ such that
\begin{equation}\label{eq:54}
\text{for almost all }(t,x)\in [0,T]\times \T,\quad \theta(t,x)\leq \overline{\theta}.
\end{equation}
\end{prop}
\begin{proof}
Let $\beta:\R_+\rightarrow\R$ be some convex and non decreasing function. Multiplying \eqref{eq:5} by $\beta'(\theta)$ we get
\begin{equation}\label{eq:55}
\partial_t(\rho\cv\beta(\theta)) + \partial_x(\rho u\cv\beta(\theta)) = (\sigma\partial_x u)\beta'(\theta) + (\partial_x(\kappa\partial_x\theta))\beta'(\theta).
\end{equation}
Integrating now \eqref{eq:55} on the torus, and using some integration by parts,
\begin{equation*}
\frac{d}{dt}\int_0^1\rho\cv\beta(\theta) + \int_0^1\kappa (\partial_x\theta)^2\beta''(\theta) = \int_0^1(\sigma\partial_x u)\beta'(\theta) = \frac{1}{\mu}\int_0^1 \sigma^2\beta'(\theta) + \frac{R}{\mu}\int_0^1 \sigma \rho \theta\beta'(\theta),
\end{equation*}
hence
\begin{equation*}
\frac{d}{dt}\int_0^1 \rho\cv\beta(\theta)\leq \frac{\Vert \sigma\Vert_\infty^2}{\mu}\int_0^1\beta'(\theta) + \frac{R}{\mu}\Vert\sigma\Vert_\infty\int_0^1 \rho\theta\beta'(\theta),
\end{equation*}
since $\beta''(\theta)\geq 0$ and $\beta'(\theta)\geq 0$. Choosing $\beta:x\mapsto x^k$ for $k\in\N^*$, we get
\begin{equation}\label{eq:56}
\frac{d}{dt}\int_0^1\rho \cv\theta^k \leq \frac{k\Vert\sigma\Vert_\infty^2}{\mu}\int_0^1 \theta^{k-1} + \frac{Rk}{\mu\cv}\Vert\sigma\Vert_\infty\int_0^1 \rho\cv\theta^k.
\end{equation}
Multiplying now \eqref{eq:56} by 
\begin{equation*}
\frac{1}{k}\left(\int_0^1 \rho\cv\theta^k\right)^{1/k - 1}
\end{equation*}
we obtain
\begin{equation}\label{eq:57}
\frac{d}{dt}\left(\int_0^1\rho\cv\theta^k\right)^{1/k}\leq \frac{\Vert\sigma\Vert_\infty^2}{\mu}\int_0^1 \theta^{k-1}\left(\int_0^1\rho\cv\theta^k\right)^{1/k-1} + \frac{R}{\mu\cv}\Vert\sigma\Vert_\infty\left(\int_0^1\rho\cv\theta^k\right)^{1/k}. 
\end{equation}
Moreover, by Jensen's inequality,
\begin{equation*}
\int_0^1\theta^{k-1}\leq \left(\int_0^1 \theta^k\right)^{1-1/k},
\end{equation*}
hence
\begin{equation}\label{eq:58}
\int_0^1\theta^{k-1}\left(\int_0^1\rho\cv\theta^k\right)^{1/k-1}\leq \overline{\rho}^{1/k-1}\cv^{1/k-1}.
\end{equation}
Thus \eqref{eq:58} in \eqref{eq:57} gives
\begin{equation*}
\frac{d}{dt}\left(\int_0^1\rho\cv\theta^k\right)^{1/k}\leq \frac{\overline{\rho}^{1/k-1}\cv^{1/k-1}}{\mu}\Vert\sigma\Vert_\infty^2+ \frac{R}{\mu\cv}\Vert\sigma\Vert_\infty\left(\int_0^1\rho\cv\theta^k\right)^{1/k}. 
\end{equation*}
Thus by Grönwall's lemma, for all $t\in [0,T]$,
\begin{equation}\label{eq:59}
\begin{split}
\left(\int_0^1\rho\cv\theta^k\right)^{1/k}(t)&\leq \left(\int_0^1\rho_0\cv\theta_0^k\right)^{1/k}\exp\left(\int_0^t \frac{R}{\mu\cv}\Vert\sigma\Vert_\infty\right) \\&+ \frac{\overline{\rho}^{1/k-1}\cv^{1/k-1}}{\mu}\int_0^t \Vert\sigma\Vert_\infty^2(s)\exp\left(\int_s^t \frac{R}{\mu\cv}\Vert\sigma\Vert_\infty\right).
\end{split}
\end{equation}
Hence, from \eqref{eq:52} and Cauchy-Schwarz's inequality
\begin{equation}\label{eq:60}
\int_0^t \Vert\sigma\Vert_\infty \leq \sqrt{t}\sqrt{\int_0^t\Vert\sigma\Vert_\infty^2},
\end{equation}
we get passing to the sup in \eqref{eq:59}
\begin{equation}\label{eq:61}
\begin{split}
&\sup_{0\leq t \leq T}\left(\int_0^1\rho\cv\theta^k\right)^{1/k}(t)\\&\leq \left[\left(\int_0^1\rho_0\cv\theta_0^k\right)^{1/k} +\frac{\overline{\rho}^{1/k-1}\cv^{1/k-1}}{\mu}(2T+1)H_4 \right]\exp\left(\frac{R}{\mu\cv}\sqrt{T}\sqrt{(2T+1)H_4}\right).
\end{split}
\end{equation}
Finally, taking $k\rightarrow +\infty$ in \eqref{eq:61}, we obtain
\begin{equation*}
\text{for almost all }(t,x)\in [0,T]\times \T,\quad\theta(t,x) \leq \left(\overline{\theta_0}+\frac{(2T+1)H_4}{\mu\overline{\rho}\cv} \right)\exp\left(\frac{R}{\mu\cv}\sqrt{T}\sqrt{(2T+1)H_4}\right).
\end{equation*}
\end{proof}
We then get, as a corollary,
\begin{prop}\label{prop:9}
Assume that \eqref{eq:51} is verified for some $H_4>0$. Then there exists some \newline$\underline{\rho} = \underline{\rho}( {\cal E}, H_4, \underline{\rho_0}, \overline{\rho_0}, T, \overline{\theta_0})>0$ such that
\begin{equation*}
\text{for almost all }(t,x)\in [0,T]\times\T,\quad \rho(t,x)\geq \underline{\rho}.
\end{equation*}
\begin{proof}
Simply recall that
\begin{equation}
\text{for almost all }(t,x)\in [0,T]\times \T,\quad 1/\rho(t,x) \leq H_3\left(1+\int_0^t \Vert\theta\Vert_\infty\right),\tag{\ref{eq:46}}
\end{equation}
Then, thanks to Proposition \ref{prop:8}, we get
\begin{equation*}
\text{for almost all }(t,x)\in [0,T]\times \T,\quad \rho(t,x)\geq \frac{1}{H_3(1+T\overline{\theta})}=:\underline{\rho}.
\end{equation*}
\end{proof}
\end{prop}

Finally, it is easy to find the missing bounds to show \eqref{eq:14}:
\begin{prop}\label{prop:10} Assume that \eqref{eq:51} is verified for some $H_4>0$. Then there exists some \newline$H_5 = H_5({\cal E}, H_4, \underline{\rho_0}, \overline{\rho_0}, T)>0$ and $H_6 = H_6({\cal E}, H_4, {\cal M}, \underline{\rho_0}, \overline{\rho_0}, T)>0$ such that
\begin{equation}\label{eq:62}
\kappa\int_0^T\int_0^1 (\partial_x\theta)^2 \leq H_1 \overline{\theta}^2,
\end{equation}
\begin{equation}\label{eq:63}
\sup_{0\leq t \leq T}\int_0^1 (\partial_x u)^2 \leq  H_5,
\end{equation}
\begin{equation}\label{eq:64}
\int_0^T \Vert \partial_x u\Vert_\infty^2 \leq (2T+1)H_5,
\end{equation}
\begin{equation}\label{eq:65}
\sup_{[0,T]\times\T} u^2 \leq H_6,
\end{equation}
\begin{equation*}
\int_0^T\int_0^1 (\partial_t u)^2\leq \frac{2H_4}{\underline{\rho}^2} + 2 T H_5 H_6,
\end{equation*}
where $\overline{\theta}$ is defined in Proposition \ref{prop:8} as function in particular of $H_4$.
\end{prop}
\begin{proof}
The inequality \eqref{eq:62} is a straightforward consequence of \eqref{eq:31} and \eqref{eq:54}. Moreover, we have
\begin{equation*}
(\partial_x u)^2 \leq \frac{2\sigma^2}{\mu^2} + \frac{2p^2}{\mu^2},
\end{equation*}
hence \eqref{eq:63} and \eqref{eq:64}, denoting
\begin{equation*}
H_5 := \frac{2H_4}{\mu^2} + \frac{2R^2\overline{\rho}^2\overline{\theta}^2}{\mu^2}
\end{equation*}
and using \eqref{eq:51}. Moreover, 
\begin{equation*}
\text{for all }x,y\in\T,\quad u^2(x) = u^2(y) + 2\int_y^x u(z)\partial_z u(z) dz
\end{equation*}
then multiplying this last equality by $\rho(y)$ and integrating on the torus in $y$, we get by Hölder's inequality and \eqref{eq:25}, \eqref{eq:28}, \eqref{eq:63},
\begin{align}
\text{for all }x\in\T,\quad {\cal M}u^2(x) &\leq \int_0^1 \rho u^2 + 2\left(\int_0^1 u^2\right)^{1/2}\left(\int_0^1(\partial_x u)^2\right)^{1/2} \nonumber\\&\leq 2{\cal E} + \frac{2\sqrt{{\cal E}}}{\sqrt{\underline{\rho}}}\sqrt{H_5}, \nonumber\end{align}
hence \eqref{eq:65} with
\begin{equation*}
H_6 := \frac{2{\cal E}}{{\cal M}} + \frac{2\sqrt{{\cal E}}}{\sqrt{\underline{\rho}}{\cal M}}\sqrt{H_5}.
\end{equation*}
Finally, from \eqref{eq:2} and \eqref{eq:8},
\begin{equation}\label{eq:66}
\partial_t u = \partial_t u + u\partial_x u - u\partial_x u = \frac{\partial_x \sigma}{\rho} - u\partial_x u.
\end{equation}
Thus passing \eqref{eq:66} to the square then using Young's inequality and integrating on $[0,T]\times\T$, we get from \eqref{eq:51}, \eqref{eq:63}, \eqref{eq:65},
\begin{align}
\int_0^T\int_0^1 (\partial_t u)^2&\leq \frac{2}{\underline{\rho}^2}\int_0^T\int_0^1 (\partial_x\sigma)^2 + 2\left(\sup_{[0,T]\times \T}\vert u\vert^2\right)\int_0^T\int_0^1 (\partial_x u)^2 \nonumber\\&\leq \frac{2H_4}{\underline{\rho}^2} + 2 T H_6 H_5. \nonumber\end{align}
\end{proof}

\subsection{\texorpdfstring{Bounds on $\sigma$ when $\kappa$ is far from zero}{Bounds on sigma when kappa is far from zero}}
In this subsection, we fix some $\overline{\kappa}\geq\underline{\kappa}>0$, and we assume that $\overline{\kappa}\geq\kappa\geq \underline{\kappa}$. Here we do not assume \eqref{eq:51} anymore. We get the
\begin{prop}\label{prop:11}
There exists some $H_7 = H_7({\cal E}, H_1, \underline{\kappa}, {\cal M}, \underline{\rho_0}, \overline{\rho_0}, T)>0$ such that
\begin{equation}\label{eq:67}
\int_0^T\Vert \theta\Vert_{\infty} \leq H_7.
\end{equation}
\end{prop}
\begin{proof}
Let us start by the equality, for almost all $x,y\in \T$, $t\in [0,T]$,
\begin{equation}\label{eq:68}
\sqrt{\theta}(t,x) = \sqrt{\theta}(t, y) + \int_y^x \frac{\partial_z \theta}{2\sqrt{\theta}}(t, z) dz.
\end{equation}
Remark that from \eqref{eq:46} we have, for almost all $t\in[0,T]$,
\begin{equation*}
\left\vert\frac{\partial_z \theta}{\sqrt{\theta}}\right\vert = \left\vert\frac{\partial_z \theta}{\theta}\right\vert\sqrt{\rho\theta}\frac{1}{\sqrt{\rho}}\leq \left\vert\frac{\partial_z \theta}{\theta}\right\vert\sqrt{\rho\theta} \sqrt{H_3}\sqrt{1+\int_0^t\Vert\theta\Vert_\infty}.
\end{equation*}
Thus by Hölder's inequality and \eqref{eq:27}
\begin{equation}\label{eq:69}
\left\vert\int_y^x \frac{\partial_z \theta}{2\sqrt{\theta}}\right\vert \leq \frac{1}{2}\sqrt{\frac{{\cal E}H_3}{\cv}}\sqrt{\int_0^1 \frac{(\partial_z \theta)^2}{\theta^2}} \sqrt{1+\int_0^t\Vert\theta\Vert_\infty}.
\end{equation}
Then, taking the square in \eqref{eq:68}, using Young's inequality then \eqref{eq:69}, we obtain
\begin{align}
\theta(t,x) \leq 2\theta(t,y) + \frac{1}{2}\frac{{\cal E}H_3}{\cv}\int_0^1\frac{(\partial_z\theta)^2}{\theta^2}\left(1+\int_0^t\Vert\theta\Vert_\infty\right). \nonumber
\end{align}
Multiplying this last inequality by $\rho(t,y)$, then integrating on the torus in $y$, we get from \eqref{eq:25} and \eqref{eq:27}
\begin{equation}\label{eq:70}
{\cal M}\theta(t,x) \leq \frac{2{\cal E}}{\cv} + \frac{1}{2}\frac{{\cal M}{\cal E}H_3}{\cv}\int_0^1 \frac{(\partial_z \theta)^2}{\theta^2}\left(1 + \int_0^t\Vert \theta\Vert_\infty\right).
\end{equation}
Finally, dividing \eqref{eq:70} by $\cal M$ and passing to the sup on space at the left hand side, we obtain
\begin{equation*}
\Vert \theta\Vert_\infty(t) \leq \frac{2{\cal E}}{\cv {\cal M}} + \frac{1}{2}\frac{{\cal E}H_3}{\cv}\int_0^1 \frac{(\partial_z\theta)^2}{\theta^2}\left(1+\int_0^t\Vert\theta\Vert_\infty\right).
\end{equation*}
Using Grönwall's lemma then \eqref{eq:31}, we get
\begin{equation}\label{eq:71}
1 + \int_0^t\Vert \theta\Vert_\infty \leq  \left(1 + \frac{2{\cal E}}{\cv {\cal M}} T\right)\exp\left(\frac{{\cal E} H_3 H_1}{2\cv\underline{\kappa}}\right)
\end{equation}
hence the result.
\end{proof}

\begin{remark}
The bound \eqref{eq:71} explodes very quickly when $\underline{\kappa}$ tends towards zero. This clearly illustrates the fact that the estimates obtained so far in the case with thermal conductivity are not satisfactory when very small conductivity coefficients are considered. 
\end{remark}

We can now obtain the first Hoff energy estimates, as in \cite{Li19}.
\begin{prop}\label{prop:12}
There exists some $H_4^{\underline{\kappa}} = H_4^{\underline{\kappa}}(C_0, \underline{\kappa}, \overline{\kappa}, \underline{\rho_0}, \overline{\rho_0}, T, \overline{\theta_0})>0$ such that
\begin{equation*}
\sup_{0\leq t \leq T}\int_0^1 \sigma^2 (t) + \int_0^T\int_0^1 (\partial_x\sigma)^2 + \kappa\int_0^T\int_0^1 (\partial_x\theta)^2 \leq H_4^{\underline{\kappa}}.
\end{equation*}
\end{prop}
\begin{proof}
Multiplying \eqref{eq:9} by $\sigma$, we get
\begin{equation}\label{eq:72}
\partial_t \left(\frac{\sigma^2}{2}\right) + u\partial_x\left(\frac{\sigma^2}{2}\right) - \mu \sigma \partial_x\left(\frac{\partial_x\sigma}{\rho}\right) = -\gamma\sigma^2\partial_x u - (\gamma-1)\sigma\partial_x(\kappa\partial_x\theta).
\end{equation}
Then, integrating \eqref{eq:72} on the torus, we get after some integration by parts
\begin{equation}\label{eq:73}
\frac{1}{2}\frac{d}{dt}\int_0^1 \sigma^2 +\mu\int_0^1 \frac{(\partial_x\sigma)^2}{\rho} = -\left(\gamma-\frac{1}{2}\right)\int_0^1 \sigma^2\partial_x u +(\gamma-1)\kappa\int_0^1(\partial_x\sigma)(\partial_x\theta).
\end{equation}
From one hand, we get by some integration by parts and by Hölder's inequality, using \eqref{eq:28},
\begin{equation}\label{eq:74}
\left\vert\int_0^1 \sigma^2(\partial_x u)\right\vert = 2\left\vert\int_0^1 u\sigma\partial_x \sigma\right\vert\leq 2\int_0^1 \vert \sqrt{\rho} u \vert \vert \sigma\vert \left\vert\frac{\partial_x\sigma}{\sqrt{\rho}}\right\vert\leq 2\sqrt{2}\sqrt{{\cal E}}\Vert\sigma\Vert_\infty\sqrt{\int_0^1 \frac{(\partial_x\sigma)^2}{\rho}}.
\end{equation}
Using now the inequality
\begin{equation*}
\Vert \sigma\Vert_\infty \leq \Vert \sigma\Vert_2 + \sqrt{2}\sqrt{\Vert \sigma\Vert_2}\sqrt{\Vert\partial_x\sigma\Vert_2},
\end{equation*}
obtained from \eqref{eq:53} and the subadditivity of $\sqrt{\cdot}$, we get from \eqref{eq:74} and Young's inequality
\begin{align}
\left\vert \int_0^1 \sigma^2\partial_x u\right\vert &\leq 2\sqrt{2}\sqrt{{\cal E}}\sqrt{\int_0^1 \sigma^2}\sqrt{\int_0^1\frac{(\partial_x\sigma)^2}{\rho}} + 4\sqrt{{\cal E}}\overline{\rho}^{1/4}\left(\int_0^1\sigma^2\right)^{1/4}\left(\int_0^1\frac{\partial_x\sigma^2}{\rho}\right)^{3/4} \nonumber\\&\leq \frac{16{\cal E}(\gamma-1/2)}{\mu}\int_0^1 \sigma^2 + \frac{\mu}{8(\gamma-1/2)}\int_0^1\frac{(\partial_x\sigma)^2}{\rho}\nonumber \\&+ 24^3{\cal E}^{2}\overline{\rho}\frac{(\gamma-1/2)^3}{\mu^3} \int_0^1\sigma^2 + \frac{\mu}{8(\gamma-1/2)}\int_0^1\frac{(\partial_x\sigma)^2}{\rho}.\label{eq:75}
\end{align}

From one other hand, 
\begin{align}
(\gamma-1)\kappa\left\vert\int_0^1 (\partial_x\sigma)(\partial_x\theta)\right\vert &\leq (\gamma-1)\kappa\sqrt{\overline{\rho}}\int_0^1\left\vert\frac{\partial_x\sigma}{\sqrt{\rho}}\right\vert\vert\partial_x\theta\vert \nonumber\\&\leq \frac{(\gamma-1)^2\kappa^2\overline{\rho}}{\mu}\int_0^1 (\partial_x\theta)^2 + \frac{\mu}{4}\int_0^1\frac{(\partial_x\sigma)^2}{\rho}\label{eq:76}
\end{align}
by Young's inequality.
From \eqref{eq:73}, \eqref{eq:75} and \eqref{eq:76} we get
\begin{align}
\frac{1}{2}\frac{d}{dt}\int_0^1\sigma^2 + \frac{\mu}{2}\int_0^1\frac{(\partial_x\sigma)^2}{\rho} &\leq \frac{16(\gamma-1/2)^2{\cal E}}{\mu}\left(1+ \frac{12^3}{2}{\cal E}\overline{\rho}\frac{(\gamma-1/2)^2}{\mu^2}\right)\int_0^1\sigma^2\nonumber\\&+ \frac{(\gamma-1)^2\kappa\overline{\rho}}{\mu}\kappa\int_0^1(\partial_x\theta)^2,\label{eq:77}
\end{align}
since $\gamma>1$. Besides, multiplying now \eqref{eq:5} by $\theta$, we obtain
\begin{equation}\label{eq:78}
\partial_t\left(\frac{\rho\cv\theta^2}{2}\right) + \partial_x\left(\frac{\rho\cv\theta^2 u}{2}\right) = \theta\sigma(\partial_x u) + \theta\partial_x(\kappa\partial_x\theta). 
\end{equation}
Then, integrating \eqref{eq:78} on the torus, we get
\begin{equation}\label{eq:79}
\frac{d}{dt}\int_0^1 \frac{\rho\cv\theta^2}{2} + \int_0^1 \kappa(\partial_x\theta)^2 = \int_0^1\sigma(\partial_x u)\theta.
\end{equation}
We obtain by Young's inequality
\begin{align}
\left\vert \int_0^1 \sigma(\partial_x u)\theta\right\vert &\leq \frac{1}{\mu}\left\vert\int_0^1 \sigma^2\theta\right\vert + \frac{R}{\mu}\left\vert \int_0^1 \sigma\rho\theta^2\right\vert \nonumber\\&\leq \frac{\Vert \theta\Vert_\infty}{\mu}\int_0^1\sigma^2 + \frac{R\sqrt{\overline{\rho}}\Vert\theta\Vert_\infty}{\mu}\left\vert\int_0^1 \sigma \sqrt{\rho} \theta\right\vert \nonumber\\&\leq \left(1 + \frac{R\sqrt{\overline{\rho}}}{2}\right)\frac{\Vert \theta\Vert_\infty}{\mu}\int_0^1\sigma^2 + \frac{R\sqrt{\overline{\rho}}\Vert\theta\Vert_\infty}{2\mu}\int_0^1 \rho\theta^2.\label{eq:80}
\end{align}
Finally, \eqref{eq:79} and \eqref{eq:80} give
\begin{equation}\label{eq:81}
\frac{d}{dt}\int_0^1\frac{\rho\cv\theta^2}{2} + \int_0^1\kappa(\partial_x\theta)^2 \leq \left(1 + \frac{R\sqrt{\overline{\rho}}}{2}\right)\frac{\Vert \theta\Vert_\infty}{\mu}\int_0^1\sigma^2 + \frac{R\sqrt{\overline{\rho}}\Vert\theta\Vert_\infty}{\cv\mu}\int_0^1 \frac{\rho\cv\theta^2}{2}.
\end{equation}
Summing \eqref{eq:77} with $2(\gamma-1)^2\kappa\overline{\rho}/\mu\times$\eqref{eq:81}, we finally obtain
\begin{align}
&\frac{1}{2}\frac{d}{dt}\int_0^1\sigma^2 + \frac{(\gamma-1)^2\kappa\overline{\rho}}{\mu}\frac{d}{dt}\int_0^1\rho\cv\theta^2+ \frac{\mu}{2}\int_0^1\frac{(\partial_x\sigma)^2}{\rho} + \frac{(\gamma-1)^2\kappa\overline{\rho}}{\mu}\kappa\int_0^1(\partial_x\theta)^2 \nonumber\\&\leq \left[\frac{16(\gamma-1/2)^2{\cal E}}{\mu}\left(1+ \frac{12^3}{2}{\cal E}\overline{\rho}\frac{(\gamma-1/2)^2}{\mu^2}\right) + \frac{2(\gamma-1)^2\kappa\overline{\rho}}{\mu}\left(1 + \frac{R\sqrt{\overline{\rho}}}{2}\right)\frac{\Vert \theta\Vert_\infty}{\mu} \right]\int_0^1\sigma^2 \nonumber\\&+\frac{R\sqrt{\overline{\rho}}}{\cv\mu}\Vert\theta\Vert_\infty\frac{(\gamma-1)^2\kappa\overline{\rho}}{\mu}\int_0^1\rho\cv\theta^2. \nonumber\end{align}
Hence by Grönwall's lemma, for all $t\in [0,T]$, using \eqref{eq:67},
\begin{equation}\label{eq:82}
\int_0^1\sigma^2(t) + \int_0^1\rho\theta^2(t) + \int_0^t\int_0^1(\partial_x\sigma)^2 + \kappa\int_0^t\int_0^1(\partial_x\theta)^2 \leq B_1\left(\int_0^1\sigma_0^2 + \int_0^1\rho_0\theta_0^2\right)\exp(B_2 t + B_3 H_7),
\end{equation}
where
\begin{equation*}
B_1 = \frac{\max(1, 2(\gamma-1)^2\overline{\kappa}\overline{\rho}\cv/\mu)}{\min(1, 2(\gamma-1)^2\underline{\kappa}\overline{\rho}\cv/\mu), 2(\gamma-1)^2\underline{\kappa}\overline{\rho}/\mu, \mu/\overline{\rho})},
\end{equation*}
\begin{equation*}
B_2 = \frac{32(\gamma-1/2)^2{\cal E}}{\mu}\left(1+ \frac{12^3}{2}{\cal E}\overline{\rho}\frac{(\gamma-1/2)^2}{\mu^2}\right),
\end{equation*}
\begin{equation*}
B_3 = \max\left(\frac{4(\gamma-1)^2\overline{\kappa}\overline{\rho}}{\mu^2}\left(1+\frac{R\sqrt{\overline{\rho}}}{2}\right), \frac{R\sqrt{\overline{\rho}}}{\cv\mu}\right).
\end{equation*}
Thus, by \eqref{eq:16},\eqref{eq:17} and \eqref{eq:18}, passing to the sup in \eqref{eq:82},
\begin{align}
\sup_{0\leq t\leq T}\int_0^1\sigma^2 &+ \sup_{0\leq t\leq T}\int_0^1\rho\theta^2 + \int_0^T\int_0^1(\partial_x\sigma)^2 + \kappa\int_0^T\int_0^1(\partial_x\theta)^2 \nonumber\\&\leq B_1(2\mu^2 C_0 + 2R^2\overline{\rho_0}^2\overline{\theta_0}^2 + \overline{\rho_0}\overline{\theta_0}^2)\exp(B_2 T + B_3 H_7) \nonumber\end{align}
which proves Proposition \ref{prop:12}.
\end{proof}
\subsection{\texorpdfstring{Bounds on $\sigma$ when $\kappa$ is close to zero}{Bounds on sigma when kappa is close to zero}}
The aim of this subsection is to show
\begin{prop}\label{prop:13}
There exists some $\kappa_0 = \kappa_0({\cal E}, H_1, \overline{\rho_0}, T, \overline{\theta_0})>0$ and some \newline$H_4^0 = H_4^0({\cal E}, H_1, \overline{\rho_0}, T, \overline{\theta_0})>0$ such that, if $0\leq\kappa<\kappa_0$, then
\begin{equation*}
\sup_{0\leq t\leq T}\int_0^1 \sigma^2(t) + \int_0^T\int_0^1 (\partial_x \sigma)^2 \leq H_4^0.
\end{equation*}
\end{prop}
The idea is the following: defining some functional $A(t)$ similar to
\begin{equation*}
\int_0^1\sigma^2(t) + \int_0^t\int_0^1 (\partial_x\sigma)^2,
\end{equation*}
we will show the existence of a constant $D>0$ and some positive non-decreasing function $\Phi:\R\mapsto \R$ not depending on $\kappa$ such that
\begin{equation*}
\frac{d}{dt}A \leq DA + \kappa^2\left(\int_0^1\frac{(\partial_x\theta)^2}{\theta^2}\right)\Phi(A).
\end{equation*}
We can then obtain some uniform bound on $A$ using the
\begin{lemma}\label{lem:17}
Assume that some family of non-negative smooth functions $\tau_\kappa:[0,T]\rightarrow \R$ satisfies the inequality
\begin{equation}\label{eq:83}
\text{for all }\kappa>0, \text{ for almost all }t\in[0,T],\quad
\frac{d}{dt}\tau_\kappa(t) \leq D \tau_\kappa(t) + \kappa\delta(t)\Phi(\tau_\kappa(t))
\end{equation} 
where $D\in\R$, $\delta\in L^1(0,T)$, and $\Phi:\R_+\rightarrow\R$ is some measurable function. Then there exists some $\kappa_0>0$ and $\overline{\tau}>0$ depending only on $D$, $\Phi$, $\Vert\delta\Vert_{L^1}$, $\tau_\kappa(0) =:\tau_0$ and $T$ such that
\begin{equation*}
\text{for all }t\in[0,T],\quad\text{for all }0\leq\kappa <\kappa_0,\quad \tau_\kappa(t)\leq \overline{\tau}.
\end{equation*}
\end{lemma}
This lemma is inspired from some work by Desjardins about weak solutions of barotropic Navier-Stokes equations in dimension $2$ and $3$ for short times (see \cite{Des}).
\begin{proof}
From \eqref{eq:83} we get, for almost all $t\in[0,T]$,
\begin{equation*}
\frac{d}{dt}\tau_{\kappa}(t)\leq \vert D\vert\tau_\kappa(t) + \kappa\vert\delta(t)\vert \left(\sup_{s\leq \tau_\kappa(t)}\vert\Phi(s)\vert+1\right).
\end{equation*}
Thus we can assume without any loss of generality that $\Phi$ is positive and non-decreasing, $\delta>0$ and $D\geq 0$.
Let's start by looking at the case where $D=0$. Let $\Psi$ be a primitive of $1/\Phi$. Then, dividing \eqref{eq:83} by $\Phi(\tau_\kappa)$, we get for almost all $t\in [0,T]$,
\begin{equation*}
\frac{d}{dt}(\Psi\circ\tau_\kappa)(t)\leq \kappa\delta(t),
\end{equation*}
thus
\begin{equation}\label{eq:84}
\Psi(\tau_\kappa(t))\leq \Psi(\tau_0) + \kappa\int_0^t\delta.
\end{equation}
As $\Phi>0$, $\Psi$ is increasing. Thus there exists some $\kappa_0>0$ depending only on $\Psi(\tau_0)$ and $\Vert\delta\Vert_{L^1}$ such that
\begin{equation}\label{eq:85}
\Psi(\tau_0) + \kappa_0\int_0^T\delta <\sup \Psi.
\end{equation}
Assume $\kappa<\kappa_0$. We deduce from \eqref{eq:84} and \eqref{eq:85} that there exists some $\overline{\tau}>0$ depending only on $\Psi(\tau_0)$ and $\Vert \delta\Vert_{L^1}$ such that
\begin{equation*}
\text{for all }t\in [0,T],\quad \tau_\kappa(t) \leq \overline{\tau}.
\end{equation*}

\begin{minipage}{0.7\textwidth}
\begin{center}
\begin{tikzpicture}
  \begin{axis}[
    name = monaxe,
    axis lines = middle,
    xlabel = {$y$},
    ylabel = {$\Psi(y)$},
    domain = 0:5,
    samples = 100,
    ymin = 0, ymax = 1.2,
    xmin = 0.05, xmax = 5,
    xtick = \empty,
    ytick = \empty,
    thick,
    every axis x label/.style = {at={(current axis.right of origin)}, anchor=north},
    every axis y label/.style = {at={(current axis.above origin)}, anchor=east},
    axis line style={->},
    ]
    \addplot[thick] {sqrt(1-exp(-x))};
    \addplot[thin, dashed] {1};
    \addplot[thin] coordinates {(0,0.95) ({-ln(1-0.95^2)},0.95)};
    \addplot[thin] coordinates {(0,{sqrt(1-exp(-0.5))}) (0.5,{sqrt(1-exp(-0.5))})};
    \addplot[thin] coordinates {(0.5, 0) (0.5, {sqrt(1-exp(-0.5))})};
    \addplot[thin] coordinates {({-ln(1-0.95^2)}, 0) ({-ln(1-0.95^2)}, 0.95)};
  \end{axis}
\node at (0.6,0) [anchor=north] {$\tau_0$};
\node at ({-1.35*ln(1-0.95^2)},0.05) [anchor=north] {$\overline{\tau}$};
\node at (1,5) [anchor=west] {$\sup \Psi$};
\node at (-2.85,4.5) [anchor=west] {$\kappa_0\int_0^T\delta + \Psi(\tau_0)$};
\node at (-1.2, 3) [anchor=west] {$\Psi(\tau_0)$};
\end{tikzpicture}
\end{center}
\end{minipage}%
\hfill
\begin{minipage}{0.25\textwidth}
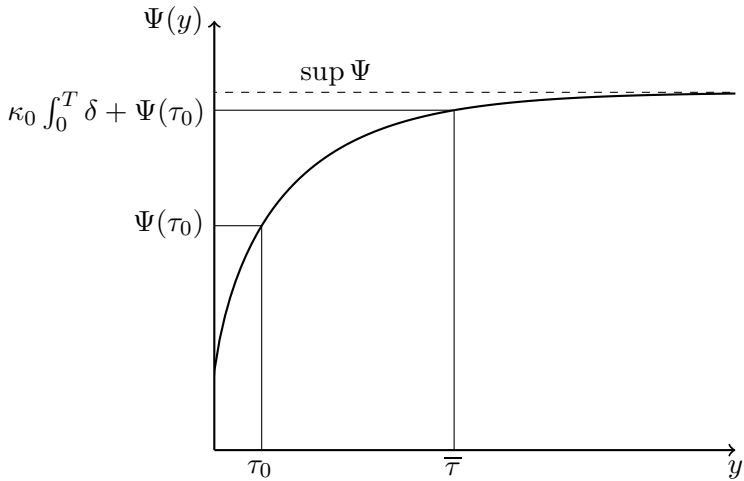
\small\textbf{Figure 2:} Illustration of the proof of Lemma \ref{lem:17}.
\end{minipage}
\begin{minipage}{0.1\textwidth}
~
\end{minipage}

For the general case $D>0$, remark that we get from \eqref{eq:83}, for all $t\in [0,T]$,
\begin{equation*}
\frac{d}{dt}(e^{-Dt}\tau_\kappa(t))\leq \kappa\delta(t)\Phi(\tau_\kappa(t)) \leq \kappa\delta(t)\Phi(e^{DT}e^{-Dt}\tau_\kappa(t))
\end{equation*}
because $\Phi$ is non-decreasing. We can then apply the result obtained in the case $D=0$, with $\widetilde{\tau_\kappa}(t) = e^{-Dt}\tau_\kappa(t)$ and $\widetilde{\Phi} = \Phi(e^{DT}\cdot)$. We then get the existence of $\kappa_0$, $\overline{\widetilde{\tau}}$ depending only on $D$, $\Phi$, $\Vert \delta\Vert_{L^1}$, $\tau_0$ and $T$ such that
\begin{equation*}
\text{for all }0\leq\kappa < \kappa_0,\quad\text{for all }t\in [0,T],\quad e^{-Dt}\tau_\kappa(t) \leq \overline{\widetilde{\tau}},
\end{equation*}
hence
\begin{equation*}
\text{for all }0\leq\kappa < \kappa_0,\quad\text{for all }t\in [0,T],\quad \tau_\kappa(t) \leq \overline{\widetilde{\tau}}e^{DT} =:\overline{\tau}.
\end{equation*}
\end{proof}
\begin{prop}\label{prop:14}
Let us denote, for all $t\in [0,T]$,
\begin{equation*}
A_1(t) : = \int_0^1\sigma^2(t) + \mu\int_0^t\int_0^1 \frac{(\partial_x\sigma)^2}{\rho}
\end{equation*}
and
\begin{equation}\label{eq:86}
A_2(t) : = A_1(t) + \int_0^t A_1.
\end{equation}
There exists some $D = D({\cal E}, \overline{\rho_0}, T)>0$ and some function $\Phi:\R\mapsto \R$ depending only on ${\cal E},\overline{\rho_0},T$ and $\overline{\theta_0}$ such that
\begin{equation*}
\frac{d}{dt}A_2 \leq DA_2 + \kappa^2\left(\int_0^1 \frac{(\partial_x\theta)^2}{\theta^2}\right)\Phi(A_2).
\end{equation*}
\end{prop}

\begin{proof}
In the proof of Proposition \ref{prop:12}, we showed that
\begin{align}
\frac{1}{2}\frac{d}{dt}\int_0^1\sigma^2 + \frac{\mu}{2}\int_0^1\frac{(\partial_x\sigma)^2}{\rho} &\leq \frac{16(\gamma-1/2)^2{\cal E}}{\mu}\left(1+ \frac{12^3}{2}{\cal E}\overline{\rho}\frac{(\gamma-1/2)^2}{\mu^2}\right)\int_0^1\sigma^2\nonumber\\&+ \frac{(\gamma-1)^2\kappa\overline{\rho}}{\mu}\kappa\int_0^1(\partial_x\theta)^2.\tag{\ref{eq:77}} \nonumber\end{align}
Hence
\begin{align}\label{eq:87}
\frac{d}{dt}A_1 &\leq \frac{32(\gamma-1/2)^2{\cal E}}{\mu}\left(1+ \frac{12^3}{2}{\cal E}\overline{\rho}\frac{(\gamma-1/2)^2}{\mu^2}\right)A_1+ \frac{2(\gamma-1)^2\overline{\rho}}{\mu}\sup_{[0,t]\times\T}\theta^2\kappa^2\int_0^1\frac{(\partial_x\theta)^2}{\theta^2}.
\end{align}
Besides, passing to the limit $k\rightarrow +\infty$ in \eqref{eq:59} and taking the supremum in $t$, we obtain
\begin{equation}\label{eq:88}
\sup_{[0,t]\times\T}\theta \leq \overline{\theta_0}\exp\left(\int_0^t \frac{R}{\mu\cv}\Vert\sigma\Vert_\infty\right) + \frac{1}{\mu\overline{\rho}\cv}\int_0^t \Vert\sigma\Vert_\infty^2(s)\exp\left(\int_s^t \frac{R}{\mu\cv}\Vert\sigma\Vert_\infty\right).
\end{equation}
Note that by the Gagliardo-Nirenberg inequality \eqref{eq:53} and Young's inequality, wet get for $t\in [0,T]$,
\begin{equation}\label{eq:89}
\int_0^t\Vert\sigma\Vert_\infty^2 \leq 2\int_0^t A_1 + \frac{\overline{\rho}}{\mu} A_1(t)\leq \max(2,\overline{\rho}/\mu)A_2(t).
\end{equation}
We deduce from \eqref{eq:88}, \eqref{eq:89} that for almost all $t\in [0,T]$,
\begin{equation*}
\sup_{[0,t]\times\T}\theta \leq \left(\overline{\theta_0} + \frac{1}{\mu\overline{\rho}\cv}\max(2,\overline{\rho}/\mu)A_2(t)\right)\exp\left(\frac{R\sqrt{\max(2,\overline{\rho}/\mu)}\sqrt{T}}{\mu\cv}\sqrt{A_2(t)}\right).
\end{equation*}
Finally, \eqref{eq:86}, \eqref{eq:87}, and $A_1\leq A_2$ give for almost all $t\in[0,T]$,
\begin{align}
\frac{d}{dt}A_2(t) &\leq \left[\frac{32(\gamma-1/2)^2{\cal E}}{\mu}\left(1+ \frac{12^3}{2}{\cal E}\overline{\rho}\frac{(\gamma-1/2)^2}{\mu^2}\right) + 1\right]A_2(t) \nonumber\\&+ \frac{2(\gamma-1)^2\overline{\rho}}{\mu}\left(\overline{\theta_0} + \frac{1}{\mu\overline{\rho}\cv}\max(2,\overline{\rho}/\mu)A_2(t)\right)^2\exp\left(\frac{2R\sqrt{\max(2,\overline{\rho}/\mu)}\sqrt{T}}{\mu\cv}\sqrt{A_2(t)}\right)\kappa^2\int_0^1\frac{(\partial_x\theta)^2}{\theta^2}. \nonumber\end{align}
Hence the result, denoting
\begin{equation*}
D : = \frac{32(\gamma-1/2)^2{\cal E}}{\mu}\left(1+ \frac{12^3}{2}{\cal E}\overline{\rho}\frac{(\gamma-1/2)^2}{\mu^2}\right) + 1,
\end{equation*}
and for all $y\in [0,+\infty[$,
\begin{equation*}
\Phi(y) = \frac{2(\gamma-1)^2\overline{\rho}}{\mu}\left(\overline{\theta_0} + \frac{1}{\mu\overline{\rho}\cv}\max(2,\overline{\rho}/\mu)y\right)^2\exp\left(\frac{2R\sqrt{\max(2,\overline{\rho}/\mu)}\sqrt{T}}{\mu\cv}\sqrt{y}\right).
\end{equation*}
\end{proof}
The proof of Proposition \ref{prop:13} is then obtained by combining Proposition \ref{prop:14} and Lemma \ref{lem:17}.
\begin{remark}
The proof is, of course, much easier in the special case $\kappa = 0$ (see \cite{Li20,BrBuGJLa24}).
\end{remark}

Combining Propositions \ref{prop:1}, \ref{prop:2}, \ref{prop:4}, \ref{prop:5}, \ref{prop:7}, \ref{prop:8}, \ref{prop:9}, \ref{prop:10} and \ref{prop:12}, we obtain the bounds \eqref{eq:12}--\eqref{eq:14} for Theorem \ref{thm:1}. Using moreover Proposition \ref{prop:14}, we get the bounds \eqref{eq:12}--\eqref{eq:14} for Theorem \ref{thm:2}. In the last sub-section, we prove \eqref{eq:15}.

\subsection{Second Hoff energy}
\begin{prop}
\label{prop:15}There exists some $H_8 = H_8(C_0, \overline{\kappa}, \underline{\rho_0}, \overline{\rho_0}, T, \underline{\theta_0}, \overline{\theta_0})>0$ such that
\begin{equation}\label{eq:90}
\sup_{0\leq t\leq T} w \int_0^1 \kappa(\partial_x\theta)^2(t) + \int_0^T w \int_0^1 [\partial_x(\kappa\partial_x\theta)]^2\leq H_8.
\end{equation}
\end{prop}
\begin{proof}
Dividing \eqref{eq:5} by $\rho$ then applying $\partial_x$, we get
\begin{equation}\label{eq:91}
\cv\partial_x D_t \theta = \partial_x\left(\frac{\sigma\partial_x u}{\rho}\right) + \partial_x\left(\frac{\partial_x(\kappa\partial_x\theta)}{\rho}\right).
\end{equation}
Using the formula 
\begin{equation*}
D_t (\partial_x \theta) = \partial_x D_t \theta - (\partial_x u)\partial_x \theta,
\end{equation*}
then multiplying \eqref{eq:91} by $\kappa\partial_x\theta$, we obtain 
\begin{equation}\label{eq:92}
\frac{\cv\kappa}{2} D_t (\partial_x\theta)^2 + \kappa\cv(\partial_x u)(\partial_x\theta)^2 = \left[\partial_x\left(\frac{\sigma\partial_x u}{\rho}\right)\right]\kappa\partial_x\theta + \left[\partial_x\left(\frac{\partial_x(\kappa\partial_x\theta)}{\rho}\right)\right]\kappa\partial_x\theta.
\end{equation}
Then, integrating \eqref{eq:92} over the torus, we get, using integration by parts,
\begin{equation*}
\frac{\cv}{2}\frac{d}{dt}\int_0^1\kappa(\partial_x\theta)^2 + \int_0^1 \frac{[\partial_x(\kappa\partial_x\theta)]^2}{\rho} = -\frac{\cv}{2}\int_0^1(\partial_x u)\kappa(\partial_x\theta)^2 - \int_0^1 \frac{\sigma\partial_x u}{\rho}\partial_x (\kappa\partial_x\theta).
\end{equation*}
Hence, by Young's inequality,
\begin{align}
\cv\frac{d}{dt}\int_0^1\kappa(\partial_x\theta)^2 + \int_0^1 \frac{[\partial_x(\kappa\partial_x\theta)]^2}{\rho} &\leq \cv\int_0^1 \vert\partial_x u\vert \kappa(\partial_x\theta)^2 + \int_0^1 \frac{(\sigma\partial_x u)^2}{\rho} \nonumber\\&\leq \cv \Vert\partial_x u\Vert_\infty \int_0^1 \kappa(\partial_x\theta)^2 + \frac{\Vert\sigma\Vert_\infty^2}{\underline{\rho}}\int_0^1 (\partial_x u)^2. \label{eq:93}\end{align}
Multiplying by $w(t)=\min(1,t)$ this last equation, we get since $w\leq 1$,
\begin{align}
\cv\frac{d}{dt} \left(w \int_0^1\kappa(\partial_x\theta)^2\right) &+ w\int_0^1 \frac{[\partial_x(\kappa\partial_x\theta)]^2}{\rho} \nonumber\\&\leq \cv\int_0^1\kappa(\partial_x\theta)^2 + \cv \Vert\partial_x u\Vert_\infty w\int_0^1\kappa(\partial_x\theta)^2 + \frac{\Vert\sigma\Vert_\infty^2}{\underline{\rho}}\int_0^1 (\partial_x u)^2. \nonumber\end{align}
Finally, by Grönwall's lemma, for all $t\in [0,T]$,
\begin{align}
\cv w \int_0^1 \kappa(\partial_x\theta)^2(t) &+ \int_0^t w \int_0^1 \frac{[\partial_x(\kappa\partial_x\theta)]^2}{\rho} \nonumber\\&\leq \exp\left(\int_0^t \Vert\partial_x u\Vert_\infty\right)\left(\cv\int_0^t\int_0^1\kappa(\partial_x\theta)^2 + \frac{1}{\underline{\rho}}\int_0^t \Vert\sigma\Vert_\infty^2 \sup_{0\leq t\leq T}\int_0^1(\partial_x u)^2(t) \right). \nonumber\end{align}
Hence, using \eqref{eq:14},
\begin{equation*}
w \int_0^1 \kappa(\partial_x\theta)^2(t) + \int_0^t w \int_0^1 [\partial_x(\kappa\partial_x\theta)]^2\leq \max(1/\cv,\overline{\rho})\exp(\sqrt{T}\sqrt{C_1})\left(\cv C_1 + \frac{1}{\underline{\rho}}C_1^2\right).
\end{equation*}
\end{proof}
\begin{prop}
\label{prop:16}
There exists some $H_9 = H_9(C_0, \overline{\kappa}, \underline{\rho_0}, \overline{\rho_0}, T, \underline{\theta_0}, \overline{\theta_0})>0$ such that
\begin{equation}\label{eq:94}
\sup_{0\leq t \leq T} w(t)\int_0^1(\partial_x\sigma)^2 + \int_0^t w \int_0^1 (\partial_t\sigma)^2\leq H_9.
\end{equation}
\end{prop}

\begin{proof}
Multiplying \eqref{eq:9} by $D_t\sigma$, we obtain
\begin{equation}\label{eq:95}
(D_t\sigma)^2 - \mu\left[\partial_x\left(\frac{\partial_x\sigma}{\rho}\right)\right]D_t\sigma = -\gamma\sigma (D_t \sigma) \partial_x u -(\gamma-1)[\partial_x(\kappa\partial_x\theta)]D_t\sigma.
\end{equation}
Let us remark that using some integration by parts and  the equality
\begin{equation*}
D_t \left(\frac{\partial_x \sigma}{\rho}\right) = \frac{\partial_x D_t \sigma}{\rho},
\end{equation*}
we get
\begin{equation*}
\int_0^1 \left[\partial_x\left(\frac{\partial_x\sigma}{\rho}\right)\right]D_t\sigma = -\int_0^1\frac{\partial_x\sigma}{\rho}\partial_x D_t\sigma = -\int_0^1 \rho \frac{\partial_x \sigma}{\rho} D_t\left(\frac{\partial_x \sigma}{\rho}\right) = -\frac{1}{2}\frac{d}{dt}\int_0^1 \frac{(\partial_x\sigma)^2}{\rho}.
\end{equation*}
Thus, integrating \eqref{eq:95} on the torus, we obtain by integration by parts and by Young's inequality
\begin{align}
\frac{\mu}{2}\frac{d}{dt}\int_0^1\frac{(\partial_x\sigma)^2}{\rho} + \int_0^1 (D_t\sigma)^2 &= -\gamma\int_0^1 \sigma D_t\sigma(\partial_x u) - (\gamma-1)\int_0^1 [\partial_x(\kappa\partial_x\theta)]D_t\sigma \nonumber\\&\leq \gamma^2\Vert\partial_x u\Vert_\infty^2\int_0^1\sigma^2 + \frac{1}{4}\int_0^1(D_t\sigma)^2 + (\gamma-1)^2\int_0^1[\partial_x(\kappa\partial_x\theta)]^2 + \frac{1}{4}\int_0^1 (D_t\sigma)^2, \nonumber\end{align}
hence
\begin{equation}\label{eq:96}
\mu\frac{d}{dt}\int_0^1\frac{(\partial_x\sigma)^2}{\rho} + \int_0^1(D_t\sigma)^2 \leq 2\gamma^2\Vert\partial_x u\Vert_\infty^2\int_0^1\sigma^2 + 2(\gamma-1)^2\int_0^1[\partial_x(\kappa\partial_x\theta)]^2.
\end{equation}
Multiplying \eqref{eq:96} by $w$ then integrating on $[0,t]$ for $t\in [0,T]$, we get
\begin{equation}\label{eq:97}
\begin{split}
\mu w(t)\int_0^1 \frac{(\partial_x\sigma)^2}{\rho}(t) &+ \int_0^t w\int_0^1 (D_t\sigma)^2 \\&\leq\mu \int_0^t\int_0^1 \frac{(\partial_x\sigma)^2}{\rho} + 2\gamma^2\int_0^t\Vert\partial_x u\Vert_\infty^2\int_0^1\sigma^2 + 2(\gamma-1)^2\int_0^t w \int_0^1 [\partial_x(\kappa\partial_x\theta)]^2.
\end{split}
\end{equation}
Remark that, as
\begin{equation}\label{eq:98}
\text{for all }a,b\in\R,\quad (a+b)^2 \geq \frac{a^2}{2}-b^2,
\end{equation}
we have
\begin{equation}\label{eq:99}
\int_0^t w\int_0^1 (D_t\sigma)^2 \geq \frac{1}{2}\int_0^t w\int_0^1 (\partial_t\sigma)^2 - \int_0^1 w \int_0^t (u\partial_x\sigma)^2.
\end{equation}
From \eqref{eq:97} and \eqref{eq:99}, we get, passing to the supremum in $t\in [0,T]$, then using \eqref{eq:14} and \eqref{eq:90},
\begin{equation*}
\mu\sup_{0\leq t \leq T} w(t)\int_0^1\frac{(\partial_x\sigma)^2}{\rho} + \frac{1}{2}\int_0^T w \int_0^1 (\partial_t\sigma)^2\leq \frac{\mu}{\underline{\rho}}C_1 + 2\gamma^2 C_1^2 + 2(\gamma-1)^2H_8 + \left(\sup_{[0,T]\times\T} \vert u\vert\right)^2 C_1.
\end{equation*}
Then, finally,
\begin{equation*}
\sup_{0\leq t \leq T} w(t)\int_0^1(\partial_x\sigma)^2 + \int_0^t w \int_0^1 (\partial_t\sigma)^2\leq \max(2,\overline{\rho}/\mu)\left[\left(\frac{1}{\underline{\rho}} + 2\gamma^2C_1 + C_1\right)C_1 + 2(\gamma-1)^2H_8\right].
\end{equation*}
\end{proof}
And of course, \eqref{eq:90} and \eqref{eq:94} give \eqref{eq:15}.

\ 

We can now use these bounds to show Theorem \ref{thm:3} and Theorem \ref{thm:main}, as will be done in the next section.
\section{Compactness results}\label{sect:2}
This section is divided as follows: the first sub-section proves strong compactness of the velocity and the Cauchy stress, valid for both $\kappa = 0$ and $\kappa > 0$. Then, the second sub-section is devoted to the proof of Theorem \ref{thm:3}, the stability result in the case $\kappa=0$. Finally, the third sub-section proves Theorem \ref{prop:4prime}, a convergence result as $\kappa\rightarrow 0$. Let fix $T>0$. 
\subsection{\texorpdfstring{Strong compactness of the velocity and the Cauchy stress (for $\kappa\geq 0$)}{Strong compactness of the velocity and the Cauchy stress (for kappa greater than zero)}}
\begin{prop}
\label{prop:17}
Let $(\kappa_n)_n\subset \R_+$ a bounded sequence. Let $(\rho_0^n,u_0^n,\theta_0^n)_n \subset L^2(\T)^3$ a sequence of triplet satisfying \eqref{eq:16}--\eqref{eq:18} with $\underline{\rho_0},\overline{\rho_0},\underline{\theta_0},\overline{\theta_0}, C_0$ not depending on $n$. Consider $(\rho^n,u^n,\theta^n)$ the global "à la Hoff" solution of \eqref{eq:1}--\eqref{eq:3} with initial conditions $(\rho_0^n,u_0^n,\theta_0^n)$ and with conductivity $\kappa_n$. Then, denoting $\sigma^n$ the Cauchy stress, there exists $u$ and $\sigma\in L^2([0,T]\times\T)$ such that, up to a subsequence,
\begin{equation}\label{eq:100}
u^n \tend{n}{+\infty} u \quad \text{in }L^2([0,T]\times \T),
\end{equation}
\begin{equation}\label{eq:101}
\sigma^n \tend{n}{+\infty} \sigma \quad \text{in }L^2([0,T]\times \T).
\end{equation}
\end{prop}
\begin{proof}
We know from Theorem \ref{thm:2} that $(\partial_x u^n)$ is bounded in $L^2([0,T]\times\T)$ and $(\partial_t u^n)$ is bounded in $L^2([0,T]\times \T)$, hence $(u^n)$ is bounded in $H^1([0,T]\times\T)$, then \eqref{eq:100} holds by the Rellich–Kondrachov theorem.

Beside, from Theorem \ref{thm:2}, $(\partial_x\sigma^n)$ is bounded in $L^2([0,T]\times\T)$. Thus there exists some $\sigma\in L^2(0,T,H^1(\T))$ such that
\begin{equation}\label{eq:102}
\sigma^n \tendf{n}{+\infty} \sigma \quad\text{in }L^2(0,T,H^1(\T)).
\end{equation} 
Recall that
\begin{equation}\label{eq:103}
\partial_t \sigma^n = - u^n\partial_x\sigma^n + \mu \partial_x\left(\frac{\partial_x\sigma^n}{\rho^n}\right)  -\gamma\sigma^n\partial_x u^n - (\gamma - 1)\partial_x(\kappa_n\partial_x\theta^n).
\end{equation}
Moreover, for all $n\in\N$, using \eqref{eq:12}--\eqref{eq:14},
\begin{equation}\label{eq:104}
\int_0^T\int_0^1 (u^n\partial_x\sigma^n)^2\leq \left(\sup_{[0,T]\times\T} (u^n)^2\right)\int_0^T\int_0^1(\partial_x\sigma^n)^2 \leq C_1^2,
\end{equation}
\begin{equation*}
\int_0^T\int_0^1 (\sigma^n\partial_x u^n)^2\leq \sup_{0\leq t\leq T}\int_0^1(\partial_x u^n)^2\int_0^T\Vert\sigma^n\Vert_\infty^2 \leq C_1^2,
\end{equation*}
\begin{equation*}
\int_0^T\int_0^1 \mu^2\left(\frac{\partial_x \sigma^n}{\rho^n}\right)^2\leq \frac{\mu^2}{\underline{\rho}^2}C_1,
\end{equation*}
\begin{equation}\label{eq:105}
\int_0^T\int_0^1 (\gamma-1)^2(\kappa_n\partial_x\theta^n)^2\leq (\gamma-1)^2\left(\sup_n\kappa_n\right)C_1.
\end{equation}
Recall that $(\kappa_n)$ is assumed to be bounded. Finally, from \eqref{eq:103} and \eqref{eq:104}--\eqref{eq:105}, we get that $(\partial_t\sigma^n)$ is bounded in $L^2(0,T,H^{-1}(\T))$. As $(\sigma^n)$ is bounded in $L^2(0,T,L^2(\T))$ (from \eqref{eq:13}) and the injection $L^2(\T)\subset H^{-1}(\T)$ is compact, we get from the Aubin-Lions lemma
\begin{equation}\label{eq:106} 
\sigma^n \tend{n}{+\infty} \sigma \quad\text{in }L^2(0,T,H^{-1}(\T)).
\end{equation}
Combining \eqref{eq:102} and \eqref{eq:106}, we get
\begin{equation}\label{eq:107}
\int_0^T\int_0^1 (\sigma^n)^2 \tend{n}{+\infty}\int_0^T\int_0^1\sigma^2,
\end{equation}
and at last, \eqref{eq:101} using \eqref{eq:102} and \eqref{eq:107}.
\end{proof}

From now until the end of the section, if a sequence of functions $(f^n)$ has a limit in ${\cal D}'([0,T]\times\T)$ when $n\rightarrow +\infty$, we will denote this limit $\langle f\rangle$. 

%

\subsection{\texorpdfstring{The case $\kappa = 0$}{The case kappa = 0}}
 The following lemma is the fundamental ingredient in the proof of Theorem \ref{thm:3}.
\begin{lemma}\label{lemm:coucou}
Let $(\rho_0^n,u_0^n,\theta_0^n)_n \subset L^2(\T)^3$ be a sequence of triplets satisfying \eqref{eq:16}--\eqref{eq:18} with $\underline{\rho_0},\overline{\rho_0},\underline{\theta_0},\overline{\theta_0}, C_0$ independant of $n$. Consider $(\rho^n,u^n,\theta^n)$ the global "à la Hoff" solution of \eqref{eq:1}--\eqref{eq:3} with initial conditions $(\rho_0^n,u_0^n,\theta_0^n)$ and without conductivity. Assume moreover that
\begin{equation*}
\rho_0^n,\theta_0^n \tend{n}{+\infty} \rho_0,\theta_0\quad \text{ in }L^2(\T).
\end{equation*} 
Then, up to a subsequence,
\begin{equation}\label{eq:108}
\rho^n\theta^n \tend{n}{+\infty} \langle\rho\theta\rangle \quad \text{in }L^1([0,T]\times\T).
\end{equation}
\end{lemma}
\begin{remark}
An other proof of this result can be found in \cite{BrBuGJLa24}. However, the proof that we give here is more elegant because it makes greater use of the structure of the Navier-Stokes equations without conductivity rather than strong convergence on the Cauchy stress. In particular, this proof could be adapted in dimension greater than 1.
\end{remark}

\begin{proof}[Proof of Lemma \ref{lemm:coucou}]
Let us recall the identity, for all $n\in\N$,
\begin{equation}\label{eq:109}
\partial_t ({(p^n)}^{1/\gamma}) + \partial_x({(p^n)}^{1/\gamma} u^n) = \frac{\gamma-1}{\gamma}\mu(\partial_x u^n)^2 {(p^n)}^{1/\gamma - 1}
\end{equation}
where $p^n=R\rho^n\theta^n$. From \eqref{eq:12},\eqref{eq:13} and \eqref{eq:14}, $({(p^n)}^{1/\gamma})$ is bounded in $L^\infty([0,T]\times\T)$ and $((\partial_x u^n)^2{(p^n)}^{1/\gamma-1})$ is bounded in $L^2([0,T], L^2(\T))$. Thus these two sequences converge (up to a subsequence) in ${\cal D}'([0,T]\times\T)$, and from \eqref{eq:109} and the strong convergence \eqref{eq:100} of $(u^n)$, we get
\begin{equation}\label{eq:110}
\partial_t \langle p^{1/\gamma}\rangle + \partial_x(\langle p^{1/\gamma}\rangle u) = \frac{\gamma-1}{\gamma}\mu \langle (\partial_x u)^2 p^{1/\gamma-1}\rangle.
\end{equation}
As $(\rho^n)$ and $(p^n)$ are bounded in $L^\infty([0,T]\times\T)$, they converge (up to a subsequence) in $L^\infty([0,T]\times\T)-\star$. 
Passing now to the limit in \eqref{eq:1}, \eqref{eq:2}, we obtain using \eqref{eq:100}
\begin{equation}
\partial_t \langle\rho\rangle + \partial_x (\langle\rho\rangle u)= 0, \label{eq:111}
\end{equation}
\begin{equation}
\partial_t (\langle \rho\rangle u) + \partial_x (\langle\rho\rangle u^2) = \partial_x(\mu\partial_x u - \langle p \rangle). \label{eq:112}
\end{equation}
Moreover, using $p^n=\cv(\gamma-1)\rho^n\theta^n$, we get from \eqref{eq:5}, for all $n\in\N$,
\begin{equation}
\frac{1}{\gamma-1}\partial_t p^n + \frac{1}{\gamma-1}\partial_x(p^n u^n) = (\mu\partial_x u^n - p^n) \partial_x  u^n
 \label{eq:113}.
\end{equation}
As $(\partial_x u^n)$ is bounded in $L^2(0,T,L^2(\T))$, and as \eqref{eq:101} holds with
\begin{equation*}
\sigma = \mu\partial_x u - \langle p \rangle,
\end{equation*}
passing to the limit in \eqref{eq:113} we obtain
\begin{equation}
\frac{1}{\gamma-1}\partial_t \langle p\rangle + \frac{1}{\gamma-1}\partial_x(\langle p \rangle u) = (\mu\partial_x u - \langle p \rangle)\partial_x u.
\label{eq:114}
\end{equation}
Thus, from \eqref{eq:111}, \eqref{eq:112} and \eqref{eq:114}, $\left(\langle \rho\rangle, u, \dfrac{\langle p\rangle}{R\langle\rho\rangle}\right)$ is solution to $(NS_0)$. So we also have the equality
\begin{equation}\label{eq:115}
\partial_t \langle p\rangle^{1/\gamma} + \partial_x(\langle p \rangle^{1/\gamma} u) = \frac{\gamma-1}{\gamma}\mu(\partial_x u)^2 \langle p\rangle^{1/\gamma-1}.
\end{equation}
The crucial point now is that the application $\varphi:\R\times\R_+\rightarrow \R, (y,z)\mapsto y^2z^{1/\gamma-1}\text{ is convex}$.
Indeed, we have, for all $(y,z)\in \R\times \R_+$,
\[
D^2\varphi(y,z) = \begin{pmatrix}
2z^{1/\gamma-1} & 2\left(\dfrac{1}{\gamma}-1\right)yz^{1/\gamma-2}\\
2\left(\dfrac{1}{\gamma}-1\right)yz^{1/\gamma-2} & \left(\dfrac{1}{\gamma}-1\right)\left(\dfrac{1}{\gamma}-2\right)y^2z^{1/\gamma-3}
\end{pmatrix},
\]
with $2z^{1/\gamma-1} \geq 0$ and
\begin{equation*}
\det D^2\varphi(y,z) = \left[\left(\frac{1}{\gamma}-1\right)\left(\frac{1}{\gamma}-2\right) - 2\left(\frac{1}{\gamma}-1\right)^2\right]2y^2z^{2/\gamma-4} = \frac{2}{\gamma}\left(1-\frac{1}{\gamma}\right)y^2z^{2/\gamma-4}\geq 0
\end{equation*}
because $\gamma > 1$. From this convexity, we deduce
\begin{equation}\label{eq:116}
\frac{\gamma-1}{\gamma}\mu \langle (\partial_x u)^2p^{1/\gamma-1}\rangle \geq \frac{\gamma-1}{\gamma}\mu (\partial_x u)^2 \langle p\rangle^{1/\gamma-1}.
\end{equation}
Then \eqref{eq:110}, \eqref{eq:115} and \eqref{eq:116} give
\begin{equation}\label{eq:117}
\partial_t \langle p^{1/\gamma}\rangle + \partial_x(\langle p^{1/\gamma}\rangle  u) \geq  \partial_t(\langle p\rangle^{1/\gamma}) + \partial_x(\langle p\rangle^{1/\gamma}u).
\end{equation}
Because $\rho_0^n \tend{n}{+\infty} \rho_0$ and $\theta_0^n \tend{n}{+\infty} \theta_0$ in $L^2(\T)$, we have $\langle p_0^{1/\gamma}\rangle = \langle p_0\rangle^{1/\gamma}$.
Hence, fixing $t\in [0,T]$ and integrating \eqref{eq:117} over $[0,t]\times\T$, we finally obtain
\begin{equation}\label{eq:118}
\text{for almost all }t\in[0,T],\quad \int_0^1 \langle p^{1/\gamma}\rangle(t) \geq \int_0^1 \langle p\rangle^{1/\gamma}(t).
\end{equation}
Besides, as $\gamma>1$, $y\mapsto y^{1/\gamma}$ is strictly concave on $\R_+$, hence
\begin{equation}\label{eq:119}
\text{for almost all }(t,x)\in [0,T]\times \T,\quad \langle p^{1/\gamma}\rangle(t,x) \leq \langle p\rangle^{1/\gamma}(t,x).
\end{equation}
Combining \eqref{eq:118} and \eqref{eq:119}, we get
\begin{equation}\label{eq:120}
\text{for almost all }(t,x)\in [0,T]\times \T,\quad \langle p^{1/\gamma}\rangle(t,x) = \langle p\rangle^{1/\gamma}(t,x),
\end{equation}
and as the concavity of $y\mapsto y^{1/\gamma}$ is strict, \eqref{eq:120} implies (see \cite{NoSt}, Lemma 3.34)
\begin{equation*}
p^n \tend{n}{\infty} \langle p\rangle \quad \text{ in }L^1(0,T,L^1(\T)),
\end{equation*}
hence \eqref{eq:108} because $\langle p \rangle = R\langle \rho\theta\rangle$.
\end{proof}
We are now in position to prove Theorem \ref{thm:3}.
\begin{proof}[Proof of Theorem \ref{thm:3}]
From \eqref{eq:108} and the fact that $\rho^n\theta^n\in L^\infty([0,T]\times\T)$, we deduce by interpolation that
\begin{equation}\label{eq:121}
p^n \tend{n}{+\infty} \langle p \rangle \quad\text{in }L^2(0,T,L^2(\T)). 
\end{equation}
Combining \eqref{eq:121} and \eqref{eq:101}, we get
\begin{equation}\label{eq:122}
\partial_x u^n \tend{n}{+\infty}\partial_x u \quad\text{in }L^2(0,T,L^2(\T)).
\end{equation}
For $n\in\N$, multiplying now \eqref{eq:1} by $1+\ln\rho^n$, we get
\begin{equation}\label{eq:123}
\partial_t(\rho^n\ln\rho^n) + \partial_x((\rho^n\ln\rho^n) u^n) = -\rho\partial_x u^n.
\end{equation}
Passing to the limit in \eqref{eq:123} and using \eqref{eq:122}, we obtain
\begin{equation}\label{eq:124}
\partial_t \langle \rho\ln\rho\rangle + \partial_x(\langle\rho\ln\rho\rangle u) = -\langle\rho\rangle\partial_x u.
\end{equation}
Besides, multiplying \eqref{eq:111} by $1 + \ln \langle\rho\rangle$ gives
\begin{equation}\label{eq:125}
\partial_t(\langle\rho\rangle \ln\langle\rho\rangle) + \partial_x(\langle\rho\rangle \ln\langle\rho\rangle u) = -\langle\rho\rangle \partial_x u.
\end{equation}
Substracting now \eqref{eq:124} to \eqref{eq:125}, we obtain
\begin{equation}\label{eq:126}
\partial_t (\langle\rho\rangle \ln\langle\rho\rangle) + \partial_x(\langle\rho\rangle \ln\langle\rho\rangle u) =  \partial_t\langle \rho\ln\rho\rangle  + \partial_x(\langle \rho\ln\rho\rangle u).
\end{equation}
Because $\rho_0^n\tend{n}{+\infty}\rho_0$ in $L^2(\T)$, we get
\begin{equation}\label{eq:127}
\langle\rho_0\rangle\ln\langle\rho_0\rangle = \langle\rho_0\ln\rho_0\rangle.
\end{equation}
As $ u\in L^2(0,T,W^{1,\infty}(\T))$ (from \eqref{eq:14}), \eqref{eq:126} and \eqref{eq:127} give
\begin{equation}\label{eq:128}
\text{for almost all }(t,x)\in [0,T]\times \T,\quad \langle \rho\ln\rho\rangle(t,x) = \langle \rho\rangle \ln\langle\rho\rangle(t,x).
\end{equation}
As $y\mapsto y\ln y$ is strictly convex, \eqref{eq:128} implies that (see \cite{NoSt}, Lemma 3.34)
\begin{equation}\label{eq:129}
\rho^n \tend{n}{+\infty} \langle \rho \rangle \quad \text{in }L^1([0,T]\times\T).
\end{equation}
This convergence is always true in $L^2([0,T]\times\T)$ because $\rho^n\in L^\infty([0,T]\times\T)$. Combining \eqref{eq:108} and \eqref{eq:129}, we get
\begin{equation*}
\theta^n \tend{n}{+\infty} \langle \theta \rangle \quad \text{in }L^1([0,T]\times\T),
\end{equation*}
convergence similarly still true in $L^2([0,T]\times\T)$. In particular,
\begin{equation*}
\langle \rho\theta\rangle = \langle\rho\rangle \langle\theta\rangle.
\end{equation*}
Thus $\langle p \rangle/(R\langle \rho\rangle) = \langle\theta\rangle$, and $(\langle \rho\rangle, u, \langle\theta\rangle)$ is solution to $(NS_0)$.
\end{proof}

\subsection{\texorpdfstring{Convergence as $\kappa$ tends to $0$ for well-prepared initial data}{Convergence as kappa tends to zero for well-prepared initial data}}
In this subsection, we show that Theorem \ref{thm:main} holds under additional conditions for the initial data. 
Let's start by showing the following result of regularity propagation over time :
\begin{prop}
\label{prop:18}
Let $(\rho_0,u_0,\theta_0)$ satisfying \eqref{eq:16}--\eqref{eq:18} and $(\rho,u,\theta)$ the global "à la Hoff" solution of $(NS_\kappa)$, with $\kappa\leq \overline{\kappa}$, and with initial conditions $\rho_0,u_0,\theta_0$. Assume moreover that there exists some constants $\alpha>0$ and $D_0>0$ such that
\begin{equation}
\int_0^1 \kappa^\alpha (\partial_x \theta_0)^2 + \int_0^1 \kappa^\alpha (\partial_x \rho_0)^2 \leq D_0.\label{eq:130}
\end{equation}
Then there exists some $D_1 = D_1(C_0, D_0, \overline{\kappa}, \underline{\rho_0}, \overline{\rho_0}, T, \underline{\theta_0}, \overline{\theta_0})>0$ such that
\begin{equation}\label{eq:131}
\text{for all }t\in [0,T],\quad \int_0^1 \kappa^\alpha (\partial_x \theta)^2(t) + \int_0^1 \kappa^\alpha (\partial_x \rho)^2(t) + \kappa^{\alpha-1}\int_0^t\int_0^1 [\partial_x(\kappa\partial_x\theta)]^2\leq D_1.
\end{equation}
\end{prop}
\begin{proof}
Applying $\partial_x$ to \eqref{eq:1}, we get
\begin{equation}\label{eq:132}
\partial_t \partial_x \rho + \partial_x(u\partial_x \rho) = -\partial_x(\rho\partial_x u).
\end{equation}
Then multiplying \eqref{eq:132} by $\partial_x \rho$ we obtain
\begin{equation*}
\frac{1}{2}\partial_t (\partial_x \rho)^2 + \frac{1}{2}\partial_x(u(\partial_x \rho)^2) = -\frac{1}{2}(\partial_x u)(\partial_x \rho)^2 -(\partial_x(\rho\partial_x u))\partial_x \rho,
\end{equation*}
hence
\begin{equation}\label{eq:133}
\frac{1}{2}\partial_t (\partial_x\rho)^2 + \frac{1}{2}\partial_x(u(\partial_x\rho)^2) = -\frac{3}{2}(\partial_x u)(\partial_x\rho)^2 - \rho(\partial_x\rho)\partial_{xx}u.
\end{equation}
On the other hand we have
\begin{equation}\label{eq:134}
\mu\partial_{xx}u = \partial_x\sigma +R\partial_x(\rho\theta) = \partial_x\sigma + R\theta\partial_x\rho + R\rho\partial_x\theta. 
\end{equation}
Thus \eqref{eq:134} in \eqref{eq:133} gives
\begin{equation}\label{eq:135}
\frac{1}{2}\partial_t (\partial_x\rho)^2 + \frac{1}{2}\partial_x(u(\partial_x\rho)^2) = -\frac{3}{2}(\partial_x u)(\partial_x\rho)^2 - \frac{1}{\mu}\rho(\partial_x\rho)(\partial_x\sigma) - \frac{R}{\mu}\rho\theta(\partial_x\rho)^2 - \frac{R}{\mu}\rho^2(\partial_x\rho)(\partial_x\theta). 
\end{equation}
Multiplying \eqref{eq:135} by $\kappa^\alpha$ then  
integrating on the torus we get by Young's inequality
\begin{align}
&\frac{1}{2}\frac{d}{dt}\int_0^1 \kappa^\alpha(\partial_x\rho)^2 + \frac{R}{\mu}\int_0^1 \kappa^\alpha\rho\theta(\partial_x\rho)^2 \nonumber\\&= -\frac{3}{2}\int_0^1 \kappa^\alpha(\partial_x u)(\partial_x\rho)^2 -\frac{1}{\mu}\int_0^1 \kappa^\alpha\rho (\partial_x\rho)(\partial_x \sigma) - \frac{R}{\mu}\int_0^1\kappa^\alpha \rho^2(\partial_x\rho)(\partial_x\theta)\nonumber
\\&\leq \frac{3}{2}\Vert\partial_x u\Vert_\infty\int_0^1 \kappa^\alpha (\partial_x\rho)^2 + \frac{\overline{\rho}^2\kappa^\alpha}{2\mu^2}\int_0^1 (\partial_x\sigma)^2 + \frac{1}{2}\int_0^1 \kappa^\alpha (\partial_x\rho)^2+ \frac{R\overline{\rho}^2}{2\mu}\int_0^1\kappa^\alpha (\partial_x\theta)^2 +\frac{R\overline{\rho}^2}{2\mu} \int_0^1 \kappa^\alpha (\partial_x\rho)^2\nonumber
\\&\leq \frac{1}{2}\left(3\Vert\partial_x u\Vert_\infty + \frac{R\overline{\rho}^2}{\mu} + 1\right)\int_0^1\kappa^\alpha (\partial_x\rho)^2 + \frac{R\overline{\rho}^2}{2\mu}\int_0^1 \kappa^\alpha(\partial_x\theta)^2 + \frac{\overline{\rho}^2\kappa^\alpha}{2\mu^2}\int_0^1(\partial_x\sigma)^2.\label{eq:136}
\end{align}
Besides, dividing \eqref{eq:5} by $\rho$ and applying $\partial_x$, we get
\begin{equation}\label{eq:137}
\partial_t(\cv\partial_x\theta) + \partial_x(u\cv\partial_x\theta) = \partial_x \left(\frac{\sigma\partial_x u}{\rho}\right) + \partial_x\left(\frac{\partial_{x}(\kappa\partial_x\theta)}{\rho}\right).
\end{equation}
Then multiplying \eqref{eq:137} by $\partial_x\theta/\cv$ we obtain
\begin{equation}\label{eq:138}
\frac{1}{2}\partial_t(\partial_x\theta)^2 + \frac{1}{2}\partial_x(u(\partial_x\theta)^2) = -\frac{1}{2}(\partial_x u)(\partial_x\theta)^2 + \frac{\partial_x\theta}{\cv}\partial_x\left(\frac{\sigma\partial_x u}{\rho}\right) + \frac{\partial_x\theta}{\cv} \partial_x\left(\frac{\partial_{x}(\kappa\partial_x\theta)}{\rho}\right).
\end{equation}
Moreover,
\begin{equation}\label{eq:139}
\partial_x \left(\frac{\sigma\partial_x u}{\rho}\right) = \frac{1}{\mu}\partial_x\left(\frac{\sigma^2}{\rho}\right) + \frac{R}{\mu}\partial_x(\sigma\theta) = \frac{2}{\rho\mu}\sigma\partial_x\sigma - \frac{\sigma^2}{\mu\rho^2}\partial_x\rho + \frac{R}{\mu}\theta\partial_x\sigma + \frac{R}{\mu}\sigma\partial_x\theta.
\end{equation}
Then \eqref{eq:139} in \eqref{eq:138} gives
\begin{equation}\label{eq:140}
\begin{split}
&\frac{1}{2}\partial_t(\partial_x\theta)^2 + \frac{1}{2}\partial_x(u(\partial_x\theta)^2) -\frac{\partial_x\theta}{\cv}\partial_x\left(\frac{\partial_{x}(\kappa\partial_x\theta)}{\rho}\right) \\& = -\frac{1}{2}(\partial_x u)(\partial_x\theta)^2 + \frac{2}{\cv\mu}\frac{\sigma}{\rho}(\partial_x\sigma)(\partial_x\theta) - \frac{1}{\cv\mu}\frac{\sigma^2}{\rho^2}(\partial_x\rho)(\partial_x\theta) + \frac{R}{\cv\mu}\theta(\partial_x\sigma)(\partial_x\theta) + \frac{R\sigma}{\cv\mu}(\partial_x\theta)^2.
\end{split}
\end{equation}
Multiplying \eqref{eq:140} by $\kappa^\alpha$ then integrating on the torus, we get after some integration by parts and by Young's inequality

\begin{align}
&\frac{1}{2}\frac{d}{dt}\int_0^1 \kappa^\alpha (\partial_x\theta)^2 + \frac{\kappa^{\alpha-1}}{\cv}\int_0^1 \frac{[\partial_x(\kappa\partial_x\theta)]^2}{\rho} \nonumber\\& = -\frac{1}{2}\int_0^1\kappa^\alpha(\partial_x u)(\partial_x\theta)^2 + \frac{2}{\cv\mu}\int_0^1\kappa^\alpha \frac{\sigma}{\rho}(\partial_x\sigma)(\partial_x\theta) - \frac{1}{\cv\mu}\int_0^1 \kappa^\alpha \frac{\sigma^2}{\rho^2}(\partial_x\rho)(\partial_x\theta) \nonumber
\\&+ \frac{R}{\cv\mu}\int_0^1\kappa^\alpha \theta (\partial_x\sigma)(\partial_x\theta) + \frac{R}{\cv\mu}\int_0^1 \kappa^\alpha\sigma (\partial_x\theta)^2\nonumber
\\&\leq \frac{1}{2}\Vert \partial_x u\Vert_\infty \int_0^1\kappa^\alpha (\partial_x\theta)^2 + \frac{2\kappa^\alpha}{\cv^2\mu^2\underline{\rho}^2}\int_0^1 (\partial_x\sigma)^2 + \frac{\Vert\sigma\Vert_\infty^2}{2}\int_0^1\kappa^\alpha (\partial_x\theta)^2  
+ \frac{1}{2\cv\mu\underline{\rho}^2} \Vert \sigma\Vert_\infty^2\int_0^1\kappa^\alpha (\partial_x\rho)^2\nonumber \\&+\frac{1}{2\cv\mu\underline{\rho}^2}\Vert\sigma\Vert_\infty^2 \int_0^1 \kappa^\alpha(\partial_x\theta)^2
+\frac{R^2\overline{\theta}^2\kappa^\alpha}{2\cv^2\mu^2}\int_0^1 (\partial_x\sigma)^2 + \frac{1}{2}\int_0^1\kappa^\alpha (\partial_x\theta)^2 + \frac{R}{\cv\mu}\Vert\sigma\Vert_\infty\int_0^1\kappa^\alpha(\partial_x\theta)^2 \nonumber
\\&\leq \frac{1}{2\cv\mu\underline{\rho}^2}\Vert\sigma\Vert_\infty^2\int_0^1\kappa^\alpha(\partial_x\rho)^2 + \frac{1}{2}\left(\Vert \partial_x u\Vert_\infty + \left(1+\frac{1}{\cv\mu\underline{\rho}^2}\right)\Vert\sigma\Vert_\infty^2 + 1 + \frac{2R}{\cv\mu}\Vert\sigma\Vert_\infty\right)\int_0^1\kappa^\alpha(\partial_x\theta)^2 \nonumber
\\&+ \left(\frac{2}{\underline{\rho}^2} + \frac{R^2\overline{\theta}^2}{2}\right)\frac{\kappa^\alpha}{\cv^2\mu^2}\int_0^1(\partial_x\sigma)^2.\label{eq:141}
\end{align}
Combining \eqref{eq:136} and \eqref{eq:141} give
\begin{align*}
&\frac{d}{dt}\int_0^1 \kappa^\alpha(\partial_x\rho)^2 + \frac{d}{dt}\int_0^1 \kappa^\alpha (\partial_x\theta)^2 + \frac{2R}{\mu}\int_0^1 \kappa^\alpha\rho\theta(\partial_x\rho)^2 + \frac{2}{\cv}\kappa^{\alpha-1}\int_0^1 \frac{[\partial_x(\kappa\partial_x\theta)]^2}{\rho}
\\&\leq \left(\overline{\rho}^2\cv^2+ \frac{4}{\underline{\rho}^2}+ R^2\overline{\theta}^2\right)\frac{\kappa^\alpha}{\cv^2\mu^2}\int_0^1 (\partial_x\sigma)^2 + \left(3\Vert \partial_x u\Vert_\infty + \frac{R\overline{\rho}^2}{\mu} + 1 + \frac{1}{\cv\mu\underline{\rho}^2}\Vert\sigma\Vert_\infty^2\right)\int_0^1 \kappa^\alpha(\partial_x\rho)^2 \\&+ \left(\frac{R\overline{\rho}^2}{\mu} + 1 + \Vert \partial_x u\Vert_\infty + \left(1 + \frac{1}{\cv\mu\underline{\rho}^2}\right)\Vert\sigma\Vert_\infty^2 + \frac{2R}{\cv\mu}\Vert\sigma\Vert_\infty\right)\int_0^1\kappa^\alpha (\partial_x\theta)^2   
\end{align*}
hence, by Grönwall's lemma, for all $t\in [0,T]$,
\begin{equation}\label{eq:142}
\int_0^1 \kappa^\alpha(\partial_x\rho)^2(t) + \int_0^1\kappa^\alpha(\partial_x\theta)^2(t) + \frac{2}{\cv}\kappa^{\alpha-1}\int_0^t\int_0^1\frac{[\partial_x(\kappa\partial_x\theta)]^2}{\rho} \leq \left(D_0 + \int_0^t G_1(s)\right)\exp\left(\int_0^t G_2(s)\right) 
\end{equation}
where
\begin{equation*}
G_1 : = \left(\overline{\rho}^2\cv^2+ \frac{4}{\underline{\rho}^2}+ R^2\overline{\theta}^2\right)\frac{\kappa^\alpha}{\cv^2\mu^2}\int_0^1 (\partial_x\sigma)^2,
\end{equation*}
\begin{equation*}
G_2 := 1 + \frac{R\overline{\rho}^2}{\mu} + 3\Vert\partial_x u\Vert_\infty + \left(1+\frac{1}{\cv\mu\underline{\rho}^2}\right)\Vert\sigma\Vert_\infty^2 + \frac{2R}{\cv\mu}\Vert\sigma\Vert_\infty.
\end{equation*}
Using \eqref{eq:60} and \eqref{eq:12}--\eqref{eq:14}, we obtain
\begin{equation}\label{eq:143}
\int_0^T G_1\leq \left(\overline{\rho}^2\cv^2+ \frac{4}{\underline{\rho}^2} + R^2\overline{\theta}^2\right)\frac{\kappa^\alpha}{\cv^2\mu^2}C_1,
\end{equation}
\begin{equation}\label{eq:144}
\int_0^T G_2 \leq \left(1 + \frac{R\overline{\rho}^2}{\mu}\right)T + \left(3 + \frac{2R}{\cv\mu}\right)\sqrt{T}\sqrt{C_1} + \left(1+\frac{1}{\cv\mu \underline{\rho}^2}\right)C_1.
\end{equation}
Finally, combining \eqref{eq:142}, \eqref{eq:143} and \eqref{eq:144}, we get for all $t\in [0,T]$
\begin{equation*}
\begin{split}
\int_0^1 \kappa^\alpha(\partial_x\rho)^2(t) &+ \int_0^1\kappa^\alpha(\partial_x\theta)^2(t) + \kappa^{\alpha-1}\int_0^t\int_0^1 [\partial_x(\kappa\partial_x\theta)]^2(s)ds \\&\leq \max(1, \cv\overline{\rho}/2)\left(D_0 + \left(\overline{\rho}^2\cv^2+ \frac{4}{\underline{\rho}^2} R^2\overline{\theta}^2\right)\frac{\kappa^\alpha}{\cv^2\mu^2}C_1\right)\\&\times\left(\left(1 + \frac{R\overline{\rho}^2}{\mu}\right)T + \left(3 + \frac{2R}{\cv\mu}\right)\sqrt{T}\sqrt{C_1} + \left(1+\frac{1}{\cv\mu \underline{\rho}^2}\right)C_1\right).
\end{split}
\end{equation*}
\end{proof}
The following proposition is a weaker version of Theorem \ref{thm:main}.
\begin{prop}\label{prop:4prime}
Let $({\kappa_n})\subset [0,+\infty[$ some sequence tending towards zero when $n$ tends towards infinity. Let $(\rho_0^{\kappa_n},u_0^{\kappa_n},\theta_0^{\kappa_n})\in L^2(\T)^3$. Assume that $(\rho_0^{\kappa_n},u_0^{\kappa_n},\theta_0^{\kappa_n})$ verifies  \eqref{eq:16}--\eqref{eq:18} and \eqref{eq:130} for some $\alpha<1$, with $\underline{\rho_0},\overline{\rho_0}, \underline{\theta_0},\overline{\theta_0}, C_0, D_0$ independant of $n$. Assume moreover that there exists some $(\rho_0,u_0, \theta_0)\in L^2(\T)\times H^1(\T)\times L^2(\T)$ such that
\begin{equation*}
\rho_0^{{\kappa_n}},\theta_0^{{\kappa_n}} \tend{n}{+\infty} \rho_0,\theta_0\quad \text{ in }L^2(\T),\quad u_0^{{\kappa_n}} \tendf{n}{+\infty} u_0\quad \text{ in } H^1(\T).
\end{equation*} 
Then, for all $T>0$, there exists some $(\rho,u,\theta)\in L^2(0,T,L^2(\T))\times L^2(0,T,H^1(\T))\times L^2(0,T,L^2(\T))$ such that the "à la Hoff" solution $(\rho^{{\kappa_n}}, u^{{\kappa_n}},\theta^{{\kappa_n}})$ of $(NS_{{\kappa_n}})$ with initial conditions $\rho_0^{{\kappa_n}},\theta_0^{{\kappa_n}}, u_0^{{\kappa_n}}$ verifies
\begin{equation*}
\rho^{{\kappa_n}},\theta^{{\kappa_n}} \tend{n}{+\infty} \rho,\theta\quad \text{ in }L^2(0,T,L^2(\T)),\quad u^{{\kappa_n}} \tendf{n}{+\infty}u\quad \text{ in } L^2(0,T,H^1(\T)). 
\end{equation*}
Moreover, $(\rho,u,\theta)$ is the "à la Hoff" solution of $(NS_0)$ with initial conditions $(\rho_0,u_0,\theta_0)$.
\end{prop}
\begin{proof} Let us repeat the proof of Theorem \ref{thm:3}. To begin with, let us show that \eqref{eq:108} is true. In the case with conductivity, \eqref{eq:11} becomes
\begin{equation}\label{eq:145}
\partial_t (p^{1/\gamma}) + \partial_x(p^{1/\gamma} u) = \frac{\gamma-1}{\gamma}\mu(\partial_x u)^2 p^{1/\gamma - 1} + \frac{\gamma-1}{\gamma}[\partial_x(\kappa\partial_x\theta)]p^{1/\gamma - 1}.
\end{equation}
And from \eqref{eq:131}, from any $n\in\N$ we have
\begin{equation*}
\int_0^T\int_0^1 [\partial_x({\kappa_n}\partial_x\theta^n)]^2\leq {\kappa_n}^{1-\alpha}D_1.
\end{equation*}
Thus, as $\alpha<1$ and ${\kappa_n}\tend{n}{+\infty}0$,
\begin{equation}\label{eq:146}
\partial_x({\kappa_n}\partial_x\theta^n) \tend{n}{+\infty} 0 \quad\text{in }L^2(0,T,L^2(\T)).
\end{equation}
Passing to the limit in \eqref{eq:145} and using \eqref{eq:146}, we get using Proposition \ref{prop:17} (as $(\kappa_n)$ is bounded)
\begin{equation*}
\partial_t \langle p^{1/\gamma}\rangle + \partial_x(\langle p^{1/\gamma}\rangle u) = \frac{\gamma-1}{\gamma}\mu \langle (\partial_x u)^2 p^{1/\gamma-1}\rangle.
\end{equation*}
Besides, using again \eqref{eq:146}, $\left(\langle \rho\rangle, \langle u\rangle, \dfrac{\langle p\rangle}{R\langle\rho\rangle}\right)$ is solution to $(NS_0)$. Thus
\begin{equation*}
\partial_t \langle p\rangle^{1/\gamma} + \partial_x(\langle p \rangle^{1/\gamma} u) = \frac{\gamma-1}{\gamma}\mu(\partial_x u)^2 \langle p\rangle^{1/\gamma-1},
\end{equation*}
and the end of the proof is the same that the proof of Theorem \ref{thm:3}.
\end{proof}
\begin{remark}\label{rmk:1}
Without the assumption \eqref{eq:130}, we can still prove that $\left(\langle \rho\rangle, \langle u\rangle, \dfrac{\langle p\rangle}{R\langle\rho\rangle}\right)$ is solution to $(NS_0)$. Indeed, as $\sqrt{\kappa}\partial_x\theta\in L^2(0,T,L^2(\T))$, we get
\begin{equation*}
\partial_x({\kappa_n}\partial_x\theta^n) \tend{n}{+\infty} 0 \quad \text{in }{\cal D}'([0,T]\times\T)
\end{equation*}
which is suffisiant to pass to the limit in \eqref{eq:5}. However, the strong convergence seems essential to pass to the limit in \eqref{eq:145}.
\end{remark}

\section{Stability results for fixed values of \texorpdfstring{$\kappa$}{kappa} and conclusions.}\label{sect:3}
\subsection{More bounds for well-prepared data}
\begin{prop}\label{prop:pullnoel}
Let $(\rho_0,u_0,\theta_0)$ satisfying \eqref{eq:16}--\eqref{eq:18} and $(\rho,u,\theta)$ be the global "à la Hoff" solution of $(NS_\kappa)$, with $0<\kappa\leq \overline{\kappa}$, with initial conditions $\rho_0,u_0,\theta_0$. Assume that there exists a constant $E_0>0$ such that
\begin{equation}\label{eq:147}
\kappa\int_0^1\vert\partial_x\theta_0\vert^2 \leq E_0.
\end{equation}
Then there exists some $E_1 = E_1(C_0, E_0, \overline{\kappa}, \underline{\rho_0}, \overline{\rho_0}, T, \underline{\theta_0}, \overline{\theta_0})>0$ such that
\begin{equation}\label{eq:148}
\sup_{0\leq t \leq T}\kappa\int_0^1\vert\partial_x\theta\vert^2(t) +\int_0^T\int_0^1 \vert D_t\theta\vert^2 \leq E_1.
\end{equation}
Let us assume furthermore that there exists some $E_0'>0$ such that
\begin{equation}\label{eq:149}
\kappa\int_0^1\vert\partial_x\sigma_0\vert^2\leq E_0'.
\end{equation}
Then there exists some $E_1' = E_1'(C_0,E_0,E_0',\overline{\kappa},\underline{\rho_0},\overline{\rho_0},T,\underline{\theta_0},\overline{\theta_0})>0$ such that
\begin{equation*}
\sup_{0\leq t\leq T}\kappa\int_0^1\vert \partial_x\sigma\vert^2(t) + \kappa\int_0^T\int_0^1\Vert\partial_x\sigma\Vert_\infty^2 \leq E_1'.
\end{equation*}
\end{prop}
\begin{proof}
Integrating \eqref{eq:93} on $[0,t]$ for $t\in [0,T]$, we get
\begin{equation*}
\int_0^1 \kappa(\partial_x\theta)^2(t) + \int_0^t\int_0^1 \frac{[\partial_x(\kappa\partial_x\theta)]^2}{\rho\cv} \leq \int_0^1\kappa\vert\partial_x\theta_0\vert^2 + \int_0^T\Vert\partial_x u\Vert_\infty \int_0^1 \kappa(\partial_x\theta)^2 + \frac{1}{\underline{\rho}\cv}\int_0^T \Vert\sigma\Vert_\infty^2\int_0^1(\partial_x u)^2.
\end{equation*}
Moreover,
\begin{equation*}
\vert \rho\cv D_t\theta\vert^2 = \vert\sigma\partial_x u + \partial_x(\kappa\partial_x\theta)\vert^2,
\end{equation*}
hence
\begin{equation*}
\vert D_t\theta\vert^2\leq \frac{2}{\underline{\rho}^2\cv^2}\sigma^2\vert\partial_x u\vert^2 + \frac{2}{\underline{\rho}\cv}\frac{\vert\partial_x(\kappa\partial_x\theta)\vert^2}{\rho\cv}
\end{equation*}
and
\begin{equation*}
\int_0^1\kappa\vert\partial_x\theta\vert^2(t) + \frac{\underline{\rho}\cv}{2}\int_0^t\int_0^1 \vert D_t\theta\vert^2 \leq E_0 + \frac{2C_1^2}{\underline{\rho}\cv} + \int_0^T\Vert\partial_x u\Vert_\infty \int_0^1 \kappa\vert\partial_x\theta\vert^2.
\end{equation*}
Thus, by Grönwall's lemma,
\begin{equation*}
\sup_{0\leq t\leq T}\left[\kappa\int_0^1\vert\partial_x\theta\vert^2(t)  +\frac{\underline{\rho}\cv}{2}\int_0^t\int_0^1\vert D_t\theta\vert^2\right]\leq \left(E_0 + \frac{2C_1^2}{\underline{\rho}\cv}\right)\exp(\sqrt{T C_1}),
\end{equation*}
hence \eqref{eq:148}.
Let us now assume that \eqref{eq:149} holds. Integrating \eqref{eq:96} on $[0,t]$ for $t\in [0,T]$, and multiplying the result by $\kappa$, we get
\begin{equation}\label{eq:150}
\begin{split}
\mu\kappa\int_0^1\frac{\vert\partial_x\sigma\vert^2}{\rho}(t) &+ \kappa\int_0^t\int_0^1\vert D_t\sigma\vert^2 \\&\leq \frac{\mu\kappa}{\underline{\rho}}\int_0^1\vert\partial_x\sigma_0\vert^2+2\gamma^2\kappa\int_0^T\Vert\partial_x u\Vert_\infty^2\int_0^1\sigma^2 + 2(\gamma-1)^2\kappa\int_0^T\int_0^1[\partial_x(\kappa\partial_x\theta)]^2.
\end{split}
\end{equation}
Moreover, from \eqref{eq:5},
\begin{align}
\int_0^t\int_0^1\vert\partial_x(\kappa\partial_x\theta)\vert^2&\leq 2\overline{\rho}^2\cv^2\int_0^t\int_0^1 \vert D_t\theta\vert^2 + 2\int_0^t\Vert\partial_x u\Vert_\infty^2\int_0^1\sigma^2 \nonumber\\&\leq 2\overline{\rho}^2\cv^2 E_1 + 2C_1^2,\label{eq:151}
\end{align}
and from \eqref{eq:9} and \eqref{eq:98},
\begin{align}
\int_0^t\int_0^1 \vert D_t\sigma\vert^2 \geq \frac{\mu^2}{2}\int_0^t\int_0^1 \left\vert \partial_x\left(\frac{\partial_x\sigma}{\rho}\right)\right\vert^2 - 2\gamma^2\int_0^t\Vert \partial_x u\Vert_\infty^2\int_0^1\sigma^2 - 2(\gamma-1)^2\int_0^t\int_0^1 \vert\partial_x(\kappa\partial_x\theta)\vert^2.\label{eq:152}
\end{align}
Using the Gagliardo-Nirenberg's inequality
\begin{equation*}
\left\Vert \frac{\partial_x\sigma}{\rho}\right\Vert_\infty^2 \leq \left\Vert \frac{\partial_x \sigma}{\rho}\right\Vert_2^2 + 2\left\Vert\frac{\partial_x\sigma}{\rho}\right\Vert_2\left\Vert\partial_x\left(\frac{\partial_x\sigma}{\rho}\right)\right\Vert_2,
\end{equation*}
we obtain
\begin{equation}\label{eq:153}
\int_0^t \Vert\partial_x\sigma\Vert_\infty^2 \leq \frac{2\overline{\rho}^2}{\underline{\rho}^2}\int_0^t\int_0^1\vert\partial_x\sigma\vert^2 + \overline{\rho}^2\int_0^t\int_0^1\left\vert\partial_x\left(\frac{\partial_x\sigma}{\rho}\right)\right\vert^2.
\end{equation}
Taking the supremum on $[0,T]$ in \eqref{eq:150} and combining \eqref{eq:151}, \eqref{eq:152} and \eqref{eq:153}, we get
\begin{equation*}
\sup_{0\leq t\leq T}\left[ \mu\kappa\int_0^1\vert\partial_x\sigma\vert^2(t) + \frac{\mu^2\kappa}{2\overline{\rho}^2}\int_0^t\Vert\partial_x\sigma\Vert_\infty^2 \right]\leq\frac{\mu}{\underline{\rho}}E_0' + 8(\gamma-1)^2\overline{\kappa}\overline{\rho}^2\cv^2 E_1 + \left(\frac{\mu^2}{\overline{\rho}^2} + 4\gamma^2 + 8(\gamma-1)^2\right)\overline{\kappa}C_1^2,
\end{equation*}
hence Proposition \ref{prop:pullnoel}.
\end{proof}
\begin{prop}\label{prop:pullnoel2}
Let $(\rho_0,u_0,\theta_0)$ satisfy \eqref{eq:16}--\eqref{eq:18} and $(\rho,u,\theta)$ be the global "à la Hoff" solution of $(NS_\kappa)$, with $0<\kappa\leq \overline{\kappa}$ and with initial conditions $\rho_0,u_0,\theta_0$. Assume that there exists a constant $F_0>0$ such that
\begin{equation}\label{eq:154}
\kappa\Vert\partial_x\rho_0\Vert_\infty^2 + \kappa\Vert\partial_x\theta_0\Vert_\infty^2 + \kappa\int_0^1\vert\partial_{xx} u_0\vert^2\leq F_0.
\end{equation}
Then there exists $F_1 = F_1(C_0, F_0, \overline{\kappa}, \underline{\rho_0}, \overline{\rho_0}, T, \underline{\theta_0}, \overline{\theta_0})>0$ such that
\begin{equation}\label{eq:155}
\sup_{0\leq t\leq T} \kappa\Vert\partial_x\rho\Vert_\infty^2(t) + \sup_{0\leq t\leq T}\kappa\Vert\partial_x\theta\Vert_\infty^2(t)\leq F_1.
\end{equation}
\end{prop}
\begin{proof}
Let $k\in\N^*$. From \eqref{eq:132}, we get, multiplying by $\vert\partial_x\rho\vert^{2k-2}\partial_x\rho$,
\begin{align}\label{eq:156}
\frac{1}{2k}\partial_t\vert\partial_x\rho\vert^{2k} + \frac{1}{2k}\partial_x(u\vert\partial_x\rho\vert^{2k}) = -(2-\frac{1}{2k})\vert\partial_x\rho\vert^{2k}\partial_x u -\rho\vert\partial_x\rho\vert^{2k-2}(\partial_x\rho)\partial_{xx} u.
\end{align}
As
\begin{equation*}
\mu\partial_{xx} u = \partial_x\sigma + R\theta\partial_x\rho + R\rho\partial_x\theta,
\end{equation*}
we get, from \eqref{eq:156},
\begin{equation*}
\begin{split}
&\frac{1}{2k}\partial_t\vert\partial_x\rho\vert^{2k} + \frac{1}{2k}\partial_x(u\vert\partial_x\rho\vert^{2k}) + \frac{1}{\mu}R\rho\theta\vert\partial_x\rho\vert^{2k} \\&= -(2-\frac{1}{2k})\vert\partial_x\rho\vert^{2k}\partial_x u - \frac{1}{\mu}\rho\vert\partial_x\rho\vert^{2k-2}(\partial_x\rho)\partial_x\sigma - \frac{1}{\mu}R\rho^2\vert\partial_x\rho\vert^{2k-2}(\partial_x\rho)\partial_x\theta.
\end{split}
\end{equation*}
Hence, integrating this last equation on the torus, we get
\begin{equation*}
\frac{1}{2k}\frac{d}{dt}\int_0^1\vert\partial_x\rho\vert^{2k}\leq 2\Vert\partial_x u\Vert_\infty \int_0^1\vert\partial_x\rho\vert^{2k} + \frac{1}{\mu}\overline{\rho}\Vert\partial_x\sigma\Vert_\infty \int_0^1 \vert\partial_x\rho\vert^{2k-1} +  \frac{R}{\mu}\overline{\rho}^2\Vert\partial_x\theta\Vert_\infty\int_0^1\vert\partial_x\rho\vert^{2k-1}.
\end{equation*}
As
\begin{equation*}
\int_0^1\vert\partial_x\rho\vert^{2k-1}\leq \left(\int_0^1\vert\partial_x\rho\vert^{2k}\right)^{(2k-1)/(2k)},
\end{equation*}
since $x\mapsto x^{(2k-1)/2k}$ is concave on $\R_+$, we get
\begin{equation*}
\frac{1}{2}\Vert\vert\partial_x\rho\vert^2\Vert_{k}^{k-1}\frac{d}{dt}\Vert\vert\partial_x\rho\vert^2\Vert_{k}\leq 2\Vert\partial_x u\Vert_\infty\Vert\vert\partial_x\rho\vert^2\Vert_{k}^{k} + \frac{1}{\mu}\overline{\rho}\Vert\partial_x\sigma\Vert_\infty\Vert\vert\partial_x\rho\vert^2\Vert_{k}^{k-1/2} + \frac{R}{\mu}\overline{\rho}^2\Vert\partial_x\theta\Vert_\infty\Vert\vert\partial_x\rho\vert^2\Vert_{k}^{k-1/2},
\end{equation*}
hence
\begin{equation*}
\frac{d}{dt}\Vert\vert\partial_x\rho\vert^2\Vert_{k}\leq 4\Vert\partial_x u\Vert_\infty\Vert\vert\partial_x\rho\vert^2\Vert_{k} + \frac{2}{\mu}\overline{\rho}\Vert\partial_x\sigma\Vert_\infty\Vert\vert\partial_x\rho\vert^2\Vert_k^{1/2} + \frac{2R}{\mu}\overline{\rho}^2\Vert\partial_x\theta\Vert_\infty\Vert\vert\partial_x\rho\vert^2\Vert_k^{1/2}.
\end{equation*}
Then passing to the limit $k\rightarrow +\infty$ we obtain
\begin{align}
\frac{d}{dt}\Vert\partial_x\rho\Vert_{\infty}^2&\leq 4\Vert\partial_x u\Vert_\infty\Vert\partial_x\rho\Vert_{\infty}^2 + \frac{2}{\mu}\overline{\rho}\Vert\partial_x\sigma\Vert_\infty\Vert\partial_x\rho\Vert_\infty + \frac{2R}{\mu}\overline{\rho}^2\Vert\partial_x\theta\Vert_\infty\Vert\partial_x \rho\Vert_\infty \nonumber\\&\leq \left(4\Vert\partial_x u\Vert_\infty + \frac{\overline{\rho}^2}{\mu^2}+ \frac{R^2\overline{\rho}^4}{\mu^2}\right)\Vert\partial_x\rho\Vert_\infty^2 +\frac{R^2\overline{\rho}^4}{\mu^2}\Vert\partial_x\theta\Vert_\infty^2 + \frac{\overline{\rho}^2}{\mu^2}\Vert\partial_x\sigma\Vert_\infty^2.\label{eq:157}
\end{align}
Besided, from \eqref{eq:137}, we get, multiplying by $\vert\partial_x\theta\vert^{2k-2}\partial_x\theta$ for $k\in\N^*$,
\begin{equation*}
\begin{split}
&\frac{1}{2k}\partial_t\vert\partial_x\theta\vert^{2k} + \frac{1}{2k}\partial_x(u\vert\partial_x\theta\vert^{2k}) \\&= -\frac{1}{\cv}\left(1-\frac{1}{2k}\right)\vert\partial_x\theta\vert^{2k}\partial_x u + \frac{1}{\cv}\partial_x\left(\frac{\sigma\partial_x u}{\rho}\right)\vert\partial_x\theta\vert^{2k-2}\partial_x\theta + \frac{1}{\cv}\partial_x\left(\frac{\partial_x(\kappa\partial_x\theta)}{\rho}\right)\vert\partial_x\theta\vert^{2k-2}\partial_x\theta.
\end{split}
\end{equation*}
Integrating on the torus, we get
\begin{equation*}
\frac{1}{2k}\frac{d}{dt}\int_0^1\vert\partial_x\theta\vert^{2k} + \frac{2k-1}{\cv}\kappa\int_0^1 \frac{\vert\partial_{xx}\theta\vert^2}{\rho}\vert\partial_x\theta\vert^{2k-2} \leq \frac{1}{\cv}\Vert\partial_x u\Vert_\infty\int_0^1 \vert\partial_x\theta\vert^{2k} + \frac{1}{\cv}\left\Vert \partial_x\left(\frac{\sigma\partial_x u}{\rho}\right)\right\Vert_\infty\int_0^1 \vert\partial_x\theta\vert^{2k-1},
\end{equation*}
hence, similarly as for \eqref{eq:157},
\begin{equation*}
\frac{d}{dt}\Vert\partial_x\theta\Vert_\infty^2 \leq \frac{2}{\cv}\Vert\partial_x u\Vert_\infty \Vert\partial_x\theta\Vert_\infty^2 + \frac{2}{\cv}\left\Vert \partial_x\left(\frac{\sigma\partial_x u}{\rho}\right)\right\Vert_\infty\Vert\partial_x\theta\Vert_\infty. 
\end{equation*}
Moreover,
\begin{align}
\mu\partial_x\left(\frac{\sigma\partial_x u}{\rho}\right) &= \partial_x\left(\frac{\sigma^2}{\rho}\right) + R\partial_x(\sigma\theta) \nonumber\\&=2\frac{\sigma\partial_x\sigma}{\rho} - \frac{\sigma^2\partial_x\rho}{\rho^2} + R\theta\partial_x\sigma + R\sigma\partial_x\theta, \nonumber\end{align}
hence
\begin{equation*}
\left\Vert \partial_x\left(\frac{\sigma\partial_x u}{\rho}\right)\right\Vert_\infty \leq \frac{2}{ \mu\underline{\rho}}\Vert\sigma\Vert_\infty\Vert\partial_x\sigma\Vert_\infty + \frac{1}{\mu\underline{\rho}^2}\Vert\sigma\Vert_\infty^2\Vert\partial_x\rho\Vert_\infty + \frac{R}{\mu}\overline{\theta}\Vert\partial_x\sigma\Vert_\infty + \frac{R}{\mu}\Vert\sigma\Vert_\infty\Vert\partial_x\theta\Vert_\infty,
\end{equation*}
and
\begin{align}
&\frac{d}{dt}\Vert\partial_x\theta\Vert_\infty^2\leq \frac{2}{\cv}\Vert\partial_x u\Vert_\infty\Vert\partial_x\theta\Vert_\infty^2 + \frac{4}{\mu\cv\underline{\rho}}\Vert\sigma\Vert_\infty\Vert\partial_x\sigma\Vert_\infty\Vert\partial_x\theta\Vert_\infty
+\frac{2}{\mu\cv\underline{\rho}^2}\Vert\sigma\Vert_\infty^2\Vert\partial_x\rho\Vert_\infty\Vert\partial_x\theta\Vert_\infty \nonumber\\&+ \frac{2R\overline{\theta}}{\mu\cv}\Vert\partial_x\sigma\Vert_\infty\Vert\partial_x\theta\Vert_\infty + \frac{2R}{\mu\cv}\Vert\sigma\Vert_\infty\Vert\partial_x\theta\Vert_\infty^2 \nonumber\\&\leq \left(\frac{2}{\cv}\Vert\partial_x u\Vert_\infty + \frac{2\Vert\sigma\Vert_\infty^2}{\mu\cv\underline{\rho}} + \frac{\Vert\sigma\Vert_\infty^2}{\mu\cv\underline{\rho}^2} + \frac{R\overline{\theta}}{\mu\cv} + \frac{2R}{\mu\cv}\Vert\sigma\Vert_\infty\right)\Vert\partial_x\theta\Vert_\infty^2 \nonumber\\&+ \frac{\Vert\sigma\Vert_\infty^2}{\mu\cv\underline{\rho}^2}\Vert\partial_x\rho\Vert_\infty^2+\left(\frac{2}{\mu\cv\underline{\rho}} + \frac{R\overline{\theta}}{\mu\cv}\right)\Vert\partial_x\sigma\Vert_\infty^2.\label{eq:158}
\end{align}
Adding \eqref{eq:157} and \eqref{eq:158}, then multiplying by $\kappa$, we get
\begin{equation}\label{eq:159}
\frac{d}{dt}\kappa\Vert\partial_x\rho\Vert_\infty^2 + \frac{d}{dt}\kappa\Vert\partial_x\theta\Vert_\infty^2 \leq G_3(\kappa\Vert\partial_x\rho\Vert_\infty^2 + \kappa\Vert\partial_x\theta\Vert_\infty^2) + G_4\kappa\Vert\partial_x\sigma\Vert_\infty^2.
\end{equation}
where
\begin{equation*}
\begin{split}
G_3 =& \frac{R^2\overline{\rho}^4}{\mu^2}+\frac{\Vert\sigma\Vert_\infty^2}{\mu\cv\underline{\rho}^2}\\&+ \max\left(4\Vert\partial_x u\Vert_\infty + \frac{\overline{\rho}^2}{\mu^2},\frac{2}{\cv}\Vert\partial_x u\Vert_\infty + \frac{2\Vert\sigma\Vert_\infty^2}{\mu\cv\underline{\rho}} + \frac{R\overline{\theta}}{\mu\cv}+ \frac{2R}{\mu\cv}\Vert\sigma\Vert_\infty\right),
\end{split}
\end{equation*}
\begin{equation*}
G_4 = \max\left(\frac{\overline{\rho}^2}{\mu^2},\frac{2}{\mu\cv\underline{\rho}} + \frac{R\overline{\theta}}{\mu\cv}\right).
\end{equation*}
Note that, according to assumption \eqref{eq:153},
\begin{equation*}
\kappa\int_0^1\vert\partial_x\theta_0\vert^2\leq F_0,
\end{equation*}
\begin{equation*}
\kappa\int_0^1\vert\partial_x\sigma_0\vert^2\leq 2(\mu^2 + R^2\overline{\theta}^2 + R^2\overline{\rho}^2)F_0.
\end{equation*}
Thus, by Proposition \ref{prop:pullnoel}, there exists some $E_1' = E_1'(C_0,\cv,F_0,\gamma,\overline{\kappa},\mu,\underline{\rho_0},\overline{\rho_0},T,\underline{\theta_0},\overline{\theta_0})$ such that
\begin{equation*}
\kappa\int_0^T\Vert\partial_x\sigma\Vert_\infty^2\leq E_1'.
\end{equation*}
We then deduce Proposition \ref{prop:pullnoel2} from \eqref{eq:159} and the Grönwall's lemma.
\end{proof}

\subsection{Stability result in Lagrangian coordinates}
In this section we will switch to switch to Lagrangian coordinates. For $u\in L^1(0,T,W^{1,\infty}(\T))$, we define the flow $X:[0,T]\times\T\rightarrow \R$ of $u$ by 
\begin{equation*}
\begin{cases}
\text{for all }(t,x)\in[0,T]\times\mathbb{T}, &
\partial_t X(t,x) = u(t, X(t,x)), \\[0.3em]
\text{for all }x\in\mathbb{T}, &
X(0,x) = x .
\end{cases}
\end{equation*}
We will moreover denote $X(t)(x) = X(t,x)$. The classical following result holds.
\begin{prop}
Let $(\rho,u,\theta)$ a "à la Hoff" solution of $(NS_\kappa)$ for $\kappa\geq 0$, with initial condition $(\rho_0,u_0,\theta_0)$ satisfying \eqref{eq:16}--\eqref{eq:18}. Let $X$ be the flow of $u$. Then 
\begin{equation}\label{eq:160}
\text{for all }(t,x)\in [0,T]\times\T,\quad \partial_x X(t,x) = \frac{\rho_0(x)}{\rho(t,X(t,x))}.
\end{equation}
In particular,
\begin{equation*}
\text{for all }(\tau,t)\in [0,T]^2,\quad \int_0^1 f(\tau,X(t,x))dx = \int_0^1 f(\tau, y)\frac{\rho(t,y)}{\rho_0(X(t)^{-1}(y))}dy,
\end{equation*}
hence
\begin{equation}\label{eq:161}
\text{for all }(\tau,t)\in [0,T]^2,\quad\frac{\underline{\rho}}{\overline{\rho}}\int_0^1 \vert f(\tau,x)\vert dx\leq\int_0^1 \vert f(\tau,X(t,x))\vert dx \leq \frac{\overline{\rho}}{\underline{\rho}}\int_0^1 \vert f(\tau,x)\vert dx.
\end{equation}
\end{prop}

We then prove the following stability result, for some fixed $\kappa$:
\begin{prop}\label{prop:poney}
Let $\overline{\kappa}>0$. Let $\kappa\in [0,\overline{\kappa}]$. For $i=1,2$, let $(\rho_{0,i},\theta_{0,i},u_{0,i})$ satisfy \eqref{eq:16}--\eqref{eq:18} and $(\rho_i,u_i,\theta_i)$ be the "à la Hoff" solution of $(NS_\kappa)$, with $0<\kappa\leq\overline{\kappa}$ and with initial condition $(\rho_{0,i},\theta_{0,i},u_{0,i})$. Assume that $(\rho_{0,2},\theta_{0,2},u_{0,2})$ satisfies \eqref{eq:147},\eqref{eq:149} and \eqref{eq:154}, with some $F_0>0$. Then there exists some \newline$L=L(C_0, F_0,\overline{\kappa},\overline{\rho_0},\underline{\rho_0},T,\underline{\theta_0},\overline{\theta_0})>0$ such that
\begin{equation*}
\begin{split}
\sup_{0\leq t\leq T}&\left(\int_0^1 \vert \rho_2(t,X_2(t,x)) - \rho_1(t,X_1(t,x))\vert^2dx+ \int_0^1\vert\theta_2(t,X_2(t,x))-\theta_1(t,X_1(t,x))\vert^2dx\right.\\&\left. + \int_0^1\vert u_2(t,X_2(t,x)) - u_1(t,X_1(t,x))\vert^2dx + \int_0^1 \vert X_2(t,x) - X_1(t,x)\vert^2dx\right)\\&\leq L(\Vert\rho_{0,2}-\rho_{0,1}\Vert_2^{2/3} + \Vert \theta_{0,2} - \theta_{0,1}\Vert_2^2 + \Vert u_{0,2} - u_{0,1}\Vert_2^2).
\end{split}
\end{equation*}
\end{prop}
\begin{proof}
For $i=1,2$ and for a map $f_i:[0,T]\times\T\rightarrow \R$, we denote
\begin{equation*}
\text{for all }(t,x)\in[0,T]\times\T,\quad \widetilde{f_i}(t,x) = f_i(t,X_i(t,x)).
\end{equation*}
Remark that $\tilde{f_i}(0,\cdot) = f_i(0,\cdot)$.
We denote, for any function $\widetilde{f}=f(\widetilde{\rho},\widetilde{\theta},\widetilde{u})$,
\begin{equation*}
\delta \widetilde{f} := \widetilde{f_2} - \widetilde{f_1}.
\end{equation*}
In Lagrangian coordinates, $(NS_\kappa)$ can be rewritten as, for $i=1,2$,

\begin{empheq}[left=\empheqlbrace]{alignat=1}
\rho_{0,i}\partial_t\widetilde{\rho_i} &= -\widetilde{\rho_i}^2\partial_x \widetilde{u_i},\label{eq:162}\\
\rho_{0,i}\partial_t \widetilde{u_i} &= \partial_x\widetilde{\sigma_i},\label{eq:163}\\
\rho_{0,i}\cv\partial_t\widetilde{\theta_i} &= \widetilde{\sigma_i}\partial_x \widetilde{u_i} + \partial_x(\kappa \frac{\widetilde{\rho_i}}{\rho_{0,i}}\partial_x\widetilde{\theta_i})\label{eq:164}
\end{empheq}
where
\begin{equation}
\widetilde{\sigma_i} = \mu\frac{\widetilde{\rho_i}}{\rho_{0,i}}\partial_x \widetilde{u_i} - R\widetilde{\rho_i}\widetilde{\theta_i}.\label{eq:165}
\end{equation}

From \eqref{eq:162} we obtain
\begin{equation*}
(\delta\rho_{0})\partial_t\widetilde{\rho_2} + \rho_{0,1}\partial_t \delta\widetilde{\rho} = - (\widetilde{\rho_1} + \widetilde{\rho_2})(\delta\widetilde{\rho})\partial_x \widetilde{u_1} - \widetilde{\rho_2}^2\partial_x \delta\widetilde{u}.
\end{equation*}
Let $\varepsilon>0$. We get, multiplying this last equation by $\delta\widetilde{\rho}$ and integrating on the torus, using Young's inequality,
\begin{align}
\frac{1}{2}\frac{d}{dt}\int_0^1\rho_{0,1} \vert\delta \widetilde{\rho}\vert^2&\leq \int_0^1 \vert \widetilde{\rho_1} + \widetilde{\rho_2}\vert \vert\partial_x \widetilde{u_1}\vert \vert\delta\widetilde{\rho}\vert^2 + \int_0^1 \vert\widetilde{\rho_2}\vert^2\vert\partial_x\delta\widetilde{u}\vert \vert\delta\widetilde{\rho}\vert + \int_0^1\vert\delta\rho_0\vert \frac{\vert\widetilde{\rho_2}\vert^2}{\vert\rho_{0,2}\vert}\vert\partial_x \widetilde{u_2}\vert \vert\delta\widetilde{\rho}\vert \nonumber\\&\leq \left(2\overline{\rho}\Vert\partial_x \widetilde{u_1}\Vert_\infty + \frac{\overline{\rho}^4}{2\underline{\rho}^4}\Vert\partial_x\widetilde{u_2}\Vert_\infty^2 + \frac{\overline{\rho}^4}{4\varepsilon}\right)\int_0^1\vert\delta\widetilde{\rho}\vert^2 + \varepsilon\int_0^1 \vert\partial_x\delta \widetilde{u}\vert^2 + \frac{1}{2}\int_0^1\vert\delta\rho_0\vert^2. \nonumber\end{align}
Moreover, from \eqref{eq:163} and \eqref{eq:165} we get
\begin{equation*}
(\delta \rho_0) \partial_t \widetilde{u_2} + \rho_{0,1}\partial_t\delta \widetilde{u} = \partial_x(\mu\delta(\widetilde{\rho}/\rho_0)\partial_x \widetilde{u_2}) + \partial_x(\mu\frac{\widetilde{\rho_1}}{\rho_{0,1}}\partial_x\delta \widetilde{u}) - \partial_x(\delta \widetilde{p}),
\end{equation*}
hence, multiplying by $\delta \widetilde{u}$ and integrating on the torus,
\begin{align}
\frac{1}{2}\frac{d}{dt}\int_0^1\rho_{0,1}\vert\delta \widetilde{u}\vert^2 + \int_0^1 \mu \frac{\widetilde{\rho_1}}{\rho_{0,1}}\vert\partial_x\delta \widetilde{u}\vert^2 &= -\int_0^1(\delta\rho_0)(\partial_t \widetilde{u_2})\delta \widetilde{u} - \int_0^1 \mu \delta(\widetilde{\rho}/\rho_0)(\partial_x \widetilde{u_2})\partial_x\delta \widetilde{u} + \int_0^1 (\delta \widetilde{p}) \partial_x\delta \widetilde{u} \nonumber\\&=:I_1 + I_2 + I_3. \nonumber\end{align}
By Young's inequality,
\begin{align}
I_2 &\leq \frac{\mu^2}{4\varepsilon}\Vert\partial_x \widetilde{u_2}\Vert_\infty^2\int_0^1 \vert\delta(\widetilde{\rho}/\rho_0)\vert^2 + \varepsilon \int_0^1 \vert \partial_x\delta \widetilde{u}\vert^2. \nonumber\end{align}
As 
\begin{equation}\label{eq:166}
\left\vert\delta\left(\frac{\widetilde{\rho}}{\rho_0}\right)\right\vert = \left\vert\frac{\delta \widetilde{\rho}}{\rho_{0,2}} - \frac{(\delta \rho_0)\widetilde{\rho_1}}{\rho_{0,1}\rho_{0,2}}\right\vert\leq  \frac{1}{\underline{\rho}}\vert\delta\widetilde{\rho}\vert + \frac{\overline{\rho}}{\underline{\rho}^2}\vert\delta\rho_0\vert ,
\end{equation}
we get 
\begin{equation}\label{eq:167}
I_2\leq \frac{\mu^2}{2\underline{\rho}^2\varepsilon} \Vert\partial_x\widetilde{u_2}\Vert_\infty^2\int_0^1\vert\delta\widetilde{\rho}\vert^2 +
\frac{\mu^2\overline{\rho}^2}{2\underline{\rho}^4\varepsilon} \Vert\partial_x\widetilde{u_2}\Vert_\infty^2\int_0^1\vert\delta\rho_0\vert^2 + \varepsilon\int_0^1\vert\partial_x\delta \widetilde{u}\vert^2.  
\end{equation}
Similarly, as
\begin{equation*}
\vert\delta \widetilde{p}\vert = \vert R\widetilde{\theta_1}\delta\widetilde{\rho} + R\widetilde{\rho_2}\delta\widetilde{\theta}\vert\leq R\overline{\theta}\vert\delta\widetilde{\rho}\vert + R\overline{\rho}\vert\delta\widetilde{\theta}\vert,
\end{equation*}
we get
\begin{equation}\label{eq:168}
I_3\leq \frac{R^2\overline{\theta}^2}{2\varepsilon}\int_0^1\vert\delta\widetilde{\rho}\vert^2 + \frac{R^2\overline{\rho}^2}{2\varepsilon}\int_0^1\vert\delta\widetilde{\theta}\vert^2 + \varepsilon\int_0^1 \vert\partial_x\delta\widetilde{u}\vert^2.
\end{equation}
Finally, using Hölder's inequality and Young's inequality,
\begin{align}
I_1&\leq \Vert\delta \rho_0\Vert_2\Vert\partial_t \widetilde{u_2}\Vert_2\Vert\delta \widetilde{u}\Vert_\infty \nonumber\\&\leq \frac{1}{2}\Vert\partial_t \widetilde{u_2}\Vert_2^2\Vert\delta\rho_0\Vert_2^2 + \frac{1}{2} \Vert\delta \widetilde{u}\Vert_\infty^2.\label{eq:169}
\end{align}
Moreover, by Galgliardo-Nirenberg's inequality (recall \eqref{eq:53}),
\begin{align}
\Vert\delta \widetilde{u}\Vert_\infty^2 &\leq \Vert\delta \widetilde{u}\Vert_2^2 + 2\Vert\delta \widetilde{u}\Vert_2\Vert\partial_x\delta \widetilde{u}\Vert_2 \nonumber\\&\leq \frac{2\varepsilon + 1}{2\varepsilon}\int_0^1 \vert\delta \widetilde{u}\vert^2 + 2\varepsilon\int_0^1 \vert \partial_x\delta\widetilde{u}\vert^2.\label{eq:170}
\end{align}
Hence \eqref{eq:167}, \eqref{eq:168}, \eqref{eq:169}, \eqref{eq:170} give
\begin{equation*}
\begin{split}
\frac{1}{2}\frac{d}{dt}\int_0^1\rho_{0,1}\vert\delta\widetilde{u}\vert^2 &+ \int_0^1\mu \frac{\widetilde{\rho_1}}{\rho_{0,1}}\vert\partial_x\delta\widetilde{u}\vert^2 \leq \left(\frac{\mu^2}{2\underline{\rho}^2\varepsilon}\Vert\partial_x\widetilde{u_2}\Vert_\infty^2 + \frac{R^2\overline{\theta}^2}{2\varepsilon}\right)\int_0^1\vert\delta\widetilde{\rho}\vert^2 + \frac{R^2\overline{\rho}^2}{2\varepsilon}\int_0^1\vert\delta\widetilde{\theta}\vert^2 \\&+ \left(\frac{\mu^2\overline{\rho}^2}{2\underline{\rho}^4\varepsilon}\Vert\partial_x \widetilde{u_2}\Vert_\infty^2 + \frac{1}{2}\Vert\partial_t \widetilde{u_2}\Vert_2^2\right)\int_0^1\vert\delta\rho_0\vert^2 + \frac{2\varepsilon + 1}{4\varepsilon}\int_0^1\vert\delta\widetilde{u}\vert^2 + 3\varepsilon\int_0^1\vert\partial_x\delta \widetilde{u}\vert^2.
\end{split}
\end{equation*}
From \eqref{eq:164} we get
\begin{align*}
\delta(\rho_0\cv\partial_t \widetilde{\theta})&=(\delta \rho_0)\cv\partial_t\widetilde{\theta_2} + \rho_{0,1}\cv \partial_t\delta\widetilde{\theta} \\&= (\delta\widetilde{\sigma})\partial_x \widetilde{u_1} + \widetilde{\sigma_2}\partial_x\delta \widetilde{u} + \partial_x(\kappa\delta\left(\frac{\widetilde{\rho}}{\rho_0}\right)\partial_x\widetilde{\theta_2}) + \partial_x(\kappa\frac{\widetilde{\rho_1}}{\rho_{0,1}}\partial_x\delta\widetilde{\theta}).
\end{align*}
Multiplying by $\delta\widetilde{\theta}$ and integrating in space, we get
\begin{align}
\frac{1}{2}\frac{d}{dt}\int_0^1 \rho_{0,1} \cv \vert\delta\widetilde{\theta}\vert^2 + \kappa\int_0^1\frac{\widetilde{\rho_1}}{\rho_{0,1}}\vert\partial_x \delta\widetilde{\theta}\vert^2 &= -\cv\int_0^1(\delta\rho_0)(\partial_t\widetilde{\theta_2})\delta\widetilde{\theta} - \kappa\int_0^1\delta\left(\frac{\widetilde{\rho}}{\rho_0}\right)(\partial_x\widetilde{\theta_2}) \partial_x \delta\widetilde{\theta} \nonumber\\&+\int_0^1(\delta\widetilde{\sigma})(\partial_x \widetilde{u_1})\delta\widetilde{\theta} + \int_0^1 \widetilde{\sigma_2}(\partial_x\delta \widetilde{u})\delta\widetilde{\theta} \nonumber\\&=: J_1 + J_2 + J_3 + J_4. \nonumber\end{align}
We then obtain by Hölder's inequality and Young's inequality,
\begin{align}
J_1 &\leq \cv\Vert\delta\rho_0\Vert_4\Vert\partial_t\widetilde{\theta_2}\Vert_2\Vert\delta\widetilde{\theta}\Vert_4 \nonumber\\&\leq \frac{3\cv}{4}\Vert\partial_t\widetilde{\theta_2}\Vert_2\Vert \delta\rho_0\Vert_4^{4/3} + \frac{\cv}{4}\Vert \partial_t \widetilde{\theta_2}\Vert_2\int_0^1\vert\delta\widetilde{\theta}\vert^4 \nonumber\\&\leq \frac{3\cv}{4}\Vert\partial_t\widetilde{\theta_2}\Vert_2 \Vert\delta\rho_0\Vert_4^{4/3} + \frac{\cv\overline{\theta}^2}{4}\Vert\partial_t\widetilde{\theta_2}\Vert_2\int_0^1\vert\delta\widetilde{\theta}\vert^2.\label{eq:171}
\end{align}
Moreover, from Young's inequality and \eqref{eq:166},
\begin{align}
J_2\leq \frac{1}{2\varepsilon\underline{\rho}^2}\kappa\Vert\partial_x\widetilde{\theta_2}\Vert_\infty^2\int_0^1\vert\delta\widetilde{\rho}\vert^2 + \frac{\overline{\rho}^2}{2\varepsilon\underline{\rho}^4}\kappa\Vert\partial_x\widetilde{\theta_2}\Vert_\infty^2\int_0^1\vert\delta\rho_0\vert^2 +\varepsilon\kappa\int_0^1\vert\partial_x\delta\widetilde{\theta}\vert^2.\label{eq:172}
\end{align}
From \eqref{eq:166},
\begin{align}
\vert\delta\widetilde{\sigma}\vert &= \left\vert\mu (\delta(\widetilde{\rho}/\rho_0))\partial_x \widetilde{u_2} + \mu\frac{\widetilde{\rho_1}}{\rho_{0,1}}\partial_x\delta \widetilde{u} -R\widetilde{\theta_1}\delta\widetilde{\rho} - R\widetilde{\rho_2}\delta\widetilde{\theta}\right\vert \nonumber\\&\leq \frac{\mu}{\underline{\rho}}\Vert\partial_x \widetilde{u_2}\Vert_\infty \vert\delta\widetilde{\rho}\vert + \frac{\mu\overline{\rho}}{\underline{\rho}^2}\Vert\partial_x \widetilde{u_2}\Vert_\infty\vert\delta\rho_0\vert + \frac{\mu\overline{\rho}}{\underline{\rho}}\vert\partial_x\delta\widetilde{u}\vert + R\overline{\theta}\vert\delta\widetilde{\rho}\vert + R\overline{\rho}\vert\delta\widetilde{\theta}\vert. \nonumber\end{align}
Hence from Young's inequality,
\begin{align}
J_3 &\leq \frac{5\mu^2\overline{\rho}^2}{4\underline{\rho}^2\varepsilon}\Vert\partial_x \widetilde{u_1}\Vert_\infty^2\int_0^1\vert\delta\widetilde{\theta}\vert^2 + \frac{\underline{\rho}^2\varepsilon}{5\mu^2\overline{\rho}^2}\int_0^1 \vert\delta\widetilde{\sigma}\vert^2 \nonumber\\&\leq \left(\frac{5\mu^2\overline{\rho}^2}{4\underline{\rho}^2\varepsilon}\Vert\partial_x \widetilde{u_1}\Vert_\infty^2 + \frac{R^2\underline{\rho}^2\varepsilon}{\mu^2}\right)\int_0^1\vert\delta\widetilde{\theta}\vert^2 + \left(\frac{\varepsilon}{\overline{\rho}^2}\Vert\partial_x \widetilde{u_2}\Vert_\infty^2 + \frac{R^2\overline{\theta}^2\underline{\rho}^2\varepsilon}{\mu^2\overline{\rho}^2}\right)\int_0^1\vert\delta\widetilde{\rho}\vert^2 \nonumber\\&+ \frac{\varepsilon}{\underline{\rho}^2}\Vert\partial_x \widetilde{u_2}\Vert_\infty^2\int_0^1\vert\delta\rho_0\vert^2 +\varepsilon\int_0^1\vert\partial_x\delta\widetilde{u}\vert^2.\label{eq:173}
\end{align}
By Young's inequality,
\begin{align}
J_4 &\leq \varepsilon\int_0^1 \vert\partial_x\delta \widetilde{u}\vert^2 + \frac{1}{4\varepsilon}\Vert\widetilde{\sigma_2}\Vert_\infty^2\int_0^1\vert\delta\widetilde{\theta}\vert^2.\label{eq:174}
\end{align}
Combining \eqref{eq:171},\eqref{eq:172},\eqref{eq:173},\eqref{eq:174}, we get
\begin{equation*}
\begin{split}
\frac{1}{2}\frac{d}{dt}\int_0^1\rho_{0,1}\cv\vert\delta\widetilde{\theta}\vert^2 &+ \kappa\int_0^1\frac{\widetilde{\rho_1}}{\rho_{0,1}}\vert\partial_x\delta\widetilde{\theta}\vert^2\leq \left(\frac{1}{2\varepsilon\underline{\rho}^2}\kappa\Vert\partial_x\widetilde{\theta}_2\Vert_\infty^2 + \frac{\varepsilon}{\overline{\rho}^2}\Vert\partial_x \widetilde{u_2}\Vert_\infty^2 + \frac{R^2\overline{\theta}^2\underline{\rho}^2\varepsilon}{\mu^2\overline{\rho}^2}\right)\int_0^1 \vert\delta\widetilde{\rho}\vert^2
\\& + \left(\frac{\cv\overline{\theta}^2}{4}\Vert\partial_t \widetilde{\theta_2}\Vert_2 + \frac{5\mu^2\overline{\rho}^2}{4\underline{\rho}^2\varepsilon}\Vert\partial_x \widetilde{u_1}\Vert_\infty^2 + \frac{R^2\underline{\rho}^2\varepsilon}{\mu^2} + \frac{1}{4\varepsilon}\Vert\widetilde{\sigma_2}\Vert_\infty^2\right)\int_0^1\vert\delta\widetilde{\theta}\vert^2
\\&+ \left(\frac{\overline{\rho}^2}{2\varepsilon\underline{\rho}^4}\kappa\Vert\partial_x\widetilde{\theta_2}\Vert_\infty^2 + \frac{\varepsilon}{\underline{\rho}^2}\Vert\partial_x\widetilde{u_2}\Vert_\infty^2\right)\int_0^1\vert\delta\rho_0\vert^2 + \frac{3\cv\overline{\rho}^{2/3}}{4}\Vert\partial_t\widetilde{\theta_2}\Vert_2\Vert\delta\rho_0\Vert_2^{2/3} \\&+ 2\varepsilon\int_0^1\vert\partial_x\delta \widetilde{u}\vert^2 + \varepsilon\kappa\int_0^1\vert\partial_x\delta\widetilde{\theta}\vert^2.
\end{split}
\end{equation*}
Finally, choosing
\begin{equation*}
\varepsilon\leq \min\left(\frac{\underline{\rho}}{2\overline{\rho}},\frac{\mu\overline{\rho}}{12\underline{\rho}}\right),
\end{equation*}
we get
\begin{equation*}
\begin{split}
&\frac{d}{dt}\int_0^1\rho_{0,1}\vert\delta\widetilde{\rho}\vert^2 + \frac{d}{dt}\int_0^1\rho_{0,1}\vert \delta \widetilde{u} \vert^2 + \int_0^1\mu \frac{\widetilde{\rho_1}}{\rho_{0,1}}\vert\partial_x\delta\widetilde{u}\vert^2 + \frac{d}{dt}\int_0^1\rho_{0,1}\cv\vert\delta\widetilde{\theta}\vert^2 + \kappa\int_0^1\frac{\widetilde{\rho_1}}{\rho_{0,1}}\vert\partial_x\delta\widetilde{\theta}\vert^2 \\&\leq G_5\left(\int_0^1\rho_{0,1}\vert\delta\widetilde{\rho}\vert^2 + \rho_{0,1}\vert\delta \widetilde{u}\vert^2 + \rho_{0,1}\cv\vert\delta\widetilde{\theta}\vert^2\right) +  G_6\left(\Vert \delta\rho_0\Vert_2^2 + \Vert\delta\rho_0\Vert_2^{2/3}\right) 
\end{split},
\end{equation*}
where
\begin{equation*}
\begin{split}
G_5 &= \frac{2}{\underline{\rho}}\max(1/\cv,1)\max\left(2\overline{\rho}\Vert\partial_x \widetilde{u_1}\Vert_\infty + \frac{\overline{\rho}^4}{2\underline{\rho}^4}\Vert\partial_x \widetilde{u_2}\Vert_\infty^2+ \frac{2\overline{\rho}^4}{\varepsilon},\frac{\mu^2}{2\underline{\rho}^2\varepsilon}\Vert\partial_x \widetilde{u_2}\Vert_\infty^2 + \frac{R^2\overline{\theta}^2}{4\varepsilon}, \frac{R^2\overline{\rho}^2}{4\varepsilon}, \frac{2\varepsilon+1}{4\varepsilon},\right. \\&\left.\frac{1}{2\varepsilon\underline{\rho}^2}\kappa\Vert\partial_x\widetilde{\theta}_2\Vert_\infty^2 + \frac{\varepsilon}{\overline{\rho}^2}\Vert\partial_x \widetilde{u_2}\Vert_\infty^2 + \frac{R^2\overline{\theta}^2\underline{\rho}^2\varepsilon}{\mu^2\overline{\rho}^2}, \frac{\cv\overline{\theta}^2}{4}\Vert\partial_t \widetilde{\theta_2}\Vert_2 + \frac{5\mu^2\overline{\rho}^2}{4\underline{\rho}^2\varepsilon}\Vert\partial_x \widetilde{u_1}\Vert_\infty^2 + \frac{R^2\underline{\rho}^2\varepsilon}{\mu^2} + \frac{\Vert\widetilde{\sigma_2}\Vert_\infty^2}{4}\right),
\end{split}
\end{equation*}
\begin{equation*}
G_6 = 2\max\left(\frac{\mu^2\overline{\rho}^2}{2\underline{\rho}^4\varepsilon}\Vert\partial_x \widetilde{u_2}\Vert_\infty^2 + \frac{1}{2}\Vert\partial_t \widetilde{u_2}\Vert_2^2, \frac{1}{2},\frac{\overline{\rho}^2}{2\varepsilon\underline{\rho}^4}\kappa\Vert\partial_x\widetilde{\theta_2}\Vert_\infty^2 + \frac{\varepsilon}{\underline{\rho}^2}\Vert\partial_x \widetilde{u_2}\Vert_\infty^2,\frac{3\cv\overline{\rho}^{2/3}}{4}\Vert\partial_t\widetilde{\theta_2}\Vert_2\right).
\end{equation*}
Hence by Grönwall's lemma, for all $t\in[0,T]$,
\begin{equation}
\begin{split}
&\int_0^1(\vert\delta\widetilde{\rho}\vert^2 + \vert\delta\widetilde{u}\vert^2 + \vert\delta\widetilde{\theta}\vert^2)(t) \\&\leq  \frac{\overline{\rho}}{\underline{\rho}}\max(\cv,1/\cv)\exp\left(\int_0^T G_5(s)ds\right)\left(\int_0^1(\vert\delta\rho_0\vert^2 + \vert\delta\theta_0\vert^2 + \vert\delta u_0\vert^2)\right)
\\& + \frac{\overline{\rho}}{\underline{\rho}}\max(\cv,1/\cv)\exp\left(\int_0^T(G_5(s) + G_6(s))ds\right) \left(\int_0^1\vert\delta\rho_0\vert^2 + \left(\int_0^1\vert\delta\rho_0\vert^2\right)^{1/3}\right).
\end{split}\label{eq:175}
\end{equation}
Remark that
\begin{equation}\label{eq:176}
\int_0^1\vert\delta\rho_0\vert^2 \leq \overline{\rho}^{4/3}\left(\int_0^1\vert\delta\rho_0\vert^2\right)^{1/3}.
\end{equation}
Moreover, starting from, for $i=1,2$,
\begin{equation*}
\text{for all }(t,x)\in [0,T]\times\T,\quad X_i(t,x) = x + \int_0^t u_i(s,X_i(s,x)) ds
\end{equation*}
we get
\begin{equation*}
\text{for all }(t,x)\in [0,T]\times\T,\quad X_2(t,x)-X_1(t,x) = \int_0^t \delta \widetilde{u}(s,x) ds
\end{equation*}
hence
\begin{equation}\label{eq:177}
\text{for all }t\in [0,T],\quad \int_0^1\vert X_2(t,x)-X_1(t,x)\vert^2 dx\leq T^2\sup_{0\leq t\leq T}\int_0^1\vert\delta\widetilde{u}\vert^2(t)dt.
\end{equation}

Note that \eqref{eq:147} holds with $E_0 = F_0$. Notably by using \eqref{eq:160} and \eqref{eq:161}, we deduce from \eqref{eq:14}, \eqref{eq:148}, \eqref{eq:155} that there exists some $\widetilde{C_1}$, $\widetilde{C_2}$ depending only on $C_0,F_0,\overline{\kappa},\overline{\rho_0},\underline{\rho_0},T,\underline{\theta_0},\overline{\theta_0}$ such that
\begin{equation}\label{eq:178}
\text{for }i=1,2,\quad 
\int_0^T\Vert\partial_x\widetilde{u_i}\Vert_\infty^2 + \int_0^T\Vert\widetilde{\sigma_i}\Vert_\infty^2 + \int_0^T\int_0^1\vert\partial_t \widetilde{u_i}\vert^2\leq \widetilde{C_1},
\end{equation}
\begin{equation}\label{eq:179}
\sup_{0\leq t\leq T}\kappa\Vert\partial_x\widetilde{\theta_2}\Vert_\infty^2 + \int_0^T\int_0^1\vert\partial_t \widetilde{\theta_2}\vert^2\leq \widetilde{C_2}. 
\end{equation}
Thus, combining \eqref{eq:175}, \eqref{eq:176}, \eqref{eq:177}, \eqref{eq:178} and \eqref{eq:179} gives Proposition \ref{prop:poney}.

\end{proof}
\subsection{Proof of the main result}
\begin{proof}[Proof of Theorem \ref{thm:main}]
Consider some $\varphi \in {\cal D}(\R)$ such that
\begin{equation*}
\int_{-\infty}^{\infty}\varphi = 1,\quad \varphi_{[-1,1]}=1, \quad 0\leq\varphi\leq 1.
\end{equation*}
Then we define, for $\eta>0$, $\varphi^\eta:\R\mapsto \R$ by
\begin{equation*}
\text{for all }x\in \R,\quad \varphi^\eta(x) = \frac{1}{\eta} \varphi(x/\eta).
\end{equation*}
Let $0 < \kappa\leq \overline{\kappa}$. Define 
\begin{equation*}
\widehat{\rho_0}^\kappa = \rho_0 \star \varphi^{\kappa^{1/4}},\quad \widehat{\theta_0}^\kappa = \theta_0 \star \varphi^{\kappa^{1/4}},\quad \widehat{u_0}^\kappa = u_0 \star \varphi^{\kappa^{1/2}},
\end{equation*}
and consider $(\widehat{\rho}^\kappa, \widehat{\theta}^\kappa,\widehat{u}^\kappa)$ the solution of $(NS_\kappa)$ with initial condition $(\widehat{\rho_0}^\kappa,\widehat{\theta_0}^\kappa, \widehat{u_0}^\kappa)$. 
We have
\begin{equation*}
\Vert\widehat{\rho_0}^\kappa - \rho_0\Vert_2 + \Vert\widehat{\theta_0}^\kappa - \theta_0\Vert_2 + \Vert\widehat{u_0}^\kappa - u_0\Vert_2 \tend{\kappa}{0} 0. 
\end{equation*}
Moreover,
\begin{equation*}
\sqrt{\kappa}\Vert\partial_x \widehat{\rho_0}^\kappa\Vert_\infty \leq\kappa^{1/4}\overline{\rho_0}\Vert\varphi'\Vert_1, \quad \sqrt{\kappa}\Vert\partial_x \widehat{\theta_0}^\kappa\Vert_\infty \leq \kappa^{1/4}\overline{\theta_0}\Vert \varphi'\Vert_1, \quad \sqrt{\kappa}\Vert \partial_{xx} \widehat{u_0}^\kappa\Vert_2\leq \Vert \partial_x u_0\Vert_2\Vert \varphi'\Vert_1,
\end{equation*}
and
\begin{equation*}
\kappa^{1/4}\Vert\partial_x\widehat{\rho_0}^\kappa\Vert_2\leq \Vert \rho_0\Vert_2\Vert\varphi'\Vert_1,\quad \kappa^{1/4}\Vert\partial_x\widehat{\theta_0}^\kappa\Vert_2\leq \Vert\theta_0\Vert_2\Vert\varphi'\Vert_1,
\end{equation*}
thus \eqref{eq:130} and \eqref{eq:154} are satisfied with $D_0, F_0$ that do not depend on $\kappa$. In the following, we denote $\widehat{X}^\kappa$ the flow of $\widehat{u}^\kappa$, 
For $f=\rho,\theta,u,$ we get, using \eqref{eq:161}, and abbreviating $L^2(0,T,L^2(\T))$ in $L^2$,
\begin{align}
\Vert f^\kappa - f\Vert_{L^2}&\leq \Vert f^\kappa - \widehat{f}^\kappa\Vert_{L^2} + \Vert \widehat{f}^\kappa - f\Vert_{L^2} \nonumber\\&\leq \sqrt{\frac{\overline{\rho}}{\underline{\rho}}}\Vert f^\kappa(\cdot,X^\kappa) - \widehat{f}^\kappa(\cdot,X^\kappa)\Vert_{L^2} + \Vert \widehat{f}^\kappa - f\Vert_{L^2} \nonumber\\&\leq \sqrt{\frac{\overline{\rho}}{\underline{\rho}}}\Vert f^\kappa(\cdot,X^\kappa) - \widehat{f}^\kappa(\cdot,{\widehat{X}^\kappa})\Vert_{L^2} + \sqrt{\frac{\overline{\rho}}{\underline{\rho}}}\Vert \widehat{f}^\kappa(\cdot,\widehat{X}^\kappa)-\widehat{f}^\kappa(\cdot,X^\kappa)\Vert_{L^2} +  \Vert \widehat{f}^\kappa - f\Vert_{L^2}. \nonumber\end{align}
At this step, remark that by Proposition \ref{prop:poney}, 
\begin{equation*}
\Vert f^\kappa(\cdot,X^\kappa) - \widehat{f}^\kappa(\cdot,\widehat{X}^\kappa)\Vert_{L^2} \tend{\kappa}{0} 0,
\end{equation*}
and by Proposition \ref{prop:4prime},
\begin{equation*}
\Vert \widehat{f}^\kappa - f\Vert_{L^2} \tend{\kappa}{0} 0.
\end{equation*}
Moreover,
\begin{align}
\Vert \widehat{f}^\kappa(\cdot,\widehat{X}^\kappa)-\widehat{f}^\kappa(\cdot,X^\kappa)\Vert_{L^2} &\leq \Vert \widehat{f}^\kappa(\cdot,\widehat{X}^\kappa)-f(\cdot,\widehat{X}^\kappa)\Vert_{L^2} + \Vert f(\cdot,\widehat{X}^\kappa) - f(\cdot,X^\kappa)\Vert_{L^2} + \Vert f(\cdot,X^\kappa) - \widehat{f}^\kappa(\cdot,X^\kappa)\Vert_{L^2} \nonumber\\&\leq 2\sqrt{\frac{\overline{\rho}}{\underline{\rho}}}\Vert \widehat{f}^\kappa - f\Vert_{L^2} + \Vert f(\cdot,\widehat{X}^\kappa) - f(\cdot,X^\kappa)\Vert_{L^2}. \nonumber\end{align}
As $\widehat{f}^\kappa \tend{\kappa}{0} f$ in $L^2(0,T,L^2(\T))$ by Proposition \ref{prop:4prime}, it remains to be proven that 
\begin{equation}\label{eq:180}
\Vert f(\cdot,\widehat{X}^\kappa) - f(\cdot,X^\kappa)\Vert_{L^2} \tend{\kappa}{0}0.
\end{equation}
Remark that the case $f=u$ is easy due to the regularity of $u$. Indeed, we have from \eqref{eq:177}
\begin{align*}
\int_0^T\int_0^1 \vert u(t,\widehat{X}^\kappa(t,x)) - u(t,X^\kappa(t,x))\vert^2 \leq &\left(\int_0^T \Vert\partial_x u\Vert_\infty^2(t) dt\right)\sup_{0\leq t\leq T} \int_0^1 \vert\widehat{X}^\kappa(t,x) - X^\kappa(t,x)\vert^2
\tend{\kappa}{0}0.
\end{align*}
For the case $f=\rho,\theta$, let us define $f_\eta$ for all $\eta>0$ by
\begin{equation*}
\text{for all }(t,x)\in [0,T]\times\T,\quad f_\eta(t,x) = (f(t,\cdot)\star\varphi_\eta)(x).
\end{equation*}
Then,
\begin{equation*}
\int_0^T\Vert\partial_x f_\eta\Vert_\infty^2\leq\frac{\Vert \varphi'\Vert_1^2}{\eta^2} \int_0^T\Vert f\Vert_\infty^2,
\end{equation*}
\begin{equation*}
\int_0^T\Vert f_\eta\Vert_2^2\leq \Vert\varphi\Vert_1^2\int_0^T\Vert f\Vert_2^2,
\end{equation*}
and thanks to the dominated convergence theorem,
\begin{equation*}
f_\eta \tend{\eta}{0} f \quad\text{in }L^2(0,T,L^2(\T)).
\end{equation*}
Then by \eqref{eq:161},
\begin{align}
\Vert f(\cdot,\widehat{X}^\kappa) - f(\cdot,X^\kappa)\Vert_{L^2} &\leq \Vert f(\cdot,\widehat{X}^\kappa) - f_\eta(\cdot,\widehat{X}^\kappa)\Vert_{L^2} + \Vert f_\eta(\cdot,\widehat{X}^\kappa) - f_\eta(\cdot,X^\kappa)\Vert_{L^2} + \Vert f_\eta(\cdot,X^\kappa) - f(\cdot,X^\kappa)\Vert_{L^2} \nonumber\\&\leq 2\sqrt{\frac{\overline{\rho}}{\underline{\rho}}}\Vert f - f_\eta\Vert_{L^2} + \left(\int_0^T\Vert\partial_x f_\eta\Vert_\infty^2(t)\int_0^1 \vert \widehat{X}^\kappa(t,x) - X^\kappa(t,x)\vert^2 dx dt\right)^{1/2} \nonumber\\&\leq 2\sqrt{\frac{\overline{\rho}}{\underline{\rho}}}\Vert f - f_\eta\Vert_{L^2} + \frac{\Vert\varphi'\Vert_1}{\eta}\sqrt{T}\Vert f\Vert_{L^\infty} \sup_{0\leq t\leq T}\Vert \widehat{X}^\kappa(t) - X^\kappa(t)\Vert_2, \nonumber\end{align}
hence
\begin{equation*}
\limsup_{\kappa\rightarrow 0} \Vert f(\cdot,\widehat{X}^\kappa) - f(\cdot,X^\kappa)\Vert_{L^2} \leq 2\sqrt{\frac{\overline{\rho}}{\underline{\rho}}}\Vert f - f_\eta\Vert_{L^2}.
\end{equation*}
We finally obtain \eqref{eq:180} taking $\eta\rightarrow 0$. In particular, we get
\begin{equation*}
\rho^\kappa \theta^\kappa \tend{\kappa}{0} \rho \theta \quad \text{in }L^1(0,T,L^1(\T)),
\end{equation*}
thus Theorem \ref{thm:main} by remembering Remark \ref{rmk:1}.
\end{proof}

\section*{Acknowledgements}  
The author would like to thank Didier Bresch and Frédéric Lagoutière for their attention and comments.
\printbibliography
\end{document}